\documentclass{article}
\usepackage{graphicx}
\usepackage{amsfonts}
\usepackage{amsmath}
\usepackage{amssymb}
\usepackage{fancyhdr}
\usepackage{indentfirst}
\usepackage{tikz}
\usepackage{tikz-cd}
\usepackage{capt-of}
\usepackage{booktabs}
\usepackage{verbatim}
\usepackage{color}
\usepackage{amsthm}
\usepackage{subfigure}
\usepackage{abstract}
\usepackage{multirow}
\usepackage{hyperref}
\usepackage[margin=1in]{geometry}
\usepackage{lipsum}
\usepackage{mathtools}
\usepackage{titletoc}
\usepackage{newfloat}
\numberwithin{equation}{section}
\newtheorem{theorem}{Theorem}[section]
\newtheorem{lemma}[theorem]{Lemma}
\newtheorem{example}[theorem]{Example}
\newtheorem{corollary}[theorem]{Corollary}
\newtheorem{proposition}[theorem]{Proposition}
\newtheorem{definition}[theorem]{Definition}
\newtheorem{remark}[theorem]{Remark}

\newcommand{\RNum}[1]{\uppercase\expandafter{\romannumeral #1\relax}}

\DeclareMathOperator{\disc}{disc}
\DeclareMathOperator{\Gal}{Gal}

\DeclareMathOperator{\Hom}{Hom}
\DeclareMathOperator{\Sur}{Sur}
\DeclareMathOperator{\Stab}{Stab}

\DeclareMathOperator{\Disc}{Disc}
\DeclareMathOperator{\Nm}{Nm}
\DeclareMathOperator{\Aut}{Aut}
\DeclareMathOperator{\Ker}{Ker}

\DeclareMathOperator{\Cl}{Cl}
\DeclareMathOperator{\Rg}{Rg}
\DeclareMathOperator{\rk}{rk}
\DeclareMathOperator{\Res}{Res}
\DeclareMathOperator{\Vol}{Vol}

\DeclareMathOperator{\ACl}{\widetilde{\mathrm{Cl}}}
\DeclareMathOperator{\Div}{\mathrm{Div}}

\DeclareMathOperator{\fp}{\mathfrak{p}}

\newcommand{\rd}{\,\mathrm{d}}

\DeclareFloatingEnvironment[fileext=dia,within=section]{diagram}
\renewcommand{\pmod}[1]{\left(\text{mod}\, #1\right)}
\usepackage{melanie}
\usepackage{pdfsync}

\newcommand{\mw}[1]{}

\begin{document}
	\title{The average size of $3$-torsion in class groups of $2$-extensions}
	\author{Robert J. Lemke Oliver, Jiuya Wang, Melanie Matchett Wood}
	\newcommand{\Addresses}{{
			\bigskip
			\footnotesize	
			Robert J.Lemke Oliver, \textsc{Department of Mathematics, Tufts University, 503 Boston Ave, Medford, MA 02155, USA
			}\par\nopagebreak
			\textit{E-mail address}: \texttt{Robert.Lemke\_Oliver@tufts.edu}\\	
			
			\bigskip
			\footnotesize		
			Jiuya Wang, \textsc{Department of Mathematics, Duke University, 120 Science Drive 117 Physics Building Durham, NC 27708, USA
			}\par\nopagebreak
			\textit{E-mail address}: \texttt{wangjiuy@math.duke.edu}\\
			
			\bigskip
			\footnotesize		
			Melanie Matchett Wood, \textsc{Department of Mathematics, Department of Mathematics,
Harvard University,
1 Oxford Street,
Cambridge, MA 02138, USA
			}\par\nopagebreak
			\textit{E-mail address}: \texttt{mmwood@math.harvard.edu}\\		
		}}
		\maketitle	
		\begin{abstract}	
		We determine the average size of the 3-torsion in class groups of $G$-extensions of a number field when $G$ is any transitive $2$-group containing a transposition, for example $D_4$.  
It follows from the Cohen--Lenstra--Martinet heuristics that the 
average size of the $p$-torsion in class groups of $G$-extensions of a number field is
conjecturally finite for any $G$ and most $p$ (including $p\nmid|G|$). Previously this conjecture had only been proven in the cases of  $G=S_2$ with $p=3$ 
and $G=S_3$ with $p=2$.  
We also show that the average  $3$-torsion in a certain relative class group for these $G$-extensions is as predicted by Cohen and Martinet, proving new cases of the Cohen--Lenstra--Martinet heuristics.  Our new method also works for many other permutation groups $G$ that are not $2$-groups.

		\end{abstract}
		
		\pagenumbering{arabic}

\section{Introduction}

Cohen, Lenstra, and Martinet \cite{Cohen1984, Cohen1990} have given heuristics on the distribution of class groups of number fields that lead to the following conjecture.  For a number field $k$, a transitive permutation group $G$, and a prime $p\nmid |G|$, there conjecturally exists a constant $c_{k,G,p}>0$ such that
\begin{equation}\label{E:conj}
\lim_{X\ra\infty} \frac{1}{|E_k(G,X)|}\sum_{K\in E_k(G,X) } |\Cl_K[p]| \stackrel{\textrm{Conj.}}{=}c_{k,G,p},
\end{equation}
where $E_k(G,X)$ is the set of extensions $K/k$ with Galois closure group $G$ (that is, the Galois group of the normal closure; see Section \ref{sec:notations}) 
and $\Disc(K)\leq X$.
The limiting average in \eqref{E:conj} was previously known for only two cases of $G$ and $p$: 
for $G=S_2$ and $p=3$
by
Davenport and Heilbronn \cite{DH71} when $k=\Q$ and  Datskovsky and Wright \cite{DW88} for general $k$, and for $G=S_3$ and $p=2$
 by  Bhargava \cite{Bha05} when $k=\Q$ (see also the work of Bhargava, Shankar, and Wang \cite{BSW15} for the case of general $k$).  
This conjecture is particularly notable in light of the fact that the best general upper bound on $|\Cl(K)[p]|$ is $\Disc(K)^{1/2+\epsilon}$, and there has  been a lot of recent breakthrough work (e.g. \cite{Heath-Brown2017, PTBW}) 
even  to show that the average in \eqref{E:conj} is bounded by $X^{1/2-\delta}$ for some $\delta>0$, for certain $G$, despite the fact that it is conjectured to have constant limit.

Our first main result is that we find the average and prove the conjecture in \eqref{E:conj} when $G$ is any transitive $2$-group containing a transposition  and $p=3$.

\begin{theorem}\label{T:G}
For any $m$, let $G\sub S_{2^m}$ be a transitive permutation $2$-group containing a transposition.
For any number field $k$, there exists a constant $c_{k,G,3}$ (given explicitly in Section~\ref{sec:main-theorem}) such that
$$
\lim_{X\ra\infty} \frac{1}{|E_k(G,X)|}\sum_{K\in E_k(G,X) } |\Cl_K[3]| =c_{k,G,3}.
$$
\end{theorem}

The simplest new example that Theorem \ref{T:G} provides is that the average size of $3$-torsion in the class groups of $D_4$-quartic fields over $k$ tends to a constant; when $k=\mathbb{Q}$, this constant works out to be around $1.42$. 
It applies also to three permutation groups $G$ of degree $8$, to $26$ groups of degree $16$, and to at least as many groups in degree $2^m$ as in degree $2^{m-1}$. 

The restriction that $G$ has a transposition is necessary for our proof, and is a natural condition in this setting.  In particular, it is expected (for example, according to Malle's conjecture \cite{Mal04}) that the groups $G$ arising with positive density as Galois groups when fields of degree $n$ are ordered by discriminant are precisely those with a transposition in their degree $n$ permutation representation. This is known for $n=4$ where there are $\sim c_1 X$ quartic $D_4$ fields \cite{CDyDO02} and $\sim c_2X$ quartic $S_4$ fields \cite{Bha05} of absolute discriminant bounded by $X$ for certain constants $c_1,c_2 > 0$, and $O_\epsilon(X^{1/2+\epsilon})$ quartic fields with any other Galois closure group.  
Thus, the $2$-groups to which Theorem \ref{T:G} applies---or, indeed, the groups susceptible to our method as a whole---are  those expected to have  positive density within the set of all fields of a given degree, which you might think of as ``generic'' Galois groups.

When $G$ is a transitive permutation $2$-group with a transposition and $K\in E_k(G,X)$, then $K$ has a unique index two subfield $F$, and we can also ask about the relative class group of $K/F$. Indeed for $G=D_4$ Cohen and Martinet \cite{Cohen1987} have said this part of $\Cl_K$ is of particular interest.  
Our next main result gives the averages of $|\Cl_{K/F}[3]|=|\Cl_K[3]|/|\Cl_F[3]|$.

\begin{theorem}\label{T:rel}
For any $m$, let $G\sub S_{2^m}$ be a transitive permutation $2$-group containing a transposition.
Let $k$ be a number field, $u$ be an integer, and  $E^u_k(G,X)$ be those extensions $K\in E_k(G,X)$ such that $\rk \O_K^*-\rk \O_{F_K}^*=u$, where $F_K$ is the unique extension of $k$ such that $[K:F_K]=2$. Then, if $E^u_k(G,\infty)$ is non-empty,
$$
\lim_{X\ra\infty} \frac{1}{|E^u_k(G,X)|}\sum_{K\in E^u_k(G,X) } |\Cl_{K/F_K}[3]| =1+3^{-u},
$$
and this particular average  value is predicted by the Cohen--Lenstra--Martinet heuristics.
\end{theorem}

Cohen, Lenstra, and Martinet actually gave a conjecture that predicts the average of $f(\Cl_K)$ over $G$-extensions $K/k$ for any function $f$ not depending on certain ``bad'' Sylow subgroups of the class group. 
The only non-trivial cases of these conjectures previously proven were for averages of the two $f$ of the form $f(\Cl_K)=|\Cl_K[p]|$ mentioned above.
Theorem \ref{T:rel} thus gives the first proof the Cohen-Lenstra-Martinet conjecture for a function not of the form $|\Cl_K[p]|$ (or trivial variations), and does this for infinitely many groups $G$.
(When $G=C_4$ and $f$ is bounded and only depends on the class group of the quadratic subfield of the cyclic quartic,
Bartel and Lenstra \cite{Bartel2020}  showed that the average is computable in finite time to arbitrary
precision, though the example they compute does not agree with the conjecture---see Section~\ref{sec:CM} for further discussion. 
There have
also been some averages proven on the ``bad'' part  of the class group including work of Fouvry and Kl\"uners \cite{Fouvry2006a} and Klys \cite{Klys2016a}, and the recent groundbreaking work of Smith \cite{Smith2017} determining the entire distribution of $\Cl_K[2^\infty]$ when $G=S_2$. Other work in this direction includes that of Gerth \cite{Gerth1984, GerthIII1987} and Koymans and Pagano \cite{KP18}.)

\subsection{Methods}

The condition that $G$ contains a transposition implies that $G=C_2\wr H$ for some $H$.
So, for the extensions $K/k$ we consider, $K$ has a unique index two subfield $F$, and the overarching strategy of the paper is to 
average and then sum over each $F$.  
Accordingly, the work of Datskovsky and Wright on the average of $|\Cl_K[3]|$ over quadratic extensions of a fixed $F$ is a key input into our work.  
If we could use just the main term of the average in place of the average, the argument would be essentially straightforward.    
However, summing the error terms over $F$ in a straightforward way leads to error that is larger than the main term. 

The major obstacle we overcome in this paper is proving a sufficiently good tail bound so that we can sum  over $F$.  In particular, the challenge is to bound the contribution when the  average 
is taken over a relatively small interval, given $\Disc(F)$.
Our first tool to do this is a heavily optimized upper bound on the average $3$-class number of quadratic extensions of a number field $F$, and the closely related
number of $S_3$ extensions of $F$, which we develop in Section~\ref{sec:uniform-cubic}.
  For very small or very large intervals, compared to $\Disc(F)$, we use class field theory and the theory of the Shintani zeta function counting cubic orders due to Datskovsky and Wright \cite{DW88}, respectively.
While cubic extensions of a general number field are counted in \cite{DW88}, in order to sum over different base fields, we must determine the explicit dependence of the count on the base field, which has not been done in previous work.

For intermediate intervals, however --- summing $|\Cl(K)[3]|$ for roughly $\Disc (K)\in[\Disc(F)^{3.1}, \Disc(F)^6]$ --- 
 we require a new technique, which we call propagation of orders.
In this range, the bound on cubic extensions given by the Shintani zeta function is too large.  However, the Shintani zeta function we use counts non-maximal orders as well as maximal orders.  Given a maximal order, one can prove a lower bound on the number of non-maximal orders inside it, up to a certain discriminant bound.  
Then, to count maximal orders of discriminant up to $X$, we count all orders of discriminant up to $Z$ for some $Z>X$, and then account for the known overcounting.
As $Z$ gets larger the total count goes up but the known overcounting also improves, and it turns out in the intermediate range these trade-offs work out to give us an improved and sufficiently good bound.

With these optimized bounds on $S_3$ extensions of a single $F$, especially in shorter intervals, we begin to take the sum over $F$, and we can bound the sum as required except in one problematic range, which we call the critical range.  In this range, $\Disc(K)\approx \Disc(F)^3$. If $\Disc(K)$ is a power smaller or larger than this, then the methods above are sufficient. In the critical range, the best individual bounds we can prove on $S_3$ extensions of $F$
use class field theory and a trivial bound for $|\Cl_{K/F}[3]|$, and are too large when summed.
Ellenberg and Venkatesh \cite{EV} have shown that one can improve the trivial bound when $K$ has enough small split primes.  
Ellenberg, Pierce, and the third author \cite{Ellen16}  introduced the idea that in many situations, one can prove that most fields in a family have enough 
small split primes to apply  \cite{EV}.  We use that approach, though there is one significant new aspect.  
As we sum $|\Cl_{K}[3]|=|\Cl_{F}[3]||\Cl_{K/F}[3]|$ over $F$, we require a significant saving over the trivial bound.
We structure our proof as an induction, which allows us to assume an optimal (constant average) bound for the $|\Cl_{F}[3]|$ factor.
We then further need a non-trivial bound on the relative class number 
$|\Cl_{K/F}[3]|= |\Cl_{K}[3]|/|\Cl_{F}[3]|.$  
For this, in Section~\ref{sec:arakelov} we develop a relative version of the widely-used lemma of Ellenberg and Venkatesh \cite{EV}, using the relative Arakelov  class group to show that one obtains a non-trivial bound on $|\Cl_{K/F}[p]|$ when there are sufficiently many small split primes in $K/F$.  
To find that most fields have small split primes, in Section~\ref{sec:zero-density} we use a uniform zero density estimate for Hecke $L$-functions as in \cite[Proposition A.2]{Pasten} or \cite[Theorem 1.2]{ThornerZaman-LargeSieve} to show that most of the Dedekind zeta functions do not have zeros in the relevant region.  
In Section~\ref{sec:tail}, we put these results together to prove the desired tail bound, and in Section~\ref{sec:main-theorem} we prove the main theorems.

Our methods apply more broadly than just to $2$-groups.  We determine in Section \ref{sec:other-groups} the average size of $3$-torsion in the class group of $(C_2 \wr G)$-extensions for many general classes of group $G$, where $C_2 \wr G$ denotes the wreath product.

\subsection{Previous Work}

There is a large body of previous work in arithmetic statistics and analytic number theory on the questions that arise in the this paper. There have been many remarkable achievements on counting cubic extensions, and relatedly averaging three torsion in class groups over quadratic extensions, after the work of Davenport and Heilbronn \cite{DH71} and Datskovsky and Wright \cite{DW88} discussed above. Taniguchi and Thorne \cite{TT13} and Bhargava, Shankar, and Tsimerman \cite{BST} improved the error term in Davenport and Heilbronn's count so far as to uncover a secondary term, and recent work by Bhargava, Taniguchi, and Thorne \cite{Bhargava2021} has improved the error term even further.  Bhargava, Shankar, and Wang \cite{BSW15} have given a geometry of numbers  proof of the counting results of Datskovsky and Wright, opening the door to further applications.  See also Thorne's exposition \cite{Th12} of the different approaches that have been used to count cubic fields.

The problem of proving non-trivial bounds for $|\Cl_K[p]|$ has also been the topic of many important works.  Pierce, Helfgott, and Venkatesh \cite{Pie05,HV06,Pie06} proved the first non-trivial bounds (at a prime $p$ not dividing the order of the Galois group), breaking the trivial bound in the case that $K$ is quadratic field and $p=3$.  This bound was improved by Ellenberg and Venkatesh \cite{EV} to $|\Cl_K[3]|=O_\epsilon(\Disc K^{1/3+\epsilon})$ for quadratic fields $K$.
 Bhargava, Shankar, Taniguchi, Thorne, Tsimerman, and Zhao \cite{BSTTTZ} gave a non-trivial bound on $|\Cl_K[2]|$  for all $K$ using geometry of numbers.   
  Ellenberg and Venkatesh \cite{EV} also proved a non-trivial bound on $|\Cl_K[3]|$ whenever $[K:\Q]\leq 4$.
The second author gave a non-trivial bound on $|\Cl_K[p]|$ when $K$ is Galois with the Galois group a non-cyclic and non-quaternion $\ell$-group or a nilpotent group where every Sylow-$\ell$ subgroup is non-cyclic and non-quaternion \cite{Wang2021a,Wang2020}.  In the setting when $p$ divides the order of the Galois closure group of $K$, Kl\"{u}ners and the second author \cite{KluWan} proved $|\Cl_K[p]|=O_\epsilon(\Disc K^{\epsilon})$ when $K$ has Galois closure group a $p$-group, generalizing the classical result from Gauss's genus theory for $K$ quadratic and $p=2$.  

 Soundararajan \cite[pg.690]{Sound-Divisibility} was the first to observe that non-trivial bounds for $|\Cl_K[p]|$ could be proved for almost all imaginary quadratic $K$, or equivalently, on average over $K$.
 Heath-Brown and Pierce \cite{Heath-Brown2017} reinvigorated the study of such average bounds for $|\Cl_K[p]|$, as well as bounds on higher moments, by using the large sieve.
As mentioned above, Ellenberg, Pierce, and the third author \cite{Ellen16} proved non-trivial bounds for most fields in low degree families using 
a non-trivial bound of Ellenberg and Venkatesh \cite{EV} conditional on the existence of small split primes, and precise field counting results with local conditions to prove such existence.  
Pierce, Turnage-Butterbaugh, and the third author \cite{PTBW} found ways to prove non-trivial bounds for many more families using
zero density estimates for $L$-functions to prove the required existence of small primes, which is a strategy we use in this paper.  
This strategy has been improved and extended to further cases by An \cite{An2020}, Frei and Widmer \cite{Frei2018,Frei2021}, Widmer \cite{Widmer2018}, Thorner and Zaman \cite{Thorner2019}, and the first author, Thorner and Zaman \cite{Oliver2021}.

\subsection{Further questions}

We expect the  methods in this paper can be used in other settings.  We have forthcoming work with Alberts which uses a similar inductive approach to prove many new cases of Malle's conjecture \cite{Mal04} on the asymptotics of $|E_k(G,X)|$.  
It would be natural to apply our approach to the $2$-part of class groups of fields that can be obtained as cubic extensions of other extensions,
using \cite{Bha05} as a base case.

It was noted by Bhargava and Varma \cite{BhargavaVarma} that the average $3$-torsion in class group of quadratic extensions does not change when the fields are subject to finitely many local conditions at finite primes (see also \cite{Wood2018}).  We expect that the average in Theorem \ref{T:rel} behaves similarly and is insensitive to local conditions, while the average in Theorem \ref{T:G} behaves differently and is quite sensitive to local conditions, and it would be interesting to investigate whether this is the case.  A subtlety in this question is in defining  a notion of local conditions that works well in towers of fields.  

The contrast between the constants in Theorems~\ref{T:G} and \ref{T:rel} naturally leads to the question: when ordering fields by a particular invariant, which averages does one expect to be as predicted by Cohen--Lenstra--Martinet?  For example, does the analogue of Theorem \ref{T:rel} hold when the fields $K$ are instead ordered by the 
Artin conductor of the representation of $\Gal(\bar{k}/k)$ induced from the sign representation of $\Gal(K/F_K)$?  (This is an irreducible representation, as we show in the proof of Theorem~\ref{P:asCM}.) In the case of quartic $D_4$ extensions of $k=\mathbb{Q}$, the fields themselves were recently counted by this Artin conductor in work of Alt\v{u}g, Shankar, Varma, and Wilson \cite{D4VS}. 

It would be interesting to extend our methods to as wide a range of groups $G$ as possible (e.g., by incorporating ideas from Sections \ref{sec:arakelov} and \ref{sec:zero-density} into an argument as in Section \ref{sec:other-groups}).  For example, does an analogue of Theorem \ref{T:G} hold for $(C_2 \wr G)$-extensions for arbitrary abelian groups $G$?  Arbitrary $p$-groups?

Finally, as noted above, Malle's conjecture predicts asymptotics for $|E_k(G,X)|$.  These predicted asymptotics are of the form $|E_k(G,X)| \sim c_{G,k} X^{1/a(G)} (\log X)^{b(k,G)-1}$ for specified integers $a(G)$ and $b(k,G)$, and an unspecified positive constant $c_{G,k}$, as $X$ tends to infinity.  Our work raises the question, what is the least exponent $\alpha = \alpha(G)$ such that $|E_k(G,X)| \ll_{[k:\mathbb{Q}],\epsilon} \Disc(k)^{\alpha+\epsilon} X^{1/a(G)} (\log X)^{b(k,G)-1}$ for all $k$, all $\epsilon>0$, and all $X \geq 1$?  Does such an exponent even exist?  When $G$ is abelian, this problem is closely related to bounding torsion subgroups of class groups.  Some of the key technical results of this paper (see Section \ref{sec:uniform-cubic}) provide the first such result for a non-abelian group $G$, namely $G=S_3$.  What can be said for other non-abelian groups $G$?  
This question when $G=S_3$ could naturally be approached by proving a discriminant-aspect subconvexity bound for the relevant Shintani zeta functions, which the methods of Section \ref{sec:uniform-cubic} fall a little short of proving.  Can these or other methods be adapted to prove such a subconvexity result?  We mention recent work of Hough and Lee \cite{HoughLee} that establishes a subconvexity bound in $t$-aspect for the Shintani zeta functions over $\mathbb{Q}$, though we expect these methods do not easily apply to discriminant-aspect.

\subsection{Technical remarks}

The conjectures of Cohen, Lenstra, and Martinet in fact fix some archimedean data and only consider families of fields with this fixed data.  In particular, for Galois $\Gamma$-extensions $L/k$, they fix the $\Gamma$-module structure of $\O_L^* \tensor_\Z \Q$, which is equivalent \cite[Th\'{e}or\`{e}me 6.7]{Cohen1990} to fixing the representation $\oplus_{v} \Ind_{\Gamma_v}^{\Gamma} \Q$, where the sum is over infinite places $v$ of $k$ and $\Gamma_v$ is the decomposition group of $L/k$ at $v$, and we  call this the \emph{enriched signature}.  The papers of Cohen, Lenstra, and Martinet are agnostic on the relative frequency of the various enriched signatures (and even whether such limiting frequencies exist). So, to be precise, to obtain \eqref{E:conj} from their conjectures one must also assume such frequencies exist (which is natural to conjecture given all known counting results, e.g. \cite{ DH71, Wright, CDyDO02, Bha05}, and is conjectured by Malle \cite{Mal04}), and the conjectural value would be given as a weighted average of particular values for each enriched signature as conjectured by Cohen, Lenstra, and Martinet, where the weights are these frequencies.  
When one fixes an enriched signature, one can use the results of \cite{Wang2021} to find the average value predicted by Cohen, Lenstra, and Martinet (see Section~\ref{sec:CM}).

We prove a version of Theorem~\ref{T:G} with a specified enriched signature (in fact with slightly more specification at infinity) in Theorem~\ref{thm:main}.  However, even when we do this, we do not apparently obtain the values conjectured by Cohen and Martinet in \cite{Cohen1990}.  We show this computationally is the case in Section~\ref{SS:wrongpred}.  This is because in this setting each $F$ appears as the index $2$ subfield for a positive proportion of $K$, and so the $3$-ranks of the particular $F$ that appear with significant frequency bias the statistics, as in the work of Bartel and Lenstra \cite{Bartel2020}.  When we only average $\Cl_{K/F}$ as in Theorem~\ref{T:rel}, this bias is removed.  See Section~\ref{sec:CM} for further discussion of this phenomenon.

\subsection{List of notations}\label{sec:notations}
\noindent
$k$: a number field, for us a subfield of a fixed algebraic closure  $\bar{\Q}$\\
$\O_k$: the ring of algebraic integers in $k$\\
$\disc(F/k)$: relative discriminant ideal in $\O_k$, for an extension $F/k$\\
$\Nm I$: norm $|\O_k/I|$ of an ideal $I$ in $\O_k$ and extend to any fractional ideals in $\O_k$ multiplicatively \\ 
$\Nm_{F/k} I$: an ideal in $\O_k$ that is $(I\cap \O_k)^{f(I| I\cap \O_k)}$ when $I$ is a prime ideal in $\O_F$ and $f$ is the inertia degree of $I$ over $I\cap \O_k$, and defined for any fractional ideals in $\O_F$ multiplicatively \\
$\Disc(F/k)$: $\Nm(\disc(F/k))$ (When $k$ omitted, $k=\Q$) \\
$\Gal(F/k)$: the \emph{Galois closure group} of $F/k$, defined as the Galois group of $\tilde{F}/k$ as a \emph{permutation group}, where $\tilde{F}$ is the Galois closure of $F$ over $k$ (in $\bar{\Q}$) acting on the embeddings of $F\ra \tilde{F}$ that fix $k$ pointwise\\  
$G$-extension: for a permutation group $G$, an extension $L/k$ of number fields (in $\bar{\Q}$), \emph{and a choice of isomorphism} $\phi:\Gal(L/k)\isom G$ \emph{as permutation groups}
(see Definition~\ref{D:permgroup})
\\
$2$-extension: a $G$-extension for some $2$-group $G$\\
$E_k(G,X)$: the set of $G$-extensions $F/k$ with $\Disc(F)\le X$ 
\\
$N_k(G,X)$: the number of $G$-extensions $F/k$ with $\Disc(F/k)\le X$ (note relative discriminant, versus absolute above)\\
$\Cl_F$:  class group of $F$ \\
$\Cl_{F/k}$: relative class group of $F/k$, when $k=\mathbb{Q}$ it is the usual class group of $F$ (see Section~\ref{sec:arakelov})
\\
$A[\ell]$: $\{ [\alpha]\in A \mid \ell [\alpha] = 0\in A \}$ for an abelian group $A$\\
$h(F/k)$, $h(F)$: the size of $\Cl_{F/k}$, $\Cl_{F}$\\
$h_{\ell}(F/k)$, $h_{\ell}(F)$: the size of $\Cl_{F/k}[\ell]$, $\Cl_{F}[\ell]$\\
$\ACl_L$: Arakelov class group\\
$\pi_k(X)$: the number of prime ideals $\mathfrak{p}$ in $k$ with $\Nm(\fp)<X$\\
$\pi_k(Y; F, \mathcal{C})$: the number of unramified prime ideals $\mathfrak{p}$ in $k$ with $\Nm_{k/\Q}(\fp)<Y$ and with Frobenius at $\fp$ in the conjugacy class of $\mathcal{C}$ in $\Gal(F/k)$\\
$r_1(k)$: the number of real embeddings of $k$\\
$r_2(k)$: the number of pairs of complex embeddings of $k$\\
$\rk \O_K^*$: the unit rank, i.e. $\dim_{\Q}  (\O_k^*\tensor \Q$)\\
$\zeta_k(s)$: Dedekind zeta function of $k$\\
$\kappa_k$: the residue $\mathrm{Res}_{s=1}\zeta_k(s)$\\
$\Rg_k$: the regulator of $k$\\
$f(X) \sim g(X)$: $\lim_{X\to \infty} f(X)/g(X) = 1$ or $\lim_{X\ra\infty} f(X)=\lim_{X\ra\infty}g(X)=0$\\ 
$f(X) =O_{p_1,\dots}(g(X))$: there exists a constant $C$ depending only on $p_1,\dots$ such that for all values of all variables $|f(X)|\leq Cg(X)$\\
$f(X) \ll_{p_1,\dots}g(X)$: $f(X) =O_{p_1,\dots}(g(X))$\\
$f(X) \gg_{p_1,\dots}g(X)$: only used when both sides non-negative, and then means $g(X) \ll_{p_1,\dots}f(X)$

\section*{Acknowledgements}
We thank Arul Shankar, Frank Thorne, and Jesse Thorner for helpful conversations related to this project.  We also thank Brandon Alberts, Alex Bartel, Jordan Ellenberg, 
J\"{u}rgen Kl\"{u}ners,
Lillian Pierce, Frank Thorne, Jesse Thorner, and Asif Zaman for comments on an earlier version of this paper.
RJLO was partially supported by National Science Foundation grant DMS-1601398.  JW was partially supported by a Foerster-Bernstein Fellowship at Duke University.   MMW was partially supported by a Packard Fellowship for Science and Engineering, National Science Foundation grant DMS-1652116, and an NSF Waterman award.

\section{Nontrivial Bound for the Relative Class Group}
\label{sec:arakelov}
%%%%%%%%%%%%%%%%
%%%%%%%%%%%%%%%%
In this section, our main goal is to establish Lemma \ref{lem:relative-E-V}, which is a \textit{relative} version of a widely used lemma of Ellenberg--Venkatesh \cite[Lemma 2.3]{EV}.   
Let $L/K$ be an extension of number fields. The relative class group $\Cl_{L/K}\subset \Cl_L$ is defined to be the kernel of the map $\Nm_{L/K}: \Cl_L\to \Cl_K$  induced from the usual norm on fractional ideals of $L$. 

\begin{lemma}[Relative Ellenberg--Venkatesh]\label{lem:relative-E-V}
	Let $L/K$ be an extension of number fields with $[L:K]=d$, let $\ell\ge 1$ be an integer,  and let $\theta$ be a real number such that $0< \theta< \frac{1}{4\ell(d-1)}$. Suppose there exist $M$ pairs of distinct primes $(\wp_i, \bar{\wp}_i)$ in $\O_L$ and coprime integers $a_i,b_i$ such that 
	\begin{enumerate}
		\item
		For each $i$, we have $\wp_i \cap \O_K = \bar{\wp}_i\cap \O_K$ and $\Nm_{L/K}(\wp_i)^{a_i} = \Nm_{L/K}(\bar{\wp}_i)^{b_i}$;
		\item 
		For each $i$, the norms satisfy $ \Nm(\wp_i)^{a_i}=\Nm(\bar{\wp_i})^{b_i}<\Disc(L/K)^{\theta}$; and
		\item
		For each $i$, there is no intermediate field $H$ with $K\subset H \subsetneq L$ such that $\wp_i$ and $\bar{\wp_i}$ are both extended from $H$. (A prime $\wp\subset \O_L $ is \emph{extended} from $H$ if there exists a prime $\fp\subset \O_H$ such that $\wp = \fp \O_L$.)
	\end{enumerate} 
	Then 
	\begin{equation}\label{eqn:EV-relative}
		|\Cl_{L/K}[\ell]|= O_{[L:\Q],\theta,\epsilon}\Big(\frac{\Disc(L/K)^{1/2+\epsilon} \Disc(K)^{(d-1)/2+\epsilon}}{M}\Big).
	\end{equation}
\end{lemma}
\begin{remark}
	In our application, we will only use prime ideals that are split completely in $L/K$, therefore the first and third conditions will be automatically satisfied, and the second condition will be equivalent to $\Nm(\wp\cap \O_K)<\Disc(L/K)^{\theta}$. 
\end{remark}
The original lemma \cite{EV} gives a non-trivial bound 
\begin{equation}\label{eqn:EV-lemma}
	|\Cl_{L}[\ell]|= O_{[L:\Q],\theta,\epsilon}\Big(\frac{\Disc(L)^{1/2+\epsilon}}{M}\Big),
\end{equation}
for the absolute class group $\Cl_L[\ell]$ under the same assumptions. Our improvement is in proving a saving from the trivial bound on the size of the \emph{relative} class group $\Cl_{L/K}[\ell]$ instead of the \emph{absolute} class group. Such a refinement is crucial for our application since we will treat $\Cl_K$ and $\Cl_{L/K}$ separately.  

When $\ell$ is relatively prime to $[L:K]$, there is a short exact sequence 
\begin{center}\label{eqn:relative-class-group-size}
	\begin{tikzcd}				
		0\arrow{r} & \Cl_{L/K}[\ell] \arrow{r} & \Cl_L[\ell]   \arrow[r,"\Nm_{L/K}"] & \Cl_K[\ell] \arrow{r} & 0.
	\end{tikzcd}
\end{center}
Let $\iota$  be the natural map from $\Cl_K$ to $\Cl_L$ by extension of ideals. The composition $\Nm_{L/K}\circ \iota: \Cl_K[\ell] \to \Cl_K[\ell]$ is equal to multiplication by $[L:K]$ and thus is an isomorphism since $(\ell, [L:K]) =1$, so $\iota$ induces a splitting section of the above sequence. So we have
$$ \Cl_{L}[\ell] \simeq \Cl_{L/K}[\ell] \times \Cl_{K}[\ell], \quad\quad   h_{\ell}(L) = h_{\ell}(K)\cdot h_{\ell} (L/K).$$
This leads to a strategy we use in this paper to bound  $h_{\ell}(L)$, that is, to bound $h_{\ell}(K)$ and $h_{\ell}(L/K)$ separately. 

The organization of this section is as follows. In Section \ref{ssec:relative-arakelov-set-up}, we give a definition of a \emph{relative Arakelov class group}, compute the volume of it and derive the ``trivial'' bound for the relative class group in Lemma \ref{lem:relative-trivial}. In Section \ref{ssec:relative-EV}, we prove Lemma \ref{lem:relative-E-V}.

\subsection{The relative Arakelov class group}\label{ssec:relative-arakelov-set-up}

In this section we will define a relative version of the Arakelov class group.  First we recall the usual 
\emph{absolute} notion of the Arakelov class group.  A nice reference for the theory of the  absolute Arakelov class group is \cite{Schoof}. 

For any number field $L$ with degree $n$ and $r=r_1+r_2$ infinite places, we let $I_L$ be the group of fractional ideals and $L_\infty := \prod_{v|\infty} \R^{\times}_{>0}$. We define the degree $0$ divisors to be $\Div^0(L) := \{(x,\mathfrak{a}) \in L^{\times}_\infty \times I_L: |x|_{\infty} =\Nm(\mathfrak{a})\}$ where $\Nm(\mathfrak{a}):=|\O_L/\mathfrak{a}|$ and $|x|_{\infty} = \prod_{v|\infty} |x_v|$. There is a natural embedding $\pi: L^{\times} \to \Div^0(L)$ by sending $\alpha\in L^{\times}$ to $ ((|\alpha|_{v}), (\alpha)) \in \Div^0(L)$ from the product formula (here $|\cdot|_v$ is the usual norm in $\R$ at the $r_1$ real embeddings and the square of the usual norm in $\R$ at the $r_2$ complex embeddings).

The  absolute Arakelov class group $\ACl_L$ of $L$  is then defined as $\ACl_L:= \Div^0(L)/\pi(L^{\times})$, and it fits into the following short exact sequence
\begin{center}\label{fml:Arakelov-SES}
	\begin{tikzcd}	
		0\arrow{r} & (\oplus \R)^{(0)}/\mathcal{L}(L)  \arrow{r} &  \ACl_L \arrow{r} & \Cl_L \arrow{r} & 0,
	\end{tikzcd}
\end{center}
where $(\oplus_v \R)^{(0)}$ is isomorphic to $L_\infty^{(1)}$ (the elements of $L^{\times}_\infty$ of norm $1$) via the Minkowski embedding, i.e. the logarithm map $\log=\oplus_v \log$, and $\mathcal{L}(L) =\log( \pi(\O_L^{\times}))$ is a full rank lattice by Dirichlet's unit theorem. 

 The Arakelov class group is not only an abelian group but also a Lie group. 
We fix a left-invariant Riemannian metric on $\widetilde{\Cl}_{L}$
that on  $(\oplus_v \R)^{(0)}$ agrees with the restriction of the standard Riemannian metric from $\oplus_v \R$.
We will use this Riemannian metric on $\widetilde{\Cl}_{L}$, and the induced metrics on various sub-Lie groups, to compute all volumes in this section.
We compute $\Vol(\ACl_L) = h(L) \cdot \Rg(L)\cdot \sqrt{r}$.
Though the language of a Riemannian metric is convenient, we remark that in our case this metric is very concrete
(in particular, because unlike \cite{EV}, we are using the \emph{non-oriented Arakelov class group} as in \cite[Section $5$]{Schoof}). 
  All of our groups
are a disjoint union of finitely many copies of quotients of subspaces of $\R^n$, and the volumes we compute are always  the ones that come from  the usual metric on $\R^n$.

The absolute Arakelov class group associated to $L$ and $K$ can be related by a norm map  $\Nm_{L/K}: \ACl_L\ra \ACl_K$
where we use the relative norm map $\Nm_{L/K}:I_L\to I_K$ for fractional ideals and $\Nm_{\infty}: (\oplus_v \R)^{(0)} \ra (\oplus_w \R)^{(0)}$ that sends $(x_v)_{v\mid\infty}\mapsto 
(\sum_{v\mid w} x_v)_{w\mid\infty}$ for archimedean embeddings. This norm map $\Nm_{L/K}$ fits into the following commutative diagram: 
\begin{equation}\label{diag:Arakelov-relative}
	\begin{tikzcd}				
		0\arrow{r} & (\oplus_v \R)^{(0)}/\mathcal{L}(L) \arrow{r}\arrow[d,"\mathrm{Nm}_\infty"] &  \ACl_L \arrow{r}\arrow[d, "\mathrm{Nm}_{L/K}"] & \Cl_L \arrow{r}\arrow[d, "\mathrm{Nm}_{L/K}"] & 0 \\		
		0\arrow{r} & (\oplus_w \R)^{(0)}/\mathcal{L}(K)  \arrow{r} & \ACl_K \arrow{r} & \Cl_K \arrow{r} & 0.
	\end{tikzcd}
\end{equation}

We now can define the relative Arakelov class group in a manner analogous to the relative class group, as a subgroup of the absolute Arakelov class group.

\begin{definition}[Relative Arakelov Class Group]
	Given an extension $L/K$ of number fields, 
	we define the \emph{relative Arakelov class group}  $\widetilde{\mathrm{Cl}}_{L/K}$
	to be the kernel of $\Nm_{L/K}: \ACl_L\to \ACl_K.$
\end{definition}

We obtain a
short exact sequence for the relative Arakelov class group as an abelian group, similar to the sequence \ref{fml:Arakelov-SES}:
\begin{equation}\label{diag: Arakelov-relative}
	\begin{tikzcd}[node distance = 1.5cm]		
		0\rar & \mathrm{Ker}(\mathrm{Nm}_\infty) \arrow{r} &  \ACl_{L/K} \arrow{r} & \Cl_{L/K} \arrow{r} & 0,
	\end{tikzcd}
\end{equation}
from an application of the snake lemma on \eqref{diag:Arakelov-relative}, since $\Nm_{\infty}$ is surjective.

Our next goal is to show that $\Vol(\Ker(\Nm_{\infty})) = C\Rg_L/\Rg_K$ where $C$ is a constant only depending on the signatures of $L$ and $K$. Let $v_{i,j}$ be the standard coordinates in $\oplus_v \R$ and $w_i$ in $\oplus_w \R$, where $\Nm_{\infty}$ maps $w_i = \sum_{1\le j\le i_k} v_{i,j}$ and $i_k$ is the number of infinite places above $w_i$. We then choose
a new coordinate system for $\oplus_v \R$ to be $\tilde{v}_{i,j}$ where $\tilde{v}_{i,j} = v_{i,j}$ for $j>1$ and $\tilde{v}_{i,1} = \sum_{j} v_{i,j} = w_i$. Then we can compute the covolume of $(\oplus_v \R)^{(0)}/\mathcal{L}(L)$ using the coordinate system $v_{i,j}$ or $\tilde{v}_{i,j}$ (without $v_{1,1}$ and $\tilde{v}_{1,1}$ respectively):
$$\Rg_L =\int_{\tilde{v} \in \mathcal{F}_L} \rd \tilde{v} = \int_{w \in \mathcal{F}_K} \rd w \int_{\tilde{v}' \in \Nm_{\infty}^{-1}(w)} \rd \tilde{v}'= \Rg_K \cdot \int_{\tilde{v}' \in \Nm_{\infty}^{-1}(0)} \rd \tilde{v}'.$$
Here ${\mathcal{F}}_L$ is the projection of a fundamental domain for $(\oplus_v \R)^{(0)}/\mathcal{L}(L)$ onto the $v_{i,j}((i,j)\neq (1,1))$-coordinate plane, and we define $\mathcal{F}_K$ analogously. The first equality comes from the definition of regulator. The second equality is an iterated integral where we use $\tilde{v}'$ to denote the variables in $\tilde{v}$ aside from $\tilde{v}_{i,1}$, and $\Nm_{\infty}^{-1}(w)= \{ v\in \mathcal{{F}}_L\mid \Nm_{\infty}(v)\equiv w \mod \mathcal{L}(K)\}$. In the last equality $\int_{\tilde{v}' \in \Nm_{\infty}^{-1}(0)} \rd \tilde{v}' = \int_{\tilde{v}' \in \Nm_{\infty}^{-1}(w)} \rd \tilde{v}'$ for any $w\in \mathcal{F}_K$ since the two domains for $\tilde{v}'$ only differ by a translation. 
We have $ \Vol(\Ker(\Nm_{\infty})) = \prod_{i} \sqrt{i_k} \cdot \int_{\tilde{v}' \in \Nm_{\infty}^{-1}(0)} \rd \tilde{v}' =  \prod_{i} \sqrt{i_k} \frac{\Rg_L}{\Rg_K}$. 

In the following lemma we give the description of the cokernel of the map $\mathrm{Nm}_{L/K}: \Cl_L  \to \Cl_K$.
\begin{lemma}\label{lem:relative-Coker}
	Given an extension $L/K$ of number fields, 
	the cokernel of $\mathrm{Nm}_{L/K}: \Cl_L  \to \Cl_K$	 is isomorphic to $\Gal(M/K)$,
	where $M = H_K\cap L$ and $H_K$ is the Hilbert class field of $K$.
\end{lemma}
\begin{proof}
	By class field theory, the map $\Nm_{L/K}: \Cl_L  \to \Cl_K$ agrees with the restriction map on the Galois groups $\Gal(H_L/L) \to \Gal(H_K/K)$ after identifying class groups and Galois groups with the reciprocity map. Since $H_K/K$ is Galois, we  have $\Gal(H_L/L) \twoheadrightarrow \Gal(H_KL/L)\simeq \Gal(H_K/M) = \mathrm{Im}(\Nm_{L/K}: \Cl_L  \to \Cl_K)$, and the lemma follows. 
\end{proof}

As a consequence, we have an upper bound  $|\mathrm{Coker}(\mathrm{Nm}_{L/K}: \Cl_L  \to \Cl_K )|\leq [L:K]$, and thus we can obtain the following upper bound on $h(L/K)$, which we consider as the ``trivial bound'' on the $\ell$-torsion in the relative class group.
\begin{lemma}\label{lem:relative-trivial}
	Given a relative extension $L/K$ with $[L:K]=d$ and an arbitrary integer $\ell>0$, we have
	$$h_{\ell}(L/K) \ll_{[L:\Q],\epsilon}\mathrm{Disc}(L/K)^{1/2+\epsilon} \cdot \mathrm{Disc}(K)^{(d-1)/2+\epsilon}.$$
\end{lemma}
\begin{proof}
	We have 
	\begin{equation}
		h_{\ell}(L/K) \le h(L/K) \le [L:K] \cdot \frac{h(L)}{h(K)} \ll_{[L:\Q], \epsilon} \mathrm{Disc}(L/K)^{1/2+\epsilon} \mathrm{Disc}(K)^{(d-1)/2+\epsilon},
	\end{equation}
	where the first inequality is trivial, and the second inequality follows from Lemma \ref{lem:relative-Coker}. The third inequality comes from  the theorem of Brauer-Siegel,  an absolute lower bound  $\frac{\Rg(L)}{\Rg(K)} \gg_{[L:\Q]} 1$ on the ratio of regulator by \cite{FrSko}, and the expression of the relative discriminant in terms of absolute discriminants.
\end{proof}

Combining the discussion of $\Vol(\Ker(\Nm_{\infty}))$ and $\Cl_{L/K}$, we obtain an estimate of $\Vol(\ACl_{L/K})$.
\begin{lemma}\label{lem:relative-volume}
	Given an extension $L/K$ of number fields with $[L:K]=d$, with the measure above, we have
	$$ \Disc(L/K)^{1/2-\epsilon} \Disc(K)^{(d-1)/2-\epsilon}\ll_{[L:\Q],\epsilon} \Vol(\ACl_{L/K}) \ll_{[L:\Q],\epsilon} \Disc(L/K)^{1/2+\epsilon} \Disc(K)^{(d-1)/2+\epsilon}.$$
\end{lemma}

\begin{remark}
		{\em
		The upper bounds in Lemmas \ref{lem:relative-trivial} and \ref{lem:relative-volume} are both effective, despite the invocation of Brauer--Siegel to estimate the quotient of residues $\mathrm{Res}_{s=1} \zeta_L(s) / \mathrm{Res}_{s=1} \zeta_K(s)$, and the lower bound of Lemma \ref{lem:relative-volume} is effective if $L$ does not contain a nontrivial quadratic extension of the form $K(\sqrt{D})$ for $D \in \mathbb{Q}$.  
		In particular, there is always an effective upper bound on the residue $\mathrm{Res}_{s=1} \zeta_L(s)$ due to Landau (see also \cite[Lemma 3]{Brauer}), and there is an effective lower bound on $\mathrm{Res}_{s=1} \zeta_K(s)$ whenever $\zeta_K(s)$ does not have an exceptional Landau--Siegel zero due to Stark \cite{Stark}.   Thus, if $\zeta_K(s)$ does not have an exceptional zero, we simply combine these two results to get an effective upper bound.  When $\zeta_K(s)$ does have an exceptional zero, the quotient $\zeta_{\widetilde{L}}(s)/\zeta_K(s)$ is entire by the Aramata--Brauer theorem \cite[Theorem 1]{Brauer}, where $\widetilde{L}$ denotes the normal closure of $L$ over $K$, and does not have an exceptional zero by Stark \cite[Theorem 3]{Stark}.  Moreover, by \cite[Lemma on p.244]{Brauer}, this quotient may be expressed as a product over all $L$-functions associated to nontrivial $1$-dimensional characters of subextensions of $\widetilde{L}/K$, where each such $L$-function appears with a strictly positive rational exponent.  These $L$-functions thus also do not have an exceptional zero, and are subject to effective lower bounds \cite[Lemma 2]{Stark} in addition to the effective upper bounds that hold for any entire $L$-function \cite[Ch. 5]{IwaniecKowalski}.  The quotient $\zeta_L(s)/\zeta_K(s)$ may be expressed as a product and quotient of these $L$-functions, and is thus subject to an effective upper bound in the case that $\zeta_K(s)$ has an exceptional zero as well.  For the lower bound on $\mathrm{Res}_{s=1} \zeta_L(s) / \mathrm{Res}_{s=1} \zeta_K(s)$, if $\zeta_L(s)$ does not have an exceptional zero, we simply use the effective lower bound on its residue and the effective upper bound on that of $\zeta_K(s)$.  If $\zeta_L(s)$ has an exceptional zero, it is inherited either from $\zeta_K(s)$ or from the Dedekind zeta function of a subextension of $L/K$ of the form $K(\sqrt{D})$ with $D \in \mathbb{Q}$ by \cite[Theorem 3]{Stark}.  In the former case, we proceed as in the proof of the effective upper bound to also obtain an effective lower bound; in the latter, there does not appear to be a way to make the lower bound effective.
	}
\end{remark}

\subsection{Relative Ellenberg--Venkatesh Lemma}\label{ssec:relative-EV}
Now we are ready to prove Lemma \ref{lem:relative-E-V}. The idea of the proof is similar to the proof of Lemma $2.3$ in \cite{EV}, where the major difference is we use a different set $T$ of ideals that cut out the non-trivial saving in the proof. As a result, we can show a saving that is completely coming from the relative class group. 

\begin{proof}[Proof of Lemma \ref{lem:relative-E-V}]
We let $\psi: \Div^0(L) \to \ACl_L$ 
to be natural quotient map, and denote the following subgroups of $\Div^0(L)$ 
	$$G = \psi^{-1}(\ACl_{L/K}),\quad P = \pi(L^{\times}), \quad  P_{\ell} = \ell G + P.$$
	We can pull-back the Riemannian metric from $\ACl_L$ to a left-invariant Riemannian metric on $G$.
	The short exact sequence (\ref{diag: Arakelov-relative}) induces an isomorphism $\Cl_{L/K}/\ell\Cl_{L/K} \simeq \ACl_{L/K}/\ell\ACl_{L/K}$.
Therefore we obtain an upper bound on $|\Cl_{L/K}[\ell]| = |\Cl_{L/K}/\ell \Cl_{L/K}|$ by $\Vol( \ACl_{L/K})/ \Vol(\ell \ACl_{L/K})$ where $\ACl_{L/K} = G/P$ and $\ell \ACl_{L/K} = P_{\ell}/P$.
	We will give a lower bound on $\Vol(P_{\ell}/P)$ and derive an upper bound for $\Vol(G/P_{\ell})$ using the trivial upper bound for $\Vol(G/P) = \Vol(\ACl_{L/K})$ in Lemma \ref{lem:relative-volume}.
	
	Now we let $\Sigma_L$  be the set of infinite places of $L$, and for a constant $C>1$ we define
	$$T:=\{(x, J)\in G \mid J= \wp_i ^{a_i}\bar{\wp}^{-b_i}_i\text{ for some $i$, and }  \forall v_1, v_2\in \Sigma_L,  C^{-1/\ell}< |x|_{v_1}/|x|_{v_2} < C^{1/\ell} \} \subseteq \Div^0(L),$$
	where $(\wp_i, \bar{\wp_i})$ for $i = 1, \cdots, M$ are the pairs of prime ideals in $\O_L$ that are given in the statement of the lemma. We will show that
	$$\Vol (P_{\ell}/P) \ge \Vol((\ell T +P)/P)  \gg_{[L:\Q],C} M ,$$
	from which it follows that $|\ACl_{L/K}[\ell]| \ll_{[L:\Q],C} \frac{\Vol(\ACl_{L/K})}{M}.$
	The first inequality is clear since there is a natural embedding $\ell T+P \hookrightarrow P_{\ell} = \ell G+P$. 
	
	It suffices to prove the second inequality. For each $i$, we let $T_i\subset T$  be the subset of elements with $J=\wp_i ^{a_i} \bar{\wp}_i^{-b_i}$
	 and let $D_i$ be the corresponding disk $D_i := (\ell T_i+P)/P \subseteq \ACl_{L/K}$. 
	We now show that for each $i$, the volume of $D_i$ satisfies $\Vol(D_i)\gg_{[L:\Q],C} 1$. 
	We can compute that $\Vol(\ell T_i)$ in $G$ is a constant in terms of $C$ and $[L:\Q]$. If $t_1^{\ell}, t_2^{\ell}\in \ell T_i$ correspond to the same point in $D_i$, then let $u:=t_1^{\ell}/t_2^{\ell}$.  We see that $u\in \O_L^{\times}$ since $u$ has no non-archimedean valuations, and moreover it satisfies $C^{-2}< |u|_{v_1}/|u|_{v_2}< C^2$ for any archimedean valuations $v_1, v_2\in \Sigma_L$. Now all archimedean valuations of $u$ are of roughly the same size and their product is $1$, therefore the number of such $u$ can be uniformly bounded in terms of $[L:\Q]$ and $C$, since the defining polynomial of $u$ has bounded integral coefficients and bounded degree. This then implies that the disk $D_i = (\ell T_i+P)/P$ has $\Vol(D_i) \gg_{[L:\Q],C} 1$.
	
	On the other hand, we now show that for any $i\neq j$, the disks $D_i$ and $D_j$ do not overlap. Suppose $t_i\in T_i, t_j\in T_j$ with $i\neq j$ satisfy $u:=t_i^{\ell} t_j^{-\ell}\in L^{\times}$. 
	Then recall that for $u\in \Div^0_L$, we define the height of $u$ as in \cite{EV} to be 
	$$H(u):= \prod_{v|\infty} \max\{ u_v,1 \} \prod_{\wp} \max\{ |\wp|^{-\text{val}_{\wp}(u)}, 1 \}.$$
	Then for the given $u$, we have $H(u) \le C^{2[L:\Q]} \Nm(\wp_j)^{a_j\ell} \Nm(\bar{\wp}_i)^{b_i\ell } $, which is bounded by $C^{2[L:\Q]}\Disc(L/K)^{2\theta \ell}$ from our assumption on $(\wp, \bar{\wp})$. By \cite[Lemma 2.2]{EV}, there is a constant $D_0 = D_0(\theta, [L:\Q])$ such that when $\Disc(L/K)$ is larger than $D_0$, the height of $u$ is not large enough to generate $L$. Therefore $u$ generates a strict subfield $K(u)\subsetneq L$. This is a contradiction, since the non-archimedean part of $u$ has valuations from $\wp_i$ and $\bar{\wp}_i$, which are not extended from any strict intermediate subfield by assumption.  Thus, when $\Disc(L/K)$ is sufficiently large, we have shown that $D_i\cap D_j = \emptyset$ for any $i\neq j$.
	
	Therefore by choosing an arbitrary absolute constant $C>1$, we get for $\Disc(L/K)$ sufficiently larger $D_0 = D_0(\theta, [L:\Q])$ that
	$$\Vol((\ell T+P)/P) \ge \sum_i \Vol(D_i) \gg_{[L:\Q]} M.$$
\end{proof}

%%%%%%%%%%%%%%%%
%%%%%%%%%%%%%%%%
\section{Uniform bounds on cubic extensions and the average $3$-part of quadratic extensions}
\label{sec:uniform-cubic}
%%%%%%%%%%%%%%%%
%%%%%%%%%%%%%%%%
In this section we are going to bound the sum of $h_3(F/k)$ for quadratic extensions $F/k$ of a general number field $k$, and the closely related
 number of  $S_3$ cubic extensions of $k$.  This proof will employ different techniques and ideas for different regimes of $X$ versus $\Disc(k)$. 
Let $N_k(G,X)$ be the number of isomorphism classes of $G$-extensions $L/k$ with $\Disc(L/k)\le X$.
The main theorem is the following, which we expect will be useful for other applications as well.
\begin{theorem}\label{thm:cubic-bound}
	Let $k$ be a number field, let $h = h_2(k)$ denote the size of the $2$-torsion subgroup of the class group of $k$, and let $D_k = \Disc(k)$.  
For any $X \geq 1$ 
	we have
	\[
\sum_{\substack{[F:k]=2\\\Disc(F/k) \leq X}} h_3(F/k)
	 \ll_{[k:\Q],\epsilon}D_k^\epsilon \cdot
	\begin{cases}
		h X^{3/2} D_k^{1/2}, & \text{if } X \leq D_k h^{-2/3},  \\
		h^{1/3} X^{1/2} D_k^{3/2}, & \text{if } D_k h^{-2/3} \leq X \leq D_k^2 h^{2/3}, \\
		X D_k^{1/2}, & \text{if } D_k^2 h^{2/3} \leq X \leq D_k^{5/2} h^{2/3}, \\
		h^{2/3} D_k^3, & \text{if } D_k^{5/2}h^{2/3} \leq X \leq D_k^3 h^{2/3}, \text{and} \\
		X, & \text{if } X \geq D_k^3 h^{2/3}.
	\end{cases}
	\]
	The same upper bound holds for $N_k(S_3, X)$. (The sum is over $F\sub \bar{\Q}$.) 
\end{theorem}

We obtain the following convenient corollary for general number fields $k$.
\begin{corollary}\label{cor:general-bound}
	For any number field $k$ and any $X \geq 1$, we have 
	$$
		\sum_{\substack{[F:k]=2\\\Disc(F/k) \leq X}} h_3(F/k)
 = O_{[k:\Q], \epsilon}(\Disc(k)^{1+\epsilon} h_2(k)^{2/3} X).
	$$
	The same upper bound holds for $N_k(S_3,X)$ under the same hypotheses.
\end{corollary}

Using the trivial bound on $h_2(k)$ in Corollary \ref{cor:general-bound}, the right-hand side becomes $O_{[k:\mathbb{Q}],\epsilon}(\Disc(k)^{4/3+\epsilon} X)$ for any $k$, and when $k$ is a $2$-extension, it follows from Lemma \ref{lem:two-part} that the right-hand side is $O_{[k:\mathbb{Q}],\epsilon}(\Disc(k)^{1+\epsilon} X)$. 

\begin{remark}
Theorem \ref{thm:cubic-bound} 
relies on an application of the Brauer--Siegel theorem to give a lower bound on the residue $\mathrm{Res}_{s=1} \zeta_k(s)$ of the Dedekind zeta function of $k$.  In particular, the implied constant
depends on $\epsilon$ ineffectively.  
However, 
the key results of Sections \ref{sec:small-cft}--\ref{sec:mid-order} from which Theorem \ref{thm:cubic-bound} is obtained, namely Propositions \ref{prop:trivial-cubic-bound}, \ref{prop:effective-shintani}, and \ref{prop:cubic-bound-order-general}, have effective constants and an explicit dependence on the residue, and could be used to prove an effective analogue of
Theorem \ref{thm:cubic-bound} with dependence on the residue. 
\end{remark}

In order to prove Theorem \ref{thm:cubic-bound}, we will use different approaches for three regimes of $X$ compared to $\Disc(k)$: the small range, the large range, and the intermediate range. In Section \ref{sec:small-cft}, for $X$ in the small range, we will use class field theory to give what we regard as a weak bound
 on the number of cubic extensions with a fixed discriminant. This will give a correspondingly weak bound on the number of cubic extensions with bounded discriminant. In Section \ref{sec:large-Shintani}, for $X$ in the large range, by using the functional equation of the {Shintani zeta function} for cubic rings, we can use the bound on the coefficients we obtained from class field theory and get improved bounds on the number of cubic fields with discriminants quite large compared to $\Disc(k)$. In Section \ref{sec:mid-order},  we will take advantage of the fact that the bound from the Shintani zeta function is actually also an upper bound for counting all cubic rings.  Since the number of cubic orders associated to each cubic field is large, we show it is impossible to get too many cubic fields with $X$ in the intermediate range.  Finally, in Section \ref{sec:effective}, we prove Theorem \ref{thm:cubic-bound} by combining the previous propositions for all ranges of $X$.

\subsection{Small range: class field theory}\label{sec:small-cft}
In this section we will use class field theory to give the following upper bound on the number of small degree extensions of $k$. 

\begin{proposition}\label{prop:trivial-cubic-bound}
	For any number field $k$ and any $X \geq 1$, we have 
	$$N_k(S_3, X)+ N_k(C_3, X)+ N_k(S_2, X) = O_{[k:\mathbb{Q}],\epsilon}(X^{3/2+\epsilon}\mathrm{Disc}(k)^{1/2+\epsilon} h_2(k)).$$
\end{proposition}

To prove this proposition, we apply class field theory and the trivial bound on torsion in relative class groups
to give the following pointwise  bound on the number of relative $S_3$ cubic extensions.

\begin{lemma}\label{lem:pointwise-bound}
	For any positive integer $n$ and any number field $k$, the number of extensions $K/k$ in $\bar{\Q}$ with $\mathrm{Disc}(K/k)=n$ and $[K:k]\leq 3$
	is 
	$O_{[k:\mathbb{Q}],\epsilon}(n^{1/2+\epsilon} h_2(k) \mathrm{Disc}(k)^{1/2+\epsilon}).$
\end{lemma}

\begin{proof}
	If $K/k$ is an $S_3$ cubic extension with $\mathrm{Disc}(K/k)=n$, then its associated quadratic resolvent $F/k$ has $\Disc(F/k) = m$ and $m|\Disc(K/k)$. By class field theory, quadratic extensions of $k$ are in bijection to surjective homomorphisms from the id\`ele class group $C_k = \mathbb{A}^{\times}_k/k^{\times}$ to $C_2$, where $\mathbb{A}^{\times}_k$ is the id\`ele group. We have the following exact sequence of $C_k$,
	$$1\rightarrow O_k^{\times} \rightarrow \prod_v O_{v}^{\times}\rightarrow C_k \rightarrow \Cl_k\rightarrow 1.$$
	Since $\Hom(\cdot, A)$ is left exact, we get
	$$1\rightarrow \Hom(\Cl_k, C_2 ) \rightarrow \Hom(C_k, C_2)  \rightarrow  \Hom( \prod_v O_{v}^{\times}, C_2).$$
	Here $|\Hom(\Cl_k, C_2)| = h_2(k)$. Given a $\rho\in \Hom(C_k, C_2)$, its image in $\Hom( \prod_v O_{v}^{\times}, C_2)$ determines the relative discriminant ideal of the quadratic field corresponding to $\rho$. Therefore for each fixed $m$, the number of possible $F/k$ with $\Disc(F/k)=m$ is bounded by $h_2(k)$. 
   For each fixed $n$, there are at most $O_{[k:\Q],\epsilon}(n^{\epsilon})$ possible $m$, therefore there are at most $O_{[k:\mathbb{Q}],\epsilon}(n^\epsilon h_2(k))$ quadratic extensions $F/k$ with $\mathrm{Disc}(F/k)|n$. 
   
   Similarly, for each fixed quadratic extension $F/k$, any associated $S_3$ cubic extension $K/k$ has Galois closure $\tilde{K}/k$ such that $\tilde{K}/F$ has $\Gal(\tilde{K}/F) = C_3$.  Thus, we can use class field theory again to determine that, for each fixed $F/k$, the number of associated $S_3$ cubic extensions with $\Disc(K/k)|n$ 
	is  
	$$\ll_{[k:\mathbb{Q}],\epsilon} n^{\epsilon} h_3(F/k) \ll_{[k:\mathbb{Q}],\epsilon} n^{\epsilon}\mathrm{Disc}(F/k)^{1/2+\epsilon}\mathrm{Disc}(k)^{1/2+\epsilon} \ll_{[k:\mathbb{Q}],\epsilon} n^{1/2+\epsilon} \mathrm{Disc}(k)^{1/2+\epsilon},$$ 
	by the trivial bound on $h_3(F/k)$ in Lemma \ref{lem:relative-trivial}. Altogether, the total number of $S_3$-extensions is now  $O_{[k:\mathbb{Q]},\epsilon}(n^{1/2+\epsilon} h_2(k) \mathrm{Disc}(k)^{1/2+\epsilon})$, as claimed. As above, we can bound the number of $C_3$-extensions $F/k$ by $O_{[k:\mathbb{Q}],\epsilon}(n^\epsilon h_3(k))$
	and the lemma follows.
\end{proof}

Summing Lemma \ref{lem:pointwise-bound} over integers $n$ immediately yields Proposition \ref{prop:trivial-cubic-bound}.

\subsection{Large range: the Shintani zeta function}\label{sec:large-Shintani}
In this section, we are going to give a better bound for $N_k(S_3, X)$ when $X$ is very large compared to $\Disc(k)$. The main idea is to use the bound 
in Lemma \ref{lem:pointwise-bound} to get a better estimate on $N_k(S_3, X)$ using \emph{Shintani zeta functions} over a general number field $k$. Moreover, the upper bound we obtain in this section is also an upper bound on the number of cubic rings over $k$, not just cubic fields.  This will be important in Section \ref{sec:mid-order}. 

Given a number field $k$, Shintani zeta functions over $k$ count \emph{cubic rings} over $k$ with a specified signature.  More specifically, given a cubic ring $R$ over $k$ (that is, a ring that is a locally free rank three $\mathcal{O}_k$-module), let $\disc(R/\mathcal{O}_k)$ denote the relative discriminant ideal of $R$ and $\mathrm{Disc}(R/\mathcal{O}_k) = \Nm_{k/\Q}(\disc(R/\mathcal{O}_k))$. 
For each archimedean place $v$ of $k$, let $k_v$ be the completion of $k$ at $v$, and let $R_v:= R\tensor_{\O_k} k_v$.
 Each $R_v$ is a cubic \'etale algebra over $k_v$.  If $k_v \simeq \mathbb{C}$, then there is only one such cubic \'etale algebra $\mathbb{C}^3$, while if $k_v \simeq \mathbb{R}$, then either $R_v \simeq \mathbb{R}^3$ or $R_v \simeq \mathbb{R}\times \mathbb{C}$.  The signature of $R$ is the tuple $\mathrm{sgn}(R) = (R_v)_{v\mid \infty}$ that records the isomorphism class of each $R_v$. 

We will use $N^3_k(X)$ to denote the number of isomorphism classes of cubic rings $K/k$ over a number field $k$ with $0<\Disc(K/k)\le X$, and $N^3_{k, \alpha}(X)$ to denote the number of isomorphism classes of cubic rings $K/k$ over a number field $k$ with $0<\Disc(K/k)\le X$ with signature $\alpha$. 
Then the main proposition we are going to prove for this subsection is as follows. 
\begin{proposition}\label{prop:effective-shintani}
For any number field $k$ 
and 	 $X \geq \mathrm{Disc}(k)^3 h_2(k)^{2/3}$, we have
	$$N^3_k(X) \ll_{[k:\mathbb{Q}],\epsilon} \mathrm{Disc}(k)^\epsilon X.$$
\end{proposition}

We start by introducing properties of Shintani zeta functions over a general number field, including all useful notations. For a fixed signature $\alpha$, let
\[
\xi_{k,\alpha}(s) := \sum_{\begin{subarray}{c} R/\mathcal{O}_k: \\ \,\mathrm{Disc}(R/\mathcal{O}_k) \neq 0, \\ \mathrm{sgn}(R) = \alpha \end{subarray}} \frac{|\mathrm{Aut}_{\O_k}(R)|^{-1}}{\mathrm{Disc}(R/\mathcal{O}_k)^s},
\]
where the sum runs over isomorphism classes of cubic rings $R$ over $k$.

Let $\zeta_k(s)$ be the Dekekind zeta function of $k$.
For convenience in what is to come, we define
\[
\mathfrak{A}_k := \frac{\zeta_k(2) \mathrm{Res}_{s=1}\zeta_k(s)}{2^{r_1(k)+r_2(k)+1}} \text{ and } \mathfrak{B}_k := \frac{3^{r_1(k)+r_2(k)/2}\zeta_k(1/3) \mathrm{Res}_{s=1} \zeta_k(s)}{6\cdot 2^{r_1(k)+r_2(k)} \mathrm{Disc}(k)^{1/2}} \left(\frac{\Gamma(1/3)^3}{2\pi}\right)^{[k:\mathbb{Q}]}.
\]
Additionally, for each pair of signatures $\alpha,\beta$ and each infinite place $v \mid \infty$ of $k$, we define a factor $c_{v,\alpha\beta}(s)$ as follows.  If $k_v \simeq \mathbb{C}$, set $c_{v,\alpha\beta}(s) = \sin^2(\pi s)\sin(\pi s - \pi/6) \sin(\pi s + \pi/6)$.  If $k_v \simeq \mathbb{R}$ and $\alpha_v \simeq \beta_v$, set $c_{v,\alpha\beta}(s) = \sin(2\pi s)/2$.  In the remaining cases, set
\begin{equation}
\begin{aligned}
c_{v,\alpha\beta}(s) = \frac{1}{2} \left\{\begin{array}{ll} 3\sin(\pi s), & \text{if } \alpha_v \simeq \mathbb{R}^3, \beta_v \simeq \mathbb{R}\times\mathbb{C}, \\ \sin(\pi s), & \text{if } \alpha_v \simeq \mathbb{R}\times\mathbb{C}, \beta_v \simeq \mathbb{R}^3, \end{array}\right.
\end{aligned}
\end{equation}
and define $c_{\alpha\beta}(s) := \prod_{v\mid \infty} c_{v,\alpha\beta}(s)$. 

The basic properties of $\xi_{k,\alpha}(s)$ are recorded in the following proposition that collects results due to Shintani \cite{Shintani}, Wright \cite{WrightThesis,WrightAdelic}, and Datskovsky and Wright \cite{DW86}.  We also refer the interested reader to the work of Taniguchi \cite{Taniguchi} for more information.

\begin{proposition} \label{prop:shintani} 
	We have the following properties for Shintani zeta functions over a general number field $k$: 
	\begin{enumerate}
		\item For each number field $k$ and each signature $\alpha$, the function $\xi_{k,\alpha}(s)$ converges absolutely in the region $\Re(s)>1$.
		
		\item The function $\xi_{k,\alpha}(s)$ has a meromorphic continuation to $\mathbb{C}$ with poles only at $s=1$ and $s=5/6$.  Each of these poles is simple, with residues
		\[
		\mathfrak{A}_k\cdot(1+3^{-r(\alpha)-r_2(k)}) \text{ and } \mathfrak{B}_k\cdot 3^{-r(\alpha)/2}
		\]
		respectively, where $r(\alpha) = \#\{ v \mid \infty : \alpha_v \simeq \mathbb{R}^3\}$.
		
		\item The function $\xi_{k,\alpha}(s)$ satisfies a functional equation,
		\[
		\xi_{k,\alpha}(1-s) = \left(\frac{3^{(6s-2)}}{\pi^{4s}}\Gamma(s)^{2}\Gamma(s-1/6)\Gamma(s+1/6)\right)^{[k:\mathbb{Q}]} \mathrm{Disc}(k)^{4s-2} \sum_{\beta} c_{\alpha\beta}(s) \hat\xi_{k,\beta}(s),
		\]
		where the sum runs over possible signatures $\beta$ and where $\hat\xi_{k,\beta}(s)$ is the {dual Shintani zeta function} defined via the Dirichlet series
		\[
		\hat\xi_{k,\beta}(s) = \sum_{\substack{R/\mathcal{O}_k: \\ \mathrm{Disc}(R/\mathcal{O}_k) \neq 0 \\ \mathrm{sgn}(R)=\beta \\ 3\mid \mathrm{tr}(t)\, \forall t \in R}} \frac{|\mathrm{Aut}_{\O_k}(R)|^{-1}}{\mathrm{Disc}(R/\mathcal{O}_k)^s},
		\]		
		where the sum runs over isomorphism classes of cubic rings $R$ over $k$.
		\item
		For each $\alpha$, the function $\hat\xi_{k,\alpha}(s)$ has a meromorphic continuation to all of $\mathbb{C}$ with poles at $s=1$ and $s=5/6$ (both simple), and satisfies the inequality $\hat\xi_{k,\alpha}(s) < \xi_{k,\alpha}(s)$ for real $s>1$.
		\item
		For each $\alpha$, the functions $(s-1)(s-5/6) \xi_{k,\alpha}(s)$ and $(s-1)(s-5/6)\hat\xi_{k,\alpha}(s)$ are entire of order $1$. 
	\end{enumerate}
\end{proposition}
\begin{proof}
	We provide references for these claims without regard to necessarily providing the original reference.  The first claim follows from \cite[Theorem 4.1]{WrightAdelic}.  The second through fourth primarily follow from \cite[Theorem 6.2]{DW86}, with the claim that $\hat\xi_{k,\alpha}(s) < \xi_{k,\alpha}(s)$ for real $s>1$ following from the absolute convergence of $\xi_{k,\alpha}(s)$ in this region.  The fifth follows from a generalization of \cite[Theorem 2.1.iii]{Shintani} as is explained on \cite[pg. 96]{WrightThesis}.
\end{proof}

We next record one more useful property of $\xi_{k,\alpha}(s)$.

\begin{lemma}\label{lem:shintani-vanishing}
	If $k$ has degree at least $2$, then $\xi_{k,\alpha}(0)=0$ for each signature $\alpha$.
\end{lemma}
\begin{proof}
	From the formulas given, it follows that each $c_{\alpha\beta}(s)$ has a zero of order $[k:\mathbb{Q}]$ at $s=1$.  The claim then follows from the functional equation for $\xi_{k,\alpha}(s)$.
\end{proof}

While Proposition \ref{prop:shintani} establishes that $\xi_{k,\alpha}(s)$ converges absolutely in the region $\Re(s)>1$, it does not guarantee that this convergence is uniform in $k$.  To control this uniformity, we must control the number of cubic rings over $k$ with small discriminant.  This will essentially be afforded to us by Proposition \ref{prop:trivial-cubic-bound} on the number of small discriminant cubic algebras over $k$, except that $\xi_{k,\alpha}(s)$ counts all orders in cubic \'etale algebras over $k$, not just the maximal orders.  We will later exploit this overcounting in the proof of Proposition \ref{prop:order-bound-sf}, but for now, we simply record the following result of Datskovsky and Wright \cite[Theorem 6.1]{DW86} in order to track the number of orders in a cubic \'etale algebra $A/k$. 

\begin{lemma}[Datskovsky--Wright]\label{lem:datskovsky-wright}
	Let $A$ be a cubic \'etale algebra over $k$ with maximal order $\mathcal{O}_A$.  Thus, either $A\simeq k^3$, $A \simeq F \times k$ where $F/k$ is a  quadratic  extension, or $A \simeq K$ where $K/k$ is a relative cubic field extension.  For each $n$, let $a_n$ denote the number of orders $\mathcal{O} \subset A$ for which $[\mathcal{O}_A:\mathcal{O}]^2 = n$, equivalently, the number of orders $\mathcal{O}\subset A$ for which $\Disc(\mathcal{O}/\mathcal{O}_k)  = \Disc(\mathcal{O}_A/\mathcal{O}_k) \cdot n$. 
	Then the Dirichlet series $f_A(s):=\sum_{n} a_n n^{-s}$ satisfies
	\[
	f_A(s)=\zeta_k(4s) \zeta_k(6s-1) \frac{\zeta_A(2s)}{\zeta_A(4s)},
	\]
	where
	\[
	\zeta_A(s) = 
	\begin{cases}
	\zeta_k(s)^3, & \text{if } A \simeq k^3, \\ 
	\zeta_k(s) \zeta_F(s), & \text{if } A \simeq F \times k,\\ 
	\zeta_K(s), & \text{if } A \simeq K,
	\end{cases}
	\]
	with notation  as above, and where for a number field $L$, $\zeta_L(s)$ denotes the Dedekind zeta function of $L$.
\end{lemma}

Combining Lemma \ref{lem:datskovsky-wright} and Lemma \ref{lem:pointwise-bound}, we derive the following lemma on the values of Shintani zeta function in a right half-plane.
\begin{lemma}\label{lem:shintani-bound}
	For any real $\sigma>3/2$ and any $t \in \mathbb{R}$, we have $\xi_{k,\alpha}(\sigma+it) \ll_{[k:\mathbb{Q}],\sigma,\epsilon} \mathrm{Disc}(k)^{1/2+\epsilon}h_2(k)$.
\end{lemma}
\begin{proof}
By the triangle inequality, it suffices to prove the lemma when $t=0$.  Assume $\sigma>3/2$ is given.
	With the notation of Lemma \ref{lem:datskovsky-wright}, for any cubic \'etale algebra $A$, we have $\zeta_A(\sigma) \leq \zeta_k(\sigma)^3$ whenever $\sigma > 1$. It follows that 
	\begin{equation}
	\begin{aligned}
	\xi_{k,\alpha}(\sigma) & \leq \sum_{[L:k]\leq 3}  \frac{1}{\mathrm{Disc}(L/k)^\sigma} \cdot \zeta_k(4\sigma) \zeta_k(6\sigma-1) \frac{\zeta_A(2\sigma)}{\zeta_A(4\sigma)} \\	
	&  \leq \zeta_k(2\sigma)^3\zeta_k(4\sigma)\zeta_k(6\sigma-1) \sum_{[L:k]\leq 3} \frac{1}{\mathrm{Disc}(L/k)^\sigma},
	\end{aligned}
	\end{equation}
	where the sum runs over isomorphism classes of field extensions  $L/k$ of degree at most $3$. The second inequality comes from $\zeta_A(4\sigma)\ge 1$, valid for $\sigma>1$. The factors of $\zeta_k(\sigma)$ may be bounded by suitable powers of the Riemann zeta function $\zeta(\sigma)$,
	depending only on the degree of $k/\mathbb{Q}$. 
	 So it suffices to understand the Dirichlet series of field discriminants.
By Lemma \ref{lem:pointwise-bound}, we then find that for any $\epsilon < \sigma - 3/2$, 
	\[
	\sum_{L/k} \frac{1}{\mathrm{Disc}(L/k)^\sigma} \ll_{[k:\mathbb{Q}],\epsilon}  h_2(k)\mathrm{Disc}(k)^{1/2+\epsilon} \sum_{n}\frac{n^{1/2+\epsilon}}{n^\sigma} \ll_{[k:\Q], \epsilon}  \mathrm{Disc}(k)^{1/2+\epsilon}h_2(k) \zeta(\sigma-1/2-\epsilon),
	\]
	yielding the lemma.
\end{proof}

Next, we apply the functional equation to deduce a discriminant-aspect ``convexity bound'' for the Shintani zeta function.

\begin{lemma}\label{lem:shintani-convexity}
	For any $-1/2 \leq \sigma \leq 3/2$ and any $t \in \mathbb{R}$, the Shintani zeta function satisfies
		\[
			\xi_{k,\alpha}(\sigma + it)
				=O_{[k:\mathbb{Q}],\epsilon}( h_2(k) \mathrm{Disc}(k)^{\frac{7}{2}-2\sigma+\epsilon} (1+|t|)^{2[k:\mathbb{Q}](\frac{3}{2}-\sigma)+\epsilon}).
		\]
\end{lemma}
\begin{proof}
	Let $\delta > 0$ be small.  For $\sigma = 3/2+\delta$, it follows from Lemma \ref{lem:shintani-bound} that
		\[
			\xi_k(\sigma+it)
				\ll_{[k:\mathbb{Q}],\delta,\epsilon} \Disc(k)^{1/2+\epsilon} h_2(k).
		\]
	For $s = \sigma + it$ with $\sigma = 3/2 + \delta$, we apply the functional equation and Stirling's formula to extract cancellation between the gamma factors and the
	$c_{\alpha\beta}(s)$ in order to find also
		\begin{align*}
			\xi_{k,\alpha}(1-s)
				&= \left(\frac{3^{(6s-2)}}{\pi^{4s}}\Gamma(s)^{2}\Gamma(s-1/6)\Gamma(s+1/6)\right)^{[k:\mathbb{Q}]} \mathrm{Disc}(k)^{4s-2} \sum_{\beta} c_{\alpha\beta}(s) \hat\xi_{k,\beta}(s) \\
				&\ll_{[k:\mathbb{Q}],\delta,\epsilon} (1+|t|)^{(4+4\delta)[k:\mathbb{Q}]+\epsilon} \mathrm{Disc}(k)^{9/2+4\delta+\epsilon} h_2(k).
		\end{align*}
	The general result follows upon using the Phragmen--Lindel\"of principle \cite[Theorem 5.53]{IwaniecKowalski} applied to the function $\frac{(s-1)(s-5/6)}{(s+1)^2} \xi_{k,\alpha}(s)$ and letting $\delta$ be sufficiently small.  Specifically, we find for any $-1/2- \delta < \sigma < 3/2 + \delta$ that
		\[
			\xi_{k,\alpha}(\sigma+it)
				\ll_{[k:\mathbb{Q}],\delta,\epsilon} (1+|t|)^{2[k:\mathbb{Q}](\frac{3}{2}+\delta-\sigma)+\epsilon} \Disc(k)^{\frac{7}{2}+2\delta-2\sigma+\epsilon}.
		\]
	Upon choosing $\delta = \epsilon$ and then replacing $\epsilon$ in the equation above by $\epsilon/(1+2[k:\mathbb{Q}])$, we obtain the required statement.
\end{proof}

Now we apply the functional equation in Proposition \ref{prop:shintani} and the above lemmas to bound the number of cubic rings $N^3_{k, \alpha}(X)$. 

\begin{proof}[Proof of Proposition \ref{prop:effective-shintani}]
	Let $\phi(x)$ be a positive smooth function satisfying $\phi(x)\geq 1$ whenever $0\leq x \leq 1$ and whose Mellin transform exhibits rapid decay away from the real axis; for example, we may take $\phi(x)=e^{1-x}$.  Letting $R$ run over cubic rings over $k$, we have
	\[
	N^3_{k,\alpha}(X) 
	= \sum_{\substack{ 0<\mathrm{Disc}(R/\mathcal{O}_k) \leq X \\ \mathrm{sgn}(R) = \alpha}} 1  
	\ll_\phi \sum_{\substack{0<\mathrm{Disc}(R/\mathcal{O}_k)  \\ \mathrm{sgn}(R) = \alpha}} \frac{\phi(\mathrm{Disc}(R/\mathcal{O}_k)/X)}{|\mathrm{Aut}_{\O_k}(R)|}.
	\]
	This latter term may be evaluated via Perron's formula and the full Shintani zeta function $\xi_{k,\alpha}(s)$,
	\[
	\sum_{\substack{0<\mathrm{Disc}(R/\mathcal{O}_k)  \\ \mathrm{sgn}(R) = \alpha}} \frac{\phi(\mathrm{Disc}(R/\mathcal{O}_k)/X)}{|\mathrm{Aut}_{\O_k}(R)|} = \frac{1}{2\pi i} \int_{2-i\infty}^{2+i\infty} \xi_{k,\alpha}(s) \Phi(s) X^s \,ds,
	\]
	where $\Phi(s)$ is the Mellin transform of $\phi(x)$.  By Lemma \ref{lem:shintani-convexity}, the exponential decay of $\Phi(s)$ off the real axis allows us to shift the contour to the line $\Re(s)=-1/2-\epsilon$ for any $\epsilon>0$.  Doing so, we must account for the poles of $\xi_{k,\alpha}(s)$ at $s=1$ and $s=5/6$ and of $\Phi(s)$ at $s=0$.  From Proposition \ref{prop:shintani}, it follows that the contribution from the pole at $s=1$ subsumes that of the pole at $s=5/6$.  Moreover, Lemma \ref{lem:shintani-vanishing} shows that in fact the integrand has no pole at $s=0$ owing to the zero of $\xi_{k,\alpha}(s)$ at $s=0$.  Altogether, the contribution from the poles is seen to be $O_{[k:\mathbb{Q}]}(X\cdot \mathrm{Res}_{s=1}\zeta_k(s))$.  
	
	Finally, appealing to the functional equation and Lemma \ref{lem:shintani-bound}, we find
	\begin{align*}
	&\frac{1}{2\pi i} \int_{-1/2-\epsilon-i\infty}^{-1/2-\epsilon+i\infty} \xi_{k,\alpha}(s) \Phi(s) X^s \,ds \\
	\quad
	&= \frac{1}{2\pi i} \int_{3/2+\epsilon-i\infty}^{3/2+\epsilon+i\infty} \left(\frac{3^{(6s-2)}}{\pi^{4s}}\Gamma(s)^{2}\Gamma(s-1/6)\Gamma(s+1/6)\right)^{[k:\mathbb{Q}]} \mathrm{Disc}(k)^{4s-2} \hat\xi(s) \Phi(1-s) X^{1-s} \,ds \\
	&\ll_{[k:\mathbb{Q}],\epsilon} \mathrm{Disc}(k)^{9/2+\epsilon} h_2(k) X^{-1/2-\epsilon},
	\end{align*}
	where $\hat\xi(s) = \sum_{\beta}c_{\alpha\beta}(s)\hat\xi_{K,\beta}(s)$.

We conclude that for any number field $k$ and any $X \geq 1$, we have
\begin{equation}\label{eqn:prop:shintani-cubic-bound}
	N^3_{k,\alpha}(X) \ll_{[k:\mathbb{Q}],\epsilon} X \cdot \mathrm{Res}_{s=1} \zeta_k(s) + X^{-1/2-\epsilon} h_2(k) \mathrm{Disc}(k)^{9/2+\epsilon}.
\end{equation}
Using the well-known upper bound $\mathrm{Res}_{s=1} \zeta_k(s) \ll_{[k:\mathbb{Q}],\epsilon} \mathrm{Disc}(k)^\epsilon$
due to Landau (see also \cite[Lemma 3]{Brauer}),  
we conclude the proposition.
\end{proof}

\subsection{Intermediate range: propagation of orders}\label{sec:mid-order}
In this section, our main proposition is the following bound on $N_k(S_3, X)$ that is better than both Proposition \ref{prop:trivial-cubic-bound} and \ref{prop:effective-shintani} when $X$ is a bit smaller than the region allowed by Proposition \ref{prop:effective-shintani}. 
	Let $\kappa_k:=\mathrm{Res}_{s=1}\zeta_k(s)$.

\begin{proposition}\label{prop:cubic-bound-order-general}
	Let $k$ a number field and let $X \geq 1$. We have
	\[
	N_k(S_3, X)
	\ll_{[k:\mathbb{Q}],\epsilon} \mathrm{Disc}(k)^{1/2+\epsilon} \kappa_k^{-2} X^{1/2} ( {\mathrm{Disc}(k) h_2(k)^{1/3}}+ {X^{1/2}  }).
	\]
\end{proposition}

The key idea is that Proposition \ref{prop:effective-shintani} actually gives an upper bound on the number of cubic rings over $k$ instead of cubic fields, and if there are too many cubic fields in this intermediate range, then the orders inside these fields would overrun these bounds.  We will first apply this idea to bound the average $3$-class number of quadratic extensions of $k$,
and then we apply class field theory to bound the number $N_k(S_3, X)$ of general $S_3$ cubic fields. 

We begin with the following lemma.

\begin{lemma}\label{lem:gen-series-order-F}
	Let $F/k$ be a quadratic extension of number fields, and let $\mathcal{R}(F)$ denote the set of isomorphism classes of 
	 cubic rings $R$ over $k$	
	 whose underlying cubic \'etale algebra $A:=R\tensor_\Q k$
	 has discriminant $\disc(F/k)$ and is either of the form $A \simeq k \times F$ or $A \simeq K$ for a non-Galois cubic extension $K/k$ with quadratic resolvent $F/k$.	
	Then we have
	\[
	\sum_{R \in \mathcal{R}(F)} \frac{|\mathrm{Aut}_{k}(A)|^{-1}}{\mathrm{Disc}(R/\mathcal{O}_k)^s}
		= \frac{h_3(F/k)}{2 \mathrm{Disc}(F/k)^{s}} \zeta_k(2s) \zeta_k(6s-1) \sum_{\substack{\mathfrak{a} \subseteq \mathcal{O}_F \,\mathrm{sq. free} \\ [\mathfrak{a}] \in 3\mathrm{Cl}_F+\Cl_k}} \frac{1}{\Nm(\mathfrak{a})^{2s}},
	\]
	where the summation on the right-hand side runs over those squarefree ideals of $\mathcal{O}_F$ whose classes are in the subgroup $3\Cl_F+\Cl_k \subseteq \Cl_F$.
\end{lemma}
\begin{proof}
	If $R \in \mathcal{R}(F)$, then $\mathrm{Disc}(R/\mathcal{O}_k) = \mathrm{Disc}(F/k) [\mathcal{O}_A:R]^2$.  Thus, with the notation of Lemma \ref{lem:datskovsky-wright}, we have
	\[
	\sum_{R \in \mathcal{R}(F)} \frac{|\mathrm{Aut}_{k}(A)|^{-1}}{\mathrm{Disc}(R/\mathcal{O}_k)^s}
		= \mathrm{Disc}(F/k)^{-s} \cdot\left(\frac{1}{2}f_{k\times F}(s) + \sum_{\substack{[K:k] =3 \\ \disc(K/k) = \disc(F/k) \\ F \subseteq \widetilde{K}/k}} f_K(s) \right).
	\]
Note that the summation on the right is restricted to those cubic extensions $K/k$ with quadratic resolvent $F$ and which satisfy $\disc(K/k) =\disc(F/k)$.  For each such $K/k$, by Lemma \ref{lem:datskovsky-wright} we have
	\begin{equation}
	f_K(s) 
		=\zeta_k(4s) \zeta_k(6s-1) \frac{\zeta_K(2s)}{\zeta_K{(4s)}} 	
		= \zeta_k(2s) \zeta_k(6s-1)\frac{L(2s,\chi_{\widetilde{K}/F})}{L(4s,\chi_{\widetilde{K}/F})},
	\end{equation}
	where $\widetilde{K}$ is the Galois closure of $K$ over $k$, and $\chi_{\widetilde{K}/F}$ is one of the non-trivial cyclic cubic characters of the class group $\mathcal{C}= \mathrm{Cl}_F/\mathrm{Cl}_k$ that cuts out $\widetilde{K}/F$.  
	The second equality comes from the relation $\zeta_K(s)/\zeta_k(s)= L(s, \chi_{\widetilde{K}/F})$. In fact, we have $L(s,\chi_{\widetilde{K}/F}) = L(s,\chi_{\widetilde{K}/F}^2) = L(s, \overline{\chi}_{\widetilde{K}/F})$, since all are equal to the irreducible degree $2$ Artin $L$-function of $K/k$.  Thus, we obtain also
	\[
	f_K(s)
		= \zeta_k(2s) \zeta_k(6s-1)\frac{L(2s,\chi_{\widetilde{K}/F})}{L(4s,\chi_{\widetilde{K}/F}^2)}.
	\]
	Meanwhile, the Dirichlet series for orders inside $F\times k$ is
	\begin{equation}
	f_{F\times k}(s) = \zeta_k(2s) \zeta_k(6s-1)\frac{\zeta_F(2s)}{\zeta_F(4s)}.
	\end{equation}
	Notice that $\zeta_F(s)$ is the $L$-function attached to the trivial character of $\mathcal{C}$.  Therefore, 
	\begin{align*}
	\sum_{R \in \mathcal{R}(F)} \frac{|\mathrm{Aut}_{k}(A)|^{-1}}{\mathrm{Disc}(R/\mathcal{O}_k)^s}
		& = \frac{ \zeta_k(2s) \zeta_k(6s-1)}{2\mathrm{Disc}(F/k)^{s}} \sum_{\chi \in \widehat{\mathcal{C}}[3]} \frac{L(2s,\chi_{\widetilde{K}/F})}{L(4s,\chi_{\widetilde{K}/F}^2)}  \\
		&= \frac{ \zeta_k(2s) \zeta_k(6s-1)}{2\mathrm{Disc}(F/k)^{s}} \sum_{\mathfrak{a}\subseteq\mathcal{O}_F\,\text{sq. free}} \frac{1}{\Nm(\mathfrak{a})^{2s}} \sum_{\chi \in \widehat{\mathcal{C}}[3]}\chi(\mathfrak{a})\\
		& =\frac{ h_3(F/k)\zeta_k(2s) \zeta_k(6s-1)}{2\mathrm{Disc}(F/k)^{s}} \sum_{\substack{\mathfrak{a}\subseteq \mathcal{O}_F\,\text{sq. free}\\ \mathfrak{a}\in 3\Cl_F + \Cl_k }} \frac{1}{\Nm(\mathfrak{a})^{2s}},
	\end{align*}
	by orthogonality of characters and the relation $|\widehat{\mathcal{C}}[3]| = h_3(F/k)$.
\end{proof}

We next have the following simple lemma.

\begin{lemma}\label{lem:ideal-lower-bound}
Let $k$ be a number field and let $\delta>0$.  There is an effectively computable constant $C$, depending at most on $\delta$ and $[k:\mathbb{Q}]$, such that for any $X \geq C \mathrm{Disc}(k)^{1/2+\delta} \kappa_k^{-1}$,
 we have
\[
\#\{\mathfrak{a} \subseteq \mathcal{O}_K \text{ ideal}: \Nm(\mathfrak{a}) \leq X\}
	\gg_{[k:\mathbb{Q}],\delta} \kappa_k X .
\]
\end{lemma}
\begin{proof}
Let $m \geq \frac{[k:\mathbb{Q}]}{2}+2$ be an integer.  Then
\begin{align*}
\#\{\mathfrak{a} \subseteq \mathcal{O}_K \text{ ideal}: \Nm(\mathfrak{a}) \leq X\}
	\geq \sum_{\Nm(\mathfrak{a}) \leq X} \left(1 - \frac{\Nm(\mathfrak{a})}{X}\right)^m
	= \frac{m!}{2\pi i} \int_{2-i\infty}^{2+i\infty} \zeta_k(s) X^s \frac{ds}{s(s+1)\dots(s+m)}.
\end{align*}
The convexity bound for $\zeta_k(s)$ (see \cite[Equation (5.20)]{IwaniecKowalski}, for example) implies that when $\sigma:=\Re(s)$ lies in the critical strip $0 \leq \sigma \leq 1$, we have
\begin{equation}\label{eqn:convexity-dedekind}
	(s-1) \zeta_k(s) \ll_{[k:\mathbb{Q}],\epsilon} \mathrm{Disc}(k)^{(1-\sigma)/2+\epsilon} (1+|s|)^{\frac{[k:\mathbb{Q}](1-\sigma)}{2}+1+\epsilon}.
\end{equation}
Thus, the integral above conveges absolutely when the line of integration is moved anywhere inside the critical strip.  Taking it arbitrarily close to the line $\Re(s)=0$, we conclude that
\[
\#\{\mathfrak{a} \subseteq \mathcal{O}_K \text{ ideal}: \Nm(\mathfrak{a}) \leq X\}
	\gg_{[k:\mathbb{Q}],\epsilon} X \cdot\mathrm{Res}_{s=1} \zeta_k(s) + O(X^\epsilon \mathrm{Disc}(k)^{\frac{1}{2}+\epsilon}).
\]
The result follows.
\end{proof}

We can now apply our count of orders to get an upper bound on the average $3$-class number of quadratic extensions.

\begin{proposition}\label{prop:order-bound-sf}
	Let $k$ a number field and let $X \geq 1$. We have
		\[
	\sum_{\substack{[F:k]=2 \\   \mathrm{Disc}(F/k) \leq X}} h_3(F/k)
		\ll_{[k:\mathbb{Q}],\epsilon} 	\mathrm{Disc}(k)^{1/2+\epsilon} \kappa_k^{-2} X^{1/2} ( {\mathrm{Disc}(k) h_2(k)^{1/3}}+ {X^{1/2}  }).
	\]

\end{proposition}
\begin{proof}
	Let $Y$ be given by
	\begin{equation}\label{eqn:Y-value}
		  Y := \max\left\{ { \mathrm{Disc}(k)^3 h_2(k)^{2/3}},\kappa_k^{-2} {X C^2\mathrm{Disc}(k)^{1+\epsilon}} \right\},
	\end{equation}
    where $C$ is the constant from Lemma \ref{lem:ideal-lower-bound}.
	For each quadratic extension $F/k$ with discriminant at most $X$, we consider the contribution to $N^3_k(Y)$ from the orders in the family $\mathcal{R}(F)$ from Lemma \ref{lem:gen-series-order-F}. 
 For any $R \in \mathcal{R}(F)$, any ring automorphism extends to an automorphism of the \'etale algebra $A = R \otimes_\mathbb{Q} k$.  Thus, $|\mathrm{Aut}_{\mathcal{O}_k} (R)| \leq |\mathrm{Aut}_k(A)|$, and we obtain an inequality of generating functions, 
		\[
			\sum_{R \in \mathcal{R}(F)} \frac{|\mathrm{Aut}_{\mathcal{O}_k} (R)|^{-1}}{\mathrm{Disc}(R/\mathcal{O}_k)^s}
				\geq \sum_{R \in \mathcal{R}(F)} \frac{|\mathrm{Aut}_{k}(A)|^{-1}}{\mathrm{Disc}(R/\mathcal{O}_k)^s},
		\]
	which is the Dirichlet series from Lemma \ref{lem:gen-series-order-F}.  Appealing to Lemma \ref{lem:gen-series-order-F}, this series is in turn bounded below by
	\[
	\frac{h_3(F/k)}{2 \mathrm{Disc}(F/k)^{s}} \zeta_k(2s) \zeta_k(6s-1)
		\geq \frac{h_3(F/k)}{2 \mathrm{Disc}(F/k)^s} \zeta_k(2s),
	\]
	since the summation over ideals $\mathfrak{a}$ has non-negative coefficients. It then follows from our assumption 
$Y\geq \kappa_k^{-2} {X C^2\mathrm{Disc}(k)^{1+\epsilon}} $	
	that we may appeal to Lemma \ref{lem:ideal-lower-bound} (with $X = Y^{1/2}/\mathrm{Disc}(F/k)^{1/2}$) to find that the contribution to $N^3_k(Y)$ of orders arising in $\mathcal{R}(F)$ is
	\[
		\gg_{[k:\mathbb{Q}],\epsilon} \frac{h_3(F/k) Y^{1/2} 
		\kappa_k
		}{\mathrm{Disc}(F/k)^{1/2}}.
	\]
	
	Thus, 
	\[
	N^3_k(Y)
		\gg_{[k:\mathbb{Q}],\epsilon} \sum_{\substack{[F:k]=2 \\ X/2 \leq \mathrm{Disc}(F/k) \leq X}} \frac{h_3(F/k) Y^{1/2} 
\kappa_k		
		}{\mathrm{Disc}(F/k)^{1/2}},
	\]
	and we conclude that
	\[
	\sum_{\substack{[F:k]=2 \\  X/2 \le \mathrm{Disc}(F/k) \leq X}} h_3(F/k)
		\ll_{[k:\mathbb{Q}],\epsilon} \frac{N^3_k(Y) X^{1/2} }{Y^{1/2} 
		\kappa_k}.
	\]
	Finally, by Proposition \ref{prop:effective-shintani}, we have 
$N^3_k(Y) \ll_{[k:\mathbb{Q}],\epsilon} \mathrm{Disc}(k)^{\epsilon/2} Y$ 	
	whenever 
		$
	Y\geq\mathrm{Disc}(k)^3h_2(k)^{2/3},$
	and the result follows 
	  using dyadic summation and the fact that $\kappa_k \ll_{[k:\mathbb{Q}],\epsilon} \mathrm{Disc}(k)^{\epsilon/2}$.
	\end{proof}

\begin{proof}[Proof of Proposition \ref{prop:cubic-bound-order-general}]
	Following the proof of Lemma $6.2$ in \cite{DW88}, for a square-free ideal $\mathfrak{q}$ of $k$ of norm $\Nm \mathfrak{q} \leq X^{1/2}$, the number of $S_3$ cubic fields $K/k$ with quadratic resolvent $F/k$ and $\disc(K/k) = \mathfrak{q}^2 \disc(F/k)$ is bounded by $O(4^{\omega(\mathfrak{q})} h_3(F/k))$, where $\omega(\mathfrak{q})$ denotes the number of prime ideals dividing $\mathfrak{q}$. 
	 Thus, for fixed $\mathfrak{q}$, the number of $S_3$-extensions $K/k$ with $\disc(K/k)\leq X$ whose discriminant is $\mathfrak{q}^2$ times that of their quadratic resolvent, may be bounded by a uniform constant times
		\[
		\sum_{\substack{[F:k] = 2 \\ \Disc(F/k) \leq X/\Nm(\mathfrak{q})^2}} 4^{\omega(\mathfrak{q})} h_3(F/k) 
			\ll_{[k:\mathbb{Q}],\epsilon}  4^{\omega(\mathfrak{q})} \mathrm{Disc}(k)^{1/2+\epsilon} \kappa_k^{-2} 
			 \left( \frac{\mathrm{Disc}(k) h_2(k)^{1/3}X^{1/2}}{\Nm(\mathfrak{q})} + \frac{X}{\Nm(\mathfrak{q})^{2}}  \right),
	\]
	the latter inequality provided by Proposition \ref{prop:order-bound-sf}.  
	
	When $X \geq \Disc(k)^2 h_2(k)^{2/3}$, the expression is parentheses above is bounded by twice the second term, and 
	adding the above inequality across $\mathfrak{q}$ with norm at most $X^{1/2}$, we have 
		\begin{equation*}
		N_k(S_3, X) \ll_{[k:\Q],\epsilon} \mathrm{Disc}(k)^{1/2+\epsilon} \kappa_k^{-2} X.
	\end{equation*}
Otherwise, 	$X \leq \Disc(k)^3$, and adding over $\mathfrak{q}$ have have
		\begin{equation*}
		N_k(S_3, X) \ll_{[k:\Q],\epsilon} \mathrm{Disc}(k)^{3/2+\epsilon} \kappa_k^{-2} h_2(k)^{1/3} X^{1/2+\epsilon} \ll_{[k:\Q],\epsilon} \mathrm{Disc}(k)^{3/2+4\epsilon} \kappa_k^{-2} h_2(k)^{1/3} X^{1/2},
	\end{equation*}	
	and the proposition follows.
\end{proof}
	
\subsection{Proof of Theorem \ref{thm:cubic-bound}}\label{sec:effective}

We can now combine our results in the various ranges to prove Theorem \ref{thm:cubic-bound}.
By  the bijection between $3$-torsion elements in relative class groups of quadratic extensions and $S_3$ relative cubic extensions with square-free discriminant ideal   \cite[Section 5]{DW88}, we have
	\begin{equation}\label{eq:h3-bijection}
\sum_{\substack{[F:k]=2\\\Disc(F/k)\leq X}} h_3(F_2/k) \leq  2N_k(S_3,X) + N_{k}(C_2,X).  
\end{equation}
		Using this, the first case  follows directly from Proposition \ref{prop:trivial-cubic-bound}.
			 The second and third cases follow from Propositions \ref{prop:order-bound-sf} and \ref{prop:cubic-bound-order-general}
	and the fact that $\mathrm{Res}_{s=1} \zeta_k(s) \gg_{[k:\Q],\epsilon} \mathrm{Disc}(k)^{-\epsilon}$
	 (though, this is only known to hold with an effectively computable implied constant when $\zeta_k(s)$ does not have an exceptional Landau--Siegel zero; see \cite{Stark}). 		 
		 The  fifth case follows from \eqref{eq:h3-bijection} and Proposition \ref{prop:effective-shintani}, and the fourth from the observation that $N_{k}(S_3, X) \leq N_k(S_3, D_k^3h^{2/3})$ for any $X \leq D_k^3h^{2/3}$ (and analogously for class numbers).

%%%%%%%%%%%%%%%%
%%%%%%%%%%%%%%%%
\section{Zero-density estimates for ray class $L$-functions}\label{sec:zero-density}
%%%%%%%%%%%%%%%%
%%%%%%%%%%%%%%%%

In this section, we collect the results we need on zero-density estimates for ray class $L$-functions and use them to show that most quadratic extensions will have sufficiently many small split primes.  This will allow us to apply our relative Ellenberg--Venkatesh Lemma~\ref{lem:relative-E-V} and get a non-trivial bound on $h_3(F/k)$ for most quadratic $F/k$.

Let $\chi$ be a nontrivial ray class character of a number field $k$ with conductor $\mathfrak{f}_\chi$.  The associated $L$-function $L(s,\chi)$ is defined by
\[
L(s,\chi) = \sum_{\mathfrak{a}} \frac{\chi(\mathfrak{a})}{\Nm(\mathfrak{a})^s},
\]
the series running over the integral ideals of $k$.  Let $r_1^+(\chi)$ denote the number of real places of $k$ that are split in the cyclic extension of $k$ cut out by $\chi$, let $r_1^-(\chi)$ denote the number of real places that are ramified in the extension, and let $r_2$ denote the number of complex places of $k$.  If we define
\[
\gamma_\chi(s) := \pi^{-[k:\mathbb{Q}]s/2} \Gamma\left(\frac{s}{2}\right)^{r_1^+(\chi)+r_2} \Gamma\left(\frac{s+1}{2}\right)^{r_1^-(\chi)+r_2},
\]
then the completed $L$-function $\Lambda(s,\chi) := (\Nm(\mathfrak{f}_\chi) \mathrm{Disc}(k))^{s/2} \gamma_\chi(s) L(s,\chi)$ 
is entire and satisfies a functional equation, $\Lambda(1-s,\chi) = \varepsilon(\chi) \Lambda(s,\bar\chi)$.  The poles of $\gamma_\chi(s)$ thus lend $L(s,\chi)$ its trivial zeros; these are the non-positive even integers with multiplicity $r_1^+(\chi)+r_2$, and the negative odd integers with multiplicity $r_1^-(\chi)+r_2$.

The following lemma summarizes the analytic properties we shall need.

\begin{lemma}\label{lem:analytic-ray}
Assume notation as above and set $q_\chi = \Nm(\mathfrak{f}_\chi) \mathrm{Disc}(k)$.  Then:
\begin{enumerate}
\item For any $t\geq 0$, the number of non-trivial zeros $\rho=\beta+i\gamma$ of $L(s,\chi)$ with $|\gamma-t| \leq 1$ is $O_{[k:\mathbb{Q}]}(\log((t+1) q_\chi))$.
\item For any $s$ with $-3/2 \leq \Re(s) \leq 2$, we have
\[
-\frac{L^\prime}{L}(s,\chi)
	= -\frac{r_1^+(\chi)+r_2}{s} - \frac{r_1^-(\chi)+r_2}{s+1}-\sum_{\substack{ \rho: \\ |s-\rho|<1}} \frac{1}{s-\rho} + O_{[k:\mathbb{Q}]}(\log((|s|+1) q_\chi)),
\]
where the sum runs over the nontrivial zeros $\rho$ of $L(s,\chi)$. 
\end{enumerate}
\end{lemma}
\begin{proof}
These are standard properties of $L$-functions.  For example, see Proposition 5.7 in Iwaniec and Kowalski \cite{IwaniecKowalski}.  The second conclusion is stated there only for $-1/2 \leq \Re(s) \leq 2$, but in the region $-3/2 \leq \Re(s) \leq -1/2$, the functional equation relates $-\frac{L^\prime}{L}(s,\chi)$ to $-\frac{L^\prime}{L}(1-s,\bar\chi)$, which is absolutely convergent.  Thus, the conclusion is easier in this case, with the factors arising from the functional equation being dominant.
\end{proof}

We next recall a consequence of Proposition A.2 in the appendix by Lemke Oliver and Thorner to work of Pasten \cite{Pasten}, which follows upon taking $\pi$ to be the trivial representation of $k$.  See also \cite[Theorem 1.2]{ThornerZaman-LargeSieve} for a somewhat stronger statement, which would allow also $\epsilon = 0$ below with an explicit value of $c$.

\begin{theorem}\label{thm:zero-density}
Let $k$ be a number field, and for any ray class character $\chi$ of $k$, let 
\[
N_\chi(\sigma,T) := \#\{ \rho : L(\rho,\chi) = 0, \Re(\rho) \in (\sigma,1), |\Im(\rho)| \leq T\}.
\]
Then there is a constant $c>0$, depending only on the degree $[k:\mathbb{Q}]$, such that for any $Q,T>1$, any $1/2 \leq \sigma <1$, and any $\epsilon > 0$, we have
\[
\sum_{\Nm(\mathfrak{q}) \leq Q} \,\sideset{}{^\prime}\sum_{\chi \pmod{\mathfrak{q}}} N_\chi(\sigma,T)
	\ll_{[k:\mathbb{Q}],\epsilon} (\mathrm{Disc}(k) QT)^{c(1-\sigma)+\epsilon}.
\]
Here, the outer summation runs over ideals $\mathfrak{q}$ of $\mathcal{O}_k$ with bounded norm and the inner summation runs over primitive ray class characters of conductor $\mathfrak{q}$.
\end{theorem}

Theorem \ref{thm:zero-density} implies that over a general number field $k$, most ray class $L$-functions do not have zeros near $s=1$. Therefore we will utilize Theorem \ref{thm:zero-density} to show that for most quadratic extensions over $k$ we can find enough prime ideals with a specified splitting behavior.  Recall that $\pi_k(Y)$ denotes the prime ideal counting function of $k$, and for an extension $F/k$, $\pi_k(Y;F,e)$ denotes the number of primes in $k$ of norm at most $Y$ that are split in $F$.

\begin{lemma}\label{lem:effective-Chebo-over-k-average}
	Let $k$ be a number field and let $\mathcal{F}_2(X)$  be the set of  quadratic extensions of $k$ with $\Disc(F/k)\le X$. Let $c = c_{[k:\Q]}$ be the constant as given in Theorem \ref{thm:zero-density}. For any $\epsilon_1>0$ 
	and $X\geq 2$,
	   there exists a set $\mathcal{E} = \mathcal{E}(k, X, 
	   \epsilon_1)\subset \mathcal{F}_2(X)$ of exceptional quadratic extensions $F/k$ such that:
		\begin{enumerate}
			\item for all $F \in \mathcal{F}_2(X)\setminus \mathcal{E}$, for all $4 \leq Y \le X$
there exists a constant $C_{[k:\Q],\epsilon_1}$ depending only on $[k:\Q],\epsilon_1$ such that
$$
\pi_k(Y; F, e) \geq \frac{1}{8} \pi_k(Y/2) - C_{[k:\Q],\epsilon_1}Y^{\sigma_1} \log^2( X \Disc(k)),
$$
	where $\sigma_1 = \max(1- \frac{\epsilon_1}{4c},1/2)$; and
			\item $|\mathcal{E}| \ll_{[k:\Q], \epsilon_1} \Disc(k)^{\epsilon_1}X^{\epsilon_1}$.
		\end{enumerate}
\end{lemma}
\begin{proof}
	For each quadratic extension $F/k$, there is a unique quadratic ray class character $\chi_{F/k}$ with conductor $\disc(F/k)$ such that $\frac{\zeta_F(s)}{\zeta_k(s)} = L(s, \chi_{F/k})$.
	As in the statement of the lemma, we set $\sigma_1 = 1-\epsilon_1/4c$
	and define 
		\[
			\mathcal{E}: = \{  F \in \mathcal{F}_2(X) : \exists \rho \text{ s.t. }  L(\rho, \chi_{F/k}) =0,  \Re(\rho) \in (\sigma_1, 1), \text{ and } |\Im(\rho)| \leq X^{1/2}
			 \}.
		\]
	Then Theorem \ref{thm:zero-density} implies $|\mathcal{E}| \ll_{[k:\Q], \epsilon_1/4} (\Disc(k) X^{3/2}
	)^{c(1-\sigma_1)+\epsilon_1/4}\ll_{[k:\Q],\epsilon_1} \Disc(k)^{\epsilon_1}X^{\epsilon_1}$. 
	
	Suppose $F \notin \mathcal{E}$.  In order to bound from below the number of primes of $k$ with $\chi_{F/k}(\mathfrak{p}) = 1$, we first observe that, away from the $O(\log \Disc(F))$ ramified primes, the function $(1+\chi_{F/k}(\mathfrak{p}))/2$ is precisely $1$ or $0$ according to whether $\mathfrak{p}$ is split or inert.  As is standard, we first consider a version weighted by the von Mangoldt function $\Lambda_k$ of $k$,  and we observe
		\begin{align*}
			\sum_{\Nm(\mathfrak{a}) \leq Y} \Lambda_k(\mathfrak{a})\, \frac{1+\chi_{F/k}(\mathfrak{a})}{2}
				&\geq \sum_{\Nm(\mathfrak{a}) \leq Y} \Lambda_k(\mathfrak{a})\, \frac{1+\chi_{F/k}(\mathfrak{a})}{2} \left(1 - \frac{\Nm(\mathfrak{a})}{Y}\right) \\
				&= \frac{1}{2} \sum_{\Nm(\mathfrak{a}) \leq Y} \Lambda_k(\mathfrak{a})\left(1 - \frac{\Nm(\mathfrak{a})}{Y}\right) + \frac{1}{2} \sum_{\Nm(\mathfrak{a}) \leq Y} \chi_{F/k}(\mathfrak{a})\Lambda_k(\mathfrak{a})\left(1 - \frac{\Nm(\mathfrak{a})}{Y}\right).
		\end{align*}
	The first term is related to $\pi_k(Y)$ by means of partial summation, so our goal is to show that the second is small.  
	
	Via Perron's formula, we find
	\begin{align*}
		\sum_{|\mathfrak{a}| \leq Y} \chi_{F/k}(\mathfrak{a})\Lambda_k(\mathfrak{a}) \left( 1-\frac{\Nm(\mathfrak{a})}{Y} \right)
		&= \frac{-1}{2\pi i} \int_{2-i\infty}^{2+i\infty} \frac{L^\prime}{L}(s,\chi_{F/k}) \cdot \frac{Y^s}{s(s+1)}\,ds.
	\end{align*}
     It follows from Lemma \ref{lem:analytic-ray} and the functional equation that the integral converges absolutely over any vertical line.  Thus, we may shift the contour all the way to the left, finding
		\[
			\sum_{\Nm(\mathfrak{a}) \leq Y}  \chi_{F/k}(\mathfrak{a}) \Lambda_k(\mathfrak{a})\left( 1 - \frac{\Nm(\mathfrak{a})}{Y}\right)
				=- \sum_{\substack{ \rho: \\ L(\rho,\chi)=0 \\ \rho\neq 0,-1}} \frac{Y^\rho}{\rho(\rho+1)} - \mathrm{Res}_0 - \mathrm{Res}_{-1},
		\]
	where the summation runs over all zeros of $L(s,\chi)$, nontrivial or trivial, and where $\mathrm{Res}_0$ and $\mathrm{Res}_{-1}$ denote the residues of the integrand at $s=0$ and $s=-1$, respectively.  
	The trivial zeros of $L(s,\chi_{F/k})$ occur at negative integers, and each has multiplicity at most $[k:\mathbb{Q}]$, so their total contribution to the sum may be bounded unifomly by $O_{[k:\mathbb{Q}]}(Y^{-2})$.  Using Lemma \ref{lem:analytic-ray}, the residue $\mathrm{Res}_{-1}$ is bounded by $O_{[k:\mathbb{Q}]}((\log Y)(\log X\mathrm{Disc}(k))Y^{-1})$, which is sufficient.  Again using Lemma \ref{lem:analytic-ray}, we see that
		\begin{align*}
			\mathrm{Res}_0
				&= -(r_1^+(\chi)+r_2) \log Y + \sum_{\substack{\rho: \\ |\rho| \leq 1}} \frac{1}{\rho} + O_{[k:\mathbb{Q}]}(\log (X \mathrm{Disc}(k))) \\
				&\ll_{[k:\mathbb{Q}]} \frac{\log (X \mathrm{Disc}(k))}{1-\sigma_1},
		\end{align*}
	since $F \not\in \mathcal{E}$ and $Y \leq X
	$.

By Lemma \ref{lem:analytic-ray}, the summation over zeros of height greater than $T = X^{1/2}
$ is bounded
\[
\sum_{\substack{ \rho: \\ L(\rho, \chi)=0 \\ |\Im(\rho)| \geq T}} \frac{Y^\rho}{\rho(\rho+1)}
\ll_{[k:\mathbb{Q}]} \frac{Y \log(X\Disc(k)T)}{T}  \ll_{[k:\Q]} Y^{1/2} \log(X\Disc(k)).
\]

Notice that by definition of $\mathcal{E}$, for $F \notin \mathcal{E}$ there are no zeros for $L(s, \chi_{F/k})$ with $\Re(\rho)>\sigma_1$ and $|\Im(\rho)| \leq T = X^{1/2}
$.  Using the functional equation, there are also no zeros with $\Re(\rho)\in (0,1-\sigma_1)$ and $|\Im(\rho)| \leq T$.  The summation over low lying zeroes of $L(s, \chi_{F/k})$ may therefore be bounded by
\[
\sum_{\substack{ \rho: \\ L(\rho,\chi)=0 \\ |\Im(\rho)| < T}} \frac{Y^\rho}{\rho(\rho+1)}
\ll_{[k:\mathbb{Q}]} \frac{Y^{\sigma_1} \log (X\Disc(k)T)}{1-\sigma_1} \ll_{[k:\Q],
\epsilon_1} Y^{\sigma_1}\log (X\Disc(k)).
\]

Finally, by the standard translation from the von Mangoldt function to the prime counting function, we can conclude the lemma.
\end{proof}

In order to apply Lemma \ref{lem:effective-Chebo-over-k-average} efficiently, we need the following effective lower bound on the number of prime ideals in a general number field $k$.
 \begin{lemma}[\cite{ZamThesis}]\label{lem:prime-number-theorem-k-dependence}
	There exist absolute, effective constants $\gamma$, $\beta$, and $D_0 >0$ such that if $k$ is a number field with $\Disc(k)\ge D_0$, then for $Y\ge \Disc(k)^{\beta}$, we have
	$$\pi_k(Y)  \gg_{[k:\Q]} \frac{1}{\Disc(k)^{\gamma}}\cdot  \frac{Y}{\log Y}.$$
\end{lemma}
An effective lower bound, weaker than that above, can be derived from \cite{LMO}, but Zaman \cite{ZamThesis} was the first to prove a bound of this strength; he did so with $\beta = 35$ and $\gamma = 19$.  It follows from \cite{ThornerZaman} that $\beta = 694$ and $\gamma = 5$ are also admissible, and from \cite{ZTChebo} that $\gamma = 1/[k:\Q]$ is admissible with an inexplicit value of $\beta$. 
%~ A result of this form was first essentially proved by \cite{LMO}, and has subsequently been made explicit in \cite{ZamThesis} with $\beta = 35$ and $\gamma = 19$, in \cite{ThornerZaman} with $\beta = 694$ and $\gamma = 5$, and  in \cite{ZTChebo} with $\beta$ inexplicit and $\gamma = 1/[k:\Q]$. 

%%%%%%%%%%%%%%%%%%%%%%%
%%%%%%%%%%%%%%%%%%%%%%%
%%%%%%%%%%%%%%%%%%%%%%%
\section{A Uniform Tail Estimate}\label{sec:tail}
%%%%%%%%%%%%%%%%%%%%%%%
%%%%%%%%%%%%%%%%%%%%%%%
%%%%%%%%%%%%%%%%%%%%%%%

A \emph{$2$-extension} is a $G$-extension for some $2$-group $G$.
In this section, our main goal is to prove the following theorem on a tail estimate for the $3$-torsion in class groups of $2$-extensions with a uniform dependence on the base field $k$. 
It will be the critical input for the proof of our main theorems in Section \ref{sec:main-theorem}.

\begin{theorem}\label{thm:tail}
Let $k$ be a number field.  For each $m\ge 2$, there exist constants $\delta_m>0$ and $\alpha_m$ depending at most on $m$ and the degree $[k:\Q]$ such that for any 
$X,Y>0$,
 we have 
	\begin{equation}
		\sum_{\substack{F_m/F_{m-1}/\cdots/ F_1/k \\ \Disc(F_{m-1})\ge Y \\ X/2\le \Disc(F_m)\le X}} h_3(F_m/k)  = O_{\epsilon, [k:\Q],m}\left(\Big(\frac{X}{Y^{1-\epsilon}} + X^{1-\delta_m} \Big) \cdot \Disc(k)^{\alpha_m} \right),
	\end{equation}	
	where the summation is over all towers $F_m/F_{m-1}/ \cdots F_1/F_0=k$
in $\bar{\Q}$	
	 with $F_{i+1}/F_{i}$ being  quadratic extensions for each $ 0\le i\le m-1$. 
\end{theorem}

We emphasize the point that Theorem \ref{thm:tail} is stated with respect to the absolute discriminant, while the results of Section \ref{sec:uniform-cubic} are with respect to the relative discriminant.
	  
\begin{remark}
	The statement of Theorem \ref{thm:tail} for $m=1$ is empty except when $Y \le \Disc(k)$.
	 If $Y \leq \Disc(k)$, the statement for $m=1$ is non-trivial, but  
	 follows from Corollary \ref{cor:general-bound}, and noting that it is stated with respect to \emph{relative} discriminant, we may take $\alpha_1$ to be any constant greater than $1/3$.
\end{remark}
\begin{remark}
	For our purposes, the actual values of $\delta_m$ and $\alpha_m$ are irrelevant; it suffices to know only that the constants $\delta_m$ and $\alpha_m$ exist.  However, in the course of the proof, we provide explicit but non-optimal values.  This will show that the value of $\delta_2$  may be chosen without dependence on $[k:\mathbb{Q}]$.
	We choose specific non-optimal values of many related exponents in the course of the proof, because we wish to avoid tracking further dependencies of implied constants that would come with avoiding such choices.
\end{remark}

Note that the tail estimate we prove in Theorem \ref{thm:tail}  handles all $2$-extensions due to the following easy lemma.
\begin{lemma}[\cite{KluWan}, Lemma 2.3]\label{lem:tower}
	Let $n=\ell^r$ be a prime power and $G\subset S_n$ be an $\ell$-group and $E/F$ be an extension of number fields with $\Gal(E/F)\cong G$. Then there exists a tower of fields
	\begin{equation}
		F=F_0 \subset F_1\subset \ldots \subset  F_{r-1} \subset F_r=E
	\end{equation}
	such that $\Gal(F_{i+1}/F_i)=C_\ell$ for all $0\leq i \leq r-1$.
\end{lemma}

We will prove Theorem~\ref{thm:tail} first for $m = 2$. Then treating $m=2$ as the base case, we will apply an inductive proof to prove Theorem~\ref{thm:tail} for general $m$.  For $m=2$, the theorem requires us to sum over all towers $F_2/F_1/F_0=k$ of relative quadratic extensions. We will separate the discussion depending on how large $\Disc(F_1)$ is. In Section \ref{ssec:m-2-non-critical}, we will consider the summation when $\Disc(F_1)$ is away from $X^{1/3}$ using results established in Section \ref{sec:uniform-cubic}. In Section \ref{ssec:initial-critical}, we will consider the summation when $\Disc(F_1)$ is near $X^{1/3}$ using results established in Section \ref{sec:arakelov} and \ref{sec:zero-density}.

Taking $Y=1$ in Theorem \ref{thm:tail}, we get the following immediate corollary, essentially an analogue of Corollary \ref{cor:general-bound}, which will be useful in the induction argument to come.

\begin{corollary} \label{cor:general-bound-m}
	Let $k$ be a number field.  For each $m \geq 2$, let $\alpha_m$ be the constant from Lemma \ref{thm:tail}.  Then for any $X \geq 1$, we have 
		\[
			\sum_{\substack{F_m/F_{m-1}/\cdots/ F_1/k \\ X/2\le \Disc(F_m)\le X}} h_3(F_m/k)  = O_{ [k:\Q],m}\Big ( X \Disc(k)^{\alpha_m} \Big ).
		\]
\end{corollary}

\subsection{Base Case: Non-critical Range}\label{ssec:m-2-non-critical}
In this subsection, we will consider the following summation:
	\begin{equation} \label{eqn:non-critical-def}
	\sum_{\substack{F_2/ F_1/k \\ \Disc(F_1) \in \mathfrak{N}(X,Y) \\ X/2\le \Disc(F_2)\le X}} h_3(F_2/k),
    \end{equation}	
where 
		$\mathfrak{N}(X,Y) 
			:= [Y, X^{1/3-\delta_0}) \cup (X^{1/3+\delta_0}, X^{1/2}]$
for an arbitrary small number $\delta_0$ satisfying $0<\delta_0<1/6$.  We regard this range for $F_1$ as non-critical because 
the results of Section~\ref{sec:uniform-cubic}
together with the following  lemma on the $2$-class number of 
relative $2$-extensions will give the desired bound (Lemma~\ref{lem:m-2-non-critical}) in this range.

\begin{lemma}[\cite{KluWan}, Theorem $2.7$]\label{lem:two-part}
	Let $k$ be a number field and let $F/k$ be a $2$-extension.  Then for any $\epsilon>0$, we have $h_2(F) \ll_{[F:\mathbb{Q}],\epsilon} \mathrm{Disc}(F)^\epsilon h_2(k)^{[F:k]}$.
\end{lemma}

\begin{lemma}\label{lem:m-2-non-critical}
	Let $k$ be a number field and let $\delta_0 \in (0, 1/6)$. There exists $\delta>0$, only depending on $\delta_0$, and an absolute constant $\alpha>0$ such that for all 
$X,Y>0$,	
	we have
	\begin{equation}
		\sum_{\substack{F_2/ F_1/k \\ \Disc(F_1) \in \mathfrak{N}(X,Y) \\ X/2\le \Disc(F_2)\le X}} h_3(F_2/k)  = O_{\epsilon, [k:\Q],\delta_0}\Big (\frac{X}{Y^{1-\epsilon}} + X^{1-\delta} \Big) \cdot \Disc(k)^{\alpha},
	\end{equation}	
	where the summation is over all towers $F_2/F_1/F_0=k$ with $F_2/F_1$ and $F_1/F_0$ being relative quadratic extensions and where $\mathfrak{N}(X,Y) = [Y, X^{1/3-\delta_0}) \cup (X^{1/3+\delta_0}, X^{1/2}]$.
\end{lemma}
\begin{proof}[Proof of Lemma \ref{lem:m-2-non-critical}]
  Writing $D := \Disc(F_1)$ and $h:=h_2(F_1)$, from Section~\ref{sec:uniform-cubic} (and the translation from relative to absolute discriminant), we have
	\begin{equation}\label{eqn:3-torsion-summation}
		\sum_{\substack{F_2/ F_1\\ X/2\le \Disc(F_2)\le X}} h_3(F_2/F_1) \ll_{[k:\Q],\epsilon} h \cdot D^{\epsilon} 		
    \begin{cases}
		X^{3/2} D^{-5/2}, & \text{for all $D^2\leq X \leq D^3$},  \\
		X^{1/2} D^{1/2} +X D^{-3/2}, & \text{for all $D^2\leq X$}, \\
		X D^{-2}, & \text{if } X \geq D^5 h^{2/3}.
	\end{cases}
	\end{equation}
Specifically, the first line follows from Proposition \ref{prop:trivial-cubic-bound} and \eqref{eq:h3-bijection},
 the second from Proposition \ref{prop:order-bound-sf}, 
 and the third from the final case of Theorem~\ref{thm:cubic-bound}.
 
We will use the first line of \eqref{eqn:3-torsion-summation} for $\Disc(F_1) \in (X^{1/3+\delta_0}, X^{1/2}]$, the second line for $\Disc(F_1) \in [Y_0, X^{1/3-\delta_0})$, 
%     with $Y_0 = X^{1/5}h_2(k)^{-4/15}$, 
      and the third line for $\Disc(F_1) \in [Y, Y_0)$, where we choose a value for $Y_0$ so that we can apply the third line to the range $[Y, Y_0)$ for every $F_1/k$. % using Lemma \ref{lem:two-part}.   
       In particular, if $C_{[K:\Q]}$ is such that we always have $h(K)^{2/3}\leq C_{[K:\Q]} \Disc(K)$, we can choose $Y_0=C_{2[k:\Q]}^{-1/6}X^{1/6}$.
Applying \eqref{eqn:3-torsion-summation} as described and applying Lemma \ref{lem:two-part} to bound $h$, we find
    \begin{equation}\label{eqn:non-critical-summation-m-2}
    		\begin{aligned}
    			\sum_{\substack{F_2/ F_1/k \\ \Disc(F_1) \in \mathfrak{N}(X,Y)\\ X/2\le \Disc(F_2)\le X}} h_3(F_2/k) 
					& \ll_{\epsilon',\epsilon, [k:\Q]}  h_2(k)^{2} \cdot X^{3/2} \sum_{\substack{F_1/k \\ \Disc(F_1) \in (X^{1/3+\delta_0}, X^{1/2})}} h_3(F_1/k)  \Disc(F_1)^{-5/2+\epsilon'}\\
					& + h_2(k)^{2} \cdot X^{1/2} \sum_{\substack{F_1/k \\ \Disc(F_1) \in [Y_0,X^{1/3-\delta_0})}} h_3(F_1/k)  \Disc(F_1)^{1/2+\epsilon'}\\
					& + h_2(k)^{2} \cdot X \sum_{\substack{F_1/k \\ \Disc(F_1) \in [Y_0,X^{1/3-\delta_0})}} h_3(F_1/k)  \Disc(F_1)^{-3/2+\epsilon'}\\
					& + h_2(k)^{2} \cdot X \sum_{\substack{F_1/k \\ \Disc(F_1) \in (Y,Y_0)}} h_3(F_1/k)  \Disc(F_1)^{-2+\epsilon}\\
					& \ll_{\epsilon',\epsilon, [k:\Q]}  \Disc(k)^{1/3+\epsilon} \cdot (X^{1-3\delta_0/2+\epsilon'} + X^{11/12+\epsilon'} + \frac{X}{Y^{1-\epsilon}}),
    		\end{aligned}
      \end{equation}
  where for the last inequality we have applied  partial summation and Corollary~\ref{cor:general-bound}, as well as the trivial bound for $h_2(k)$.
Then  we can choose, for example, $\delta=\min\{ \delta_0,1/15\}$, and $\alpha=2/5$. 
Above if we take $\epsilon'=\min\{\delta_0/2,1/60\}$, then the lemma follows for $0<\epsilon\leq \alpha- 1/3$, and hence for all $\epsilon> 0.$
\end{proof}

\subsection{Base Case: Critical Range}\label{ssec:initial-critical}
In this section, we will consider the following summation:
	\begin{equation}
	\sum_{\substack{F_2/ F_1/k \\ \Disc(F_1) \in \mathfrak{C}(X,Y) \\ X/2\le \Disc(F_2)\le X}} h_3(F_2/k),
\end{equation}	
where $\mathfrak{C}(X,Y) = \mathfrak{C}(X)= [X^{1/3-\delta_0}, X^{1/3+\delta_0}]$, which we regard as the critical range of $\Disc(F_1)$. 
In this range, summing the best bound for each $F_1$ is not sufficient, and we need to extract an additional saving from the sum over $F_1$. 
 We will apply the results developed in Sections \ref{sec:arakelov} and \ref{sec:zero-density} to obtain this saving. 

It will be convenient to use the following simple bound on the number of relative quadratic extensions. 

\begin{lemma}\label{lem:quadratic-bound}
	Let $k$ be a number field. Then
	$$N_k(C_2, X) = O_{[k:\Q], \epsilon}(h_2(k)\Disc(k)^{\epsilon} X).$$
\end{lemma}
\begin{proof}	
	By class field theory, the number
	 $N_k(C_2, X)$ is bounded by the product of $h_2(k)$ and the number of integral ideals with bounded norm. 
	 The latter can be counted by integrating $\zeta_k(s)$. 
	If $a_n$ denotes the number of integral ideals of norm $n$, then by Perron's formula we have
	\begin{equation}
		\sum_{n \leq X} a_n \leq \sum_{n=1}^\infty a_n e^{1-\frac{n}{X}} = \frac{e}{2\pi i} \int_{1+\epsilon-i\infty}^{1+\epsilon+i\infty} \zeta_k(s)\cdot \Gamma(s)\cdot X^s\,ds.
	\end{equation}
   Shifting the contour integral to $\Re(s)=1-\epsilon$ and using the convexity bound \eqref{eqn:convexity-dedekind} on $\zeta_k(s)$,  
	we get the upper bound 
	\begin{equation}
		\sum_{n \leq X} a_n = \Res(\zeta_k(s) \Gamma(s) X^s)_{s = 1} + O_{[k:\mathbb{Q}],\epsilon}(\Disc(k)^{\epsilon/2} X^{1-\epsilon}) = O_{[k:\mathbb{Q}],\epsilon}(\Disc(k)^{\epsilon} X),
	\end{equation}
    where the upper bound on the residue of the Dedekind zeta function comes from Landau.  
\end{proof}

\begin{lemma}\label{lem:m-2-critical}
	Let $k$ be a number field and let $\delta_0 \in (0, 1/100]$. There exists an absolute positive constant $\delta$, and a value $\alpha$, depending only on $[k:\Q]$, such that for all $X\ge 1$ we have 
	\begin{equation}
		\sum_{\substack{F_2/ F_1/k \\ \Disc(F_1) \in \mathfrak{C}(X) \\ X/2\le \Disc(F_2)\le X}} h_3(F_2/k)  = O_{[k:\Q] }\Big (X^{1-\delta} \Big) \cdot \Disc(k)^{\alpha},
	\end{equation}	
	where the summation is over all towers $F_2/F_1/F_0=k$ with $F_2/F_1$ and $F_1/F_0$ being relative quadratic extensions and $\mathfrak{C}(X) = [X^{1/3-\delta_0}, X^{1/3+\delta_0}]$.
\end{lemma}
\begin{proof}
We first give the idea of the proof in broad terms to motivate the notation that is to follow. 
 Exploiting the factorization $h_3(F_2/k) = h_3(F_2/F_1) h_3(F_1/k)$, Corollary \ref{cor:general-bound} provides strong control on the average of $h_3(F_1/k)$, while Lemma \ref{lem:relative-E-V} provides a non-trivial bound on $h_3(F_2/F_1)$ if there are ``enough'' small primes that split in $F_2/F_1$.  Lemma \ref{lem:prime-number-theorem-k-dependence} ensures there are enough small primes in $k$.  Applying Lemma \ref{lem:effective-Chebo-over-k-average} twice, we see first that ``most'' extensions $F_1/k$ have roughly the same number of small primes as $k$, and then second that in ``most'' extensions $F_2/F_1$ for such $F_1$, many of those small primes are split.  This lets us apply Lemma \ref{lem:relative-E-V}, and yields the lemma.

To make this precise, and in preparation for applying Lemma \ref{lem:effective-Chebo-over-k-average}, we let $\delta_1=1/20$ and $\epsilon_1=1/30$. 
We set $\sigma_1 = \max(1 - \epsilon_1/4c,1/2)$, where $c = c_{[k:\Q]}$ is the constant given in Theorem \ref{thm:zero-density}. 
Throughout the proof, we will make assumptions on $X$, including that $X$ is at least any absolute  constant necessary to apply the results we use, and then we will handle the remaining  values of $X$ at the end of the proof.  
 Applying Lemma \ref{lem:effective-Chebo-over-k-average} the first time, we obtain an exceptional set $\mathcal{E}_0 =\mathcal{E}(k, X^{1/3+\delta_0},  \epsilon_1)$.
  Then for each $F_1/k\notin \mathcal{E}_0$, we apply Lemma \ref{lem:effective-Chebo-over-k-average} a second time to obtain an exceptional set $\mathcal{E}(F_1) = \mathcal{E}(F_1, X/\Disc(F_1)^2,  \epsilon_1)$.  By Lemma \ref{lem:effective-Chebo-over-k-average}, both exceptional sets contain few elements: 
	\begin{equation} \label{eqn:exceptional-sparse}
		|\mathcal{E}_0|\ll_{[k:\Q]} \Disc(k)^{\epsilon_1} X^{\epsilon_1} \quad \text{and} \quad |\mathcal{E}(F_1)|\ll_{[k:\Q]} \Disc(k)^{\epsilon_1} X^{\epsilon_1}. 
	\end{equation}

We now consider the consequences of Lemma \ref{lem:effective-Chebo-over-k-average} for fields outside the exceptional sets.  For $F_1/k \notin \mathcal{E}_0$ with $\Disc(F_1) \in \mathfrak{C}(X)$ and $x =(X/(2\Disc(F_1)^2))^{\delta_1} \leq X^{1/3+\delta_0}$, by Lemma \ref{lem:effective-Chebo-over-k-average} we have, for $X$ sufficiently large, 
\begin{equation}\label{eqn:getprimes}
		\pi_{F_1}(x/2) \ge \pi_k(x/2;F_1, e) \geq \frac{1}{8}\pi_k(x/4)-C_{[k:\Q],\epsilon_1}(x/2)^{\sigma_1} \log^2( X^{1/3+\delta_0}\Disc(k)).
\end{equation}
We next assume that $X\ge   4^{\frac{1}{(1/3-2/100)\delta_1}} \Disc(k)^A$
 where $A =\frac{ \max \{ \beta, 8c(\gamma+2)/\epsilon_1 \} }{(1/3-2/100) \delta_1} = A([k:\Q])$, and $\gamma$ and $\beta$ 
 are  absolute constants for which Lemma \ref{lem:prime-number-theorem-k-dependence} holds; let also $D_0$ be an absolute constant such that Lemma \ref{lem:prime-number-theorem-k-dependence} holds for those $\gamma$ and $\beta$.
	If $\Disc(k) \geq D_0$, we use Lemma \ref{lem:prime-number-theorem-k-dependence}, 
	and for $X\ge \Disc(k)^A$  and $X$ sufficiently large given  $[k:\Q]$, we obtain
	\begin{equation} \label{eqn:X0-requirement}
		\frac{1}{8}\pi_k(x/4)-C_{[k:\Q],\epsilon_1}(x/2)^{\sigma_1} \log^2( X^{1/3+\delta_0}\Disc(k)) \gg_{[k:\Q]} \frac{x}{\Disc(k)^{\gamma} \log x}.
	\end{equation}
We require that $X$ is large enough 
so that \eqref{eqn:X0-requirement} holds also for the finitely many fields $k$ with $\Disc(k) \leq D_0$.  Under these assumptions, we conclude for $F_1 \notin \mathcal{E}_0$ that
\begin{equation}\label{eqn:halfway}
		\pi_{F_1}(x/2) \gg_{[k:\Q]} \frac{x}{\Disc(k)^{\gamma} \log x}.
\end{equation}
Next, for $F_2/F_1\notin \mathcal{E}(F_1)$, by Lemma \ref{lem:effective-Chebo-over-k-average} and \eqref{eqn:halfway}, and the assumptions above on $X$, we have
	\begin{equation}\label{eqn:non-exceptional-primes}
		\pi_{F_1}(x; F_2, e)  \gg_{[k:\Q]} \frac{x}{ \Disc(k)^{\gamma} \log x}. 
	\end{equation}

We now consider the contribution to the summation from the various kinds of exceptional and non-exceptional fields. First, we consider the summation $F_2/F_1/k$ where $F_1/k\in \mathcal{E}_0$. Using the trivial bound on the  class group 
and the sparsity
of exceptional fields along with Lemmas~\ref{lem:quadratic-bound} and \ref{lem:two-part}, we find
\begin{equation}
	\begin{aligned}
		\sum_{\substack{F_2/ F_1/k \\ \Disc(F_1) \in \mathfrak{C}(X)\\ X/2\le \Disc(F_2)\le X, F_1/k\in \mathcal{E}_0}} h_3(F_2/k) 
		& \ll_{[k:\Q],\epsilon}  X^{1/2+\epsilon} \sum_{\substack{F_1/k \in \mathcal{E}_0 \\ \Disc(F_1) \in \mathfrak{C}(X)}}  \sum_{\substack{F_2/ F_1 \\ X/2\le \Disc(F_2)\le X}} 1  \\
		& \ll_{[k:\Q],\epsilon} h_2(k)^2 \Disc(k)^{\epsilon_1} \cdot X^{5/6+2\delta_0+\epsilon_1+\epsilon} \\
		& \ll_{[k:\Q],\epsilon} \Disc(k)^{1+\epsilon} \cdot X^{5/6+2\delta_0+\epsilon_1+\epsilon}.
	\end{aligned}
\end{equation}

Second, we consider the contribution from those $F_2/F_1/k$ where $F_1 \notin \mathcal{E}_0$ but $F_2/F_1\in \mathcal{E}(F_1)$. Using the trivial bound on relative class group $h_3(F_2/F_1)$ and an average bound on $h_3(F_1/k)$,
we find
\begin{equation}
	\begin{aligned}
		\sum_{\substack{F_2/ F_1/k \\ \Disc(F_1) \in \mathfrak{C}(X)\\ X/2\le \Disc(F_2)\le X\\ F_2/F_1\in \mathcal{E}(F_1)}} h_3(F_2/k) & \ll_{[k:\Q],\epsilon}  \sum_{\substack{F_1/k \\ \Disc(F_1) \in \mathfrak{C}(X)}} h_3(F_1/k) \sum_{\substack{F_2/ F_1 \\ X/2 \leq \Disc(F_2) \leq X \\ F_2/F_1\in \mathcal{E}(F_1)}} \Disc(F_2/F_1)^{1/2+\epsilon} \Disc(F_1)^{1/2+\epsilon} \\
		& \ll_{[k:\Q],\epsilon}  \Disc(k)^{\epsilon_1} X^{1/2+\epsilon_1+\epsilon}  \sum_{\substack{F_1/k \\ \Disc(F_1) \in \mathfrak{C}(X)}} h_3(F_1/k) \Disc(F_1)^{-1/2+\epsilon}\\
		& \ll_{[k:\Q],\epsilon} \Disc(k)^{4/3+\epsilon_1+\epsilon} \cdot X^{2/3+\delta_0/2+\epsilon_1+\epsilon}. \\
	\end{aligned}
\end{equation}
Precisely, the first inequality follows from Lemma \ref{lem:relative-trivial}. The second inequality follows from 
the estimate $|\mathcal{E}(F_1)|\ll_{[k:\Q], \epsilon_1}  \Disc(k)^{\epsilon_1} X^{\epsilon_1}$ 
from \eqref{eqn:exceptional-sparse}.
The third inequality follows from Corollary \ref{cor:general-bound} and partial summation.

We now consider the contribution from those $F_2/F_1/k$ where neither $F_1/k \notin \mathcal{E}_0$ nor $F_2/F_1 \notin \mathcal{E}(F_1)$.  Using the relative Ellenberg--Venkatesh method developed in Section \ref{sec:arakelov} and existence of small split primes for non-exceptional fields developed in Section \ref{sec:zero-density} and made explicit in \eqref{eqn:non-exceptional-primes}, we find 
\begin{equation}
	\begin{aligned}
		&\sum_{\substack{F_2/ F_1/k \\ \Disc(F_1) \in \mathfrak{C}(X)\\ X/2\le \Disc(F_2)\le X \\ F_1 \notin \mathcal{E}_0, F_2/F_1\notin \mathcal{E}(F_1)}}  h_3(F_2/k) \\
			&\quad\quad \ll_{[k:\Q], \epsilon} \Disc(k)^{\gamma} \cdot \sum_{\substack{F_1/k \\ \Disc(F_1) \in \mathfrak{C}(X)}} h_3(F_1/k) \left(\sum_{\substack{F_2/ F_1 \\ F_2\notin \mathcal{E}(F_1)\\ X/2\le \Disc(F_2)\le X }} \Disc(F_2/F_1)^{1/2-\delta_1+\epsilon} \Disc(F_1)^{1/2+\epsilon}\right) \\
			&\quad\quad \ll_{[k:\Q], \epsilon}  \Disc(k)^{\gamma} \cdot h_2(k)^2 \cdot 
			\left(\sum_{\substack{F_1/k \\ \Disc(F_1) \in \mathfrak{C}(X)}} h_3(F_1/k) \cdot (\frac{X}{\Disc(F_1)^2})^{3/2-\delta_1+\epsilon}\cdot \Disc(F_1)^{1/2+\epsilon}\right) \\
			& \quad\quad \ll_{[k:\Q],\epsilon} \Disc(k)^{\gamma+7/3+\epsilon}  X^{1 - \frac{\delta_1}{3}+\delta_0(\frac{3}{2}-2\delta_1)+\epsilon}.\\
	\end{aligned}
\end{equation}
Precisely, the first inequality comes from Lemma \ref{lem:relative-E-V} and \eqref{eqn:non-exceptional-primes}.  
The second inequality follows from Lemma \ref{lem:quadratic-bound} and Lemma~\ref{lem:two-part}. 
We apply Corollary \ref{cor:general-bound} and partial summation for the last inequality. 

Finally, in the complementary case that $X$ is not sufficiently large, we may assume that $X \ll_{[k:\mathbb{Q}]} \Disc(k)^A$, so that
 by applying the trivial bound $h_3(F_2/k) \ll_{[k:\mathbb{Q}],\epsilon} X^{1/2+\epsilon}$, we find
\begin{equation}
	\begin{aligned}
		\sum_{\substack{F_2/ F_1/k \\ \Disc(F_1) \in \mathfrak{C}(X) \\ X/2\le \Disc(F_2)\le X}} h_3(F_2/k) 
			& \ll_{[k:\mathbb{Q}],\epsilon} X^{1/2+\epsilon} \cdot \sum_{\substack{F_2/ F_1/k \\ \Disc(F_1) \in \mathfrak{C}(X) \\ X/2\le \Disc(F_2)\le X}} 1 \\
			& \ll_{[k:\Q], \epsilon} X^{4/3+\delta_0 + \epsilon} h_2(k)^3 \Disc(k)^{-2+\epsilon} \\
			& \ll_{[k:\Q],\epsilon}   \Disc(k)^{(4/3+\delta_0+\epsilon)A-1/2+\epsilon}.
	\end{aligned}
\end{equation}	
 Here  
the second inequality follows from Lemma \ref{lem:quadratic-bound} twice, and the last inequality follows from the bound on $X$.

So to deduce the statement of the lemma, we note we have four upper bounds of the form $\Disc(k)^{a_i} X^{b_i+\epsilon}$.  We can take, e.g. $\epsilon=1/1000$, and note that each of the $b_i$ is less than $998/1000$.  Then we can take $\delta=1/1000$, and $\alpha$ can be taken to be $\max_i{a_i}$, which depends on $[k:\Q]$.
\end{proof}

\subsection{Induction}
Finally, in this section, we will prove Theorem \ref{thm:tail} in general. We will use an induction argument with the initial case $m=2$.

\begin{proof}[Proof of Theorem \ref{thm:tail}]
	The statement for $m = 2$ follows from the combination of Lemma \ref{lem:m-2-non-critical} and \ref{lem:m-2-critical}. 	
	Now assume for some $i$ that we can prove the lemma for every $m\le i$.  Our goal is to prove the lemma for $m = i+1$. 
	
	Firstly, we treat the summation when $\Disc(F_i) \in \mathfrak{N}(X,Y)=[Y, X^{1/3-\delta_0}) \cup (X^{1/3+\delta_0}, X^{1/2}]$, the non-critical range. We use an argument similar with that in Lemma \ref{lem:m-2-non-critical} where non-critical range is treated for $m=2$. For any $0< \delta_0 < 1/6$ and $Y \geq 1$, using the results on counting cubic fields in Section \ref{sec:uniform-cubic}, as collected in \eqref{eqn:3-torsion-summation}, and the induction hypothesis in the form of 
	 Corollary \ref{cor:general-bound-m}, we find
		\begin{equation}\label{eqn:induction-noncritical}
		\begin{aligned}
			\sum_{\substack{F_{i+1}/F_i/\cdots/F_1/k \\ \Disc(F_i) \in \mathfrak{N}(X,Y) \\ X/2\le \Disc(F_{i+1})\le X}} h_3(F_{i+1}/k)  
				&= \sum_{\substack{F_i/\cdots/F_1/k \\ \Disc(F_i) \in \mathfrak{N}(X,Y)}} h_3(F_i/k) \sum_{\substack{F_{i+1}/F_i \\ X/2\le \Disc(F_{i+1})\le X }} h_3(F_{i+1}/F_i) \\
				&\ll_{[k:\Q],i,\epsilon',\epsilon}  \Disc(k)^{\alpha_i-2^{i-1}+\epsilon} \Big(X^{1-\frac{3\delta_0}{2}+\epsilon'} + X^{11/12+\epsilon'}+\frac{X}{Y^{1-\epsilon}}\Big) .\\
		\end{aligned}
		\end{equation}
	The  inequality above follows as in \eqref{eqn:non-critical-summation-m-2}, with $F_i$ in place of $F_1$,
using the cases of \eqref{eqn:3-torsion-summation} in three ranges of $\Disc(F_i)$ followed by partial summation with  Corollary \ref{cor:general-bound-m}
(in place of Corollary~\ref{cor:general-bound}),
	 and the $h_2(k)^2$ factor from bounding $h_2(F_1)$ is replaced by $h_2(k)^{2^i}$ in the bound for $h_2(F_i)$.
		
	We now treat the contribution from towers with $\Disc(F_i)\in \mathfrak{C}(X) = [X^{1/3-\delta_0},X^{1/3+\delta_0}]$. We separate the discussion into two cases depending on the size of $\Disc(F_{i-1})$. Fix $1 \leq Y_2 \leq X^{\frac{1}{6}-\frac{\delta_0}{2}}$. Then, 
	we find 
	\begin{equation}\label{eqn:induction-1}
		\begin{aligned}
			\sum_{\substack{F_{i+1}/F_i/\cdots/F_1/k \\ \Disc(F_{i-1})\ge Y_2 \\\Disc(F_i) \in \mathfrak{C}(X) \\ X/2\le \Disc(F_{i+1})\le X}} h_3(F_{i+1}/k) 
				& = \sum_{\substack{F_i/\cdots/F_1/k \\ \Disc(F_{i-1})\ge Y_2 \\ \Disc(F_i) \in \mathfrak{C}(X,Y)}} h_3(F_{i}/k) \sum_{\substack{F_{i+1}/F_i\\  X/2\le \Disc(F_{i+1})\le X}} h_3(F_{i+1}/F_i)\\
				& \ll_{[k:\Q],i, \epsilon} h_2(k)^{2^i} \sum_{\substack{F_i/\cdots/F_1/k \\ \Disc(F_{i-1})\ge Y_2 \\ \Disc(F_i) \in \mathfrak{C}(X)}} h_3(F_{i}/k) \cdot \Disc(F_i)^{1/2+\epsilon} \cdot (\frac{X}{\Disc(F_{i})^2})^{3/2+\epsilon}\\
				& \ll_{[k:\Q], i,\epsilon} h_2(k)^{2^i} \Disc(k)^{\alpha_i}\cdot X^{3/2+\epsilon} \cdot \Big( \frac{ X^{-\frac{1}{2}+\frac{3\delta_0}{2}} }{Y_2^{1-\epsilon}} + X^{-(1/3-\delta_0)(3/2+\delta_i)}\Big)\\	
				& \ll_{[k:\Q], i,\epsilon} \Disc(k)^{\alpha_i+2^{i-1}+\epsilon} \cdot \Big( \frac{ X^{1+{3\delta_0}/{2}+\epsilon} }{Y_2^{1-\epsilon}} + X^{ 1+3\delta_0/2-\delta_i/3+\delta_0\delta_i+\epsilon  }\Big).
		\end{aligned}
	\end{equation}
	 Precisely, the first inequality follows from 
\eqref{eq:h3-bijection}, 	
	Proposition \ref{prop:trivial-cubic-bound},  Lemma~\ref{lem:quadratic-bound},and Lemma~\ref{lem:two-part}. The second inequality follows from partial summation, and the induction hypothesis with $m=i$ applied to $F_i/\cdots/F_1/k$ with $Y_2$ being the lower bound for $\Disc(F_{i-1})$. 
The third inequality follows from the trivial bound on $h_2(k)$. 
    
    Finally we consider those towers with $\Disc(F_{i-1}) \le Y_2$. 	Let $\delta'_2$ and $\alpha'_2$ be the constants associated to $[F_{i-1}:\Q]$ and $m =2$. 
We have
		\begin{equation}\label{eqn:induction-2}
		\begin{aligned}
			\sum_{\substack{F_{i+1}/F_i/\cdots/F_1/k \\ \Disc(F_{i-1})\le Y_2 \\\Disc(F_i) \in \mathfrak{C}(X) \\ X/2\le \Disc(F_{i+1})\le X}} h_3(F_{i+1}/k) 
				&= \sum_{\substack{F_{i-1}/\cdots/F_1/k \\ \Disc(F_{i-1})\le Y_2}} h_3(F_{i-1}/k) \sum_{\substack{F_{i+1}/F_i/F_{i-1}\\ \Disc(F_i) \in \mathfrak{C}(X)\\ X/2\le \Disc(F_{i+1})\le X}} h_3(F_{i+1}/F_{i-1})\\	
				& \ll_{ [k:\Q],i, \epsilon} \sum_{\substack{F_{i-1}/\cdots/F_1/k \\ \Disc(F_{i-1})\le Y_2}} h_3(F_{i-1}/k) \Big( X^{2/3+\delta_0+\epsilon} + X^{1-\delta'_2} \Big) \cdot \Disc(F_{i-1})^{\alpha'_2}  \\
				& \ll_{[k:\Q],i} \Disc(k)^{\alpha_{i-1}} X^{\max\{{3/4+\delta_0, 1-\delta'_2}\}} Y_2^{1+\alpha'_2}.\\
		\end{aligned}
	\end{equation}	
Here precisely, the first inequality follows from the induction hypothesis with $m=2$ and base field $k = F_{i-1}$. The second inequality follows from the induction hypothesis with $m=i-1$. 

If we let $Y_2=X^b$, then to obtain the desired bound, \eqref{eqn:induction-1} requires us to choose $\delta_0$ small, given $b$ and $\delta_i$,
and \eqref{eqn:induction-2} requires us to take $b$ small given $\delta_2'$ and $\alpha_2'$ (as long as we have taken $\delta_0$ absolutely sufficiently small). 
Since we can do this by choosing $b$ and then $\delta_0$, we will be able to obtain the theorem.  
In particular, making the convenient but non-optimal choices $b=\min\{\delta'_2/(2(1+\alpha'_2)),1/(8(1+\alpha'_2))\}$ and $Y_2 = X^{b}$ and $\delta_0 = \min\{\delta_i/9, b/2,1/16\}$, we conclude that the statement of the theorem holds for $m=i+1$ with  
	\[
		\alpha_{i+1} = \max\{\alpha_i+2^{i}, \alpha_{i-1}\} \quad \text{and} \quad \delta_{i+1} = \min\{\frac{\delta'_2}{2}, \frac{b}{8},\frac{\delta_i}{36},\frac{1}{32}\}. 
	\]
\end{proof}

%%%%%%%%%%%%%%%
%%%%%%%%%%%%%%%
%%%%%%%%%%%%%%%
\section{Average of $3$-torsion in Class Groups of $2$-extensions}\label{sec:main-theorem}
%%%%%%%%%%%%%%%
%%%%%%%%%%%%%%%
%%%%%%%%%%%%%%%

In this section, we will prove a refined version of Theorem~\ref{T:G}, in Theorem~\ref{thm:main} below, using crucially the tail estimates from Section \ref{sec:tail}.
We use the same approach to prove Theorem~\ref{T:rel}.
The 
  \emph{group signature} $\Sigma$ of $K/k$ is the ordered tuple $(\Sigma_v)_v$, where $v$ ranges over the real places of $k$, and $\Sigma_v$ is the conjugacy class of complex conjugation in $\Gal(K/k)$ over $v$.  This further refines the enriched signature of Cohen and Martinet.  
Let $E_k(m,X)$  
be the set of degree $2^m$ $2$-extensions $K/k$ (in $\bar{\Q}$) with $\Disc K\leq X$.
Let $E_k(G,X)$  
be the set of $G$-extensions $K/k$ (in $\bar{\Q}$) with $\Disc K\leq X$ and let $E_k^\Sigma(G,X)$  be those with group signature $\Sigma$.

\begin{theorem}\label{thm:main}
	Let $k$ be a number field and $m$ a positive integer. 
	Then there exists $C_m>0$ 
	such that
	\begin{equation}\label{eqn:m-average}
		\lim_{X\to\infty}\frac{1}{|E_k(m,X)|} {\sum_{K\in E_k(m,X)}\quad h_3(K)}=C_m.
	\end{equation}
	Moreover, let $G\subset S_{2^m}$ be a transitive permutation $2$-group containing a transposition and $\Sigma$ a group signature
	that occurs for some $G$-extension of $k$. 
	Then there exists $C_{G,\Sigma},C_G>0$  such that
\begin{equation} \label{eqn:G-average}
	\lim_{X\to\infty} \frac{1}{|E_k^\Sigma(G,X)|} {\sum_{K\in E_k^\Sigma(G,X)}\quad h_3(K)}= C_{G,\Sigma}
\quad \quad \textrm{and} \quad \quad	
\lim_{X\to\infty} \frac{1}{|E_k(G,X)|} {\sum_{K\in E_k(G,X)}\quad h_3(K)} = C_{G}.
\end{equation}
\end{theorem}

The proof of Theorem \ref{thm:main} gives the constants explicitly.
We let $r_1(F)$ denote the number of real places of $F$, and $r_2(F)$ the number of pairs of complex places of $F$.
We have
$$
C_m= \left( \sum_{F \in E_k(m-1,\infty)} 
 \frac{h_3(F) \mathrm{Res}_{s=1} \zeta_{F}(s) }{2^{r_2(F)}\zeta_{F}(2)\Disc(F)^2 } \cdot \left(1 + \frac{2^{r_1(F)}}{3^{r_1(F)+r_2(F)}}\right) \right) \left( 
 \sum_{\substack{F\in E_k(m-1,\infty)}} \frac{\mathrm{Res}_{s=1} \zeta_{F}(s) }{2^{r_2(F)} \zeta_{F}(2)\Disc(F)^2
		}
  \right)^{-1}.
$$
When the group $G$ has a transposition, Lemma \ref{lem:2-grp-with-transposition-is-wreath-product} below shows that $G=C_2\wr H$ for some transitive subgroup $H \subseteq S_{2^{m-1}}$.  
The image of a group signature $\Sigma$ for $G$ gives a group signature $\bar{\Sigma}$ for $H$.
So a $G$-extension $K/k$ of signature $\Sigma$ has an index $2$ subfield $F/k$ which is an $H$-extension of signature $\bar{\Sigma}$, and 
 there is a number $u(\Sigma)$, depending only on $\Sigma$, which gives the number of infinite places of $F$ that are split in $K$, or equivalently, $\rk\O_K^*-\rk \O_F^*$, \emph{the relative unit rank of $K/F$}. Then we have
$$
C_{G,\Sigma}=\left( 
\sum_{F \in E_k^{\bar{\Sigma}}(H,\infty)} 
 \frac{h_3(F) \mathrm{Res}_{s=1} \zeta_{F}(s) }{\zeta_{F}(2)\Disc(F)^2 } \cdot \left(1 +3^{-u(\Sigma)}\right)
\right)\left(\sum_{F \in E_k^{\bar{\Sigma}}(H,\infty)} 
 \frac{ \mathrm{Res}_{s=1} \zeta_{F}(s) }{\zeta_{F}(2)\Disc(F)^2 } \right)^{-1},
$$
$$
C_G =\left( \sum_{F \in E_k(H,\infty)} \frac{h_3(F) \mathrm{Res}_{s=1} \zeta_{F}(s) }{\zeta_{F}(2)\Disc(F)^2 } \cdot 
 \left(1 + \frac{2^{r_1(F)}}{3^{r_1(F)+r_2(F)}}\right) \right)\left( \sum_{F \in E_k(H,\infty)} 
 \frac{ \mathrm{Res}_{s=1} \zeta_{F}(s) }{\zeta_{F}(2)\Disc(F)^2 }	 \right)^{-1}.
$$

Before proving Theorem \ref{thm:main}, we first give a result on the asymptotic number of $2$-extensions, i.e. the denominator in Theorem \ref{thm:main}.

\begin{theorem}  \label{thm:counting-2-ext}
	Let $k$ be a number field and let $m$ be a positive integer. 
	 There exists $D_m>0$ such that
	\begin{equation}
		|E_k(m,X)| \sim  D_mX.
	\end{equation}
Moreover, let $G\subset S_{2^m}$ be a transitive permutation $2$-group containing a transposition and $\Sigma$ a group signature. 
Then there exists $D_{G,\Sigma}\geq 0$ and $D_G>0$ such that
	\begin{equation}
	|E_k^\Sigma(G,X)| \sim  D_{G,\Sigma} X \quad \textrm{and}\quad 	|E_k(G,X)| \sim  D_{G}X.
     \end{equation}
     If there exists an $H$-extension of $k$ with group signature $\bar{\Sigma}$, then $D_{G,\Sigma}>0$. 
Further, if $G\subset S_{2^m}$ is a transitive permutation $2$-group not containing a transposition, then $|E_k(G,X)| = O_{k,G,\epsilon}(X^{1/2+\epsilon})$.
\end{theorem}
\begin{proof}
	Malle \cite{Mal02,Mal04} conjectured that for $2$-groups $|E_k(G,X)|$ has linear growth if and only if $G$ contains a transposition. 
When $G$ has no transpositions, it is proved by \cite[Corollary 7.3]{KlunersMalle} that $|E_k(G,X)| = O_{k,G,\epsilon}(X^{1/2+\epsilon})$. When $G$ contains a transposition, by \cite[Lemma $5.5$]{Klu12} $G$ must be isomorphic to $C_2\wr H$ for some permutation $2$-group $H$. It is proved in \cite[Theorem $5.8$]{Klu12} that when there exists a $H$-extension and $|E_k(H,X)|$ is not growing too fast, Malle's conjecture $|E_k(G,X)| \sim D_G X$ holds for $G = C_2 \wr H$. The existence of $2$-group extensions is shown by a celebrated theorem of Shafarevich \cite{Shafa}, and the upper bound for $|E_k(H,X)|$ is known for all permutation $2$-groups $H$, again by \cite[Corollary 7.3]{KlunersMalle}.  
This gives the first statement of the theorem and a version of the second statement without signature conditions.  
Following Kl\"{u}ners proof, or a simpler version of the proof of Theorem~\ref{thm:main}, one can 
obtain a version with the signature condition and also with the precise constants
$$
		D_m =  \sum_{\substack{F\in E_k(m-1,\infty)}} \frac{\mathrm{Res}_{s=1} \zeta_{F}(s) }{2^{r_2(F)} \zeta_{F}(2)\Disc(F)^2
		},
$$
$$
		D_{G,\Sigma}	= \sum_{F \in E_k^{\bar{\Sigma}}(H,\infty)} 
 \frac{ p_\Sigma 2^{2^{m-1}-r_2(F)} \mathrm{Res}_{s=1} \zeta_{F}(s) }{\zeta_{F}(2)\Disc(F)^2 },
\quad \quad \textrm{and}\quad\quad
D_G= 
\sum_{F \in E_k(H,\infty)} 
 \frac{2^{2^{m-1}-r_2(F)} \mathrm{Res}_{s=1} \zeta_{F}(s) }{\zeta_{F}(2)\Disc(F)^2 }	 
 ,	
$$
where $\bar{\Sigma}$ is the group signature on $H$ given from $\Sigma$ via the quotient map $G\ra H$, and $p_\Sigma$ is the proportion of the 
$2^{r_1(F)}$ possible behaviors at infinity for a quadratic extension $K/F$ such that when $K$ is a $G$-extension with such behavior then it has group signature $\Sigma$.  (We note that $p_\Sigma$ is determined group theoretically from $\Sigma$ and is non-zero---see the proof of Theorem~\ref{thm:main}.)
\end{proof}

Theorem \ref{thm:counting-2-ext} shows that, among $2$-extensions of a given degree, the fields with Galois groups without transpositions are \emph{thin} families.  Thus, the restriction in Theorem \ref{thm:main} to groups $G$ with transpositions is natural.  In fact, while Theorem \ref{thm:main} does not give class number averages for the thin families of extensions with Galois group $G$ for $G$ without transpositions, a key step in its proof is to show that class group elements arising from such extensions 
give a negligible contribution to the average in \eqref{eqn:m-average}.
  This is a result of independent interest that complements Theorem \ref{thm:main}.  Thus, we begin by considering thin families in Section \ref{ssec:thin-2-group}, which culminates in Theorem \ref{thm:3-torsion-thin-2-ext} that makes this discussion precise.   Then, in Section \ref{ssec:proof-main-theorem}, we turn to the proof of Theorem~\ref{thm:main}.

\subsection{Thin families of $2$-extensions}\label{ssec:thin-2-group}
In this section, we focus on studying $3$-torsion in class groups of $2$-extensions with a permutation Galois group $G$ without a transposition.
 We begin by proving an upper bound on the count of these extensions, but with explicit base field discriminant dependence. We first give the following lemma, which is a uniform version of Theorem $1.6$ in \cite{KluWan} for $\ell=2$.

\begin{lemma}\label{lem:disc-multi-general}
Let $k$ be a number field. Then there exists $\alpha_m>0$ depending at most on $[k:\Q]$ and $m$ such that the number of degree $2^m$ $2$-extension $K/k$ with $\Disc(K/k) = D$ is bounded by $O_{[k:\Q], m, \epsilon}(D^{\epsilon} \Disc(k)^{\alpha_m})$.
\end{lemma}
\begin{proof}
	We proceed by induction over $m$. For $m = 1$, we know from class field theory (see the proof of Lemma~\ref{lem:pointwise-bound}) that the number of relative quadratic extension $K/k$ with $\Disc(K/k) = D$ is at most $O_{[k:\Q], \epsilon}(h_2(k)\cdot D^{\epsilon}) = O_{[k:\Q], \epsilon}( D^{\epsilon} \Disc(k)^{1/2+\epsilon})$. Now assuming the statement is true for $m =i$, we will prove that it also holds for $m = i+1$.  	By Lemma \ref{lem:tower}, it suffices to count quadratic extensions of degree $2^i$ $2$-extensions $F/k$ with $\Disc(F/k)^2| D$. The number of such $F$ is bounded by $O_{[k:\Q], \epsilon, i}( D^{\epsilon} \Disc(k)^{\alpha_i})$ by the induction hypothesis.  For each such $F$, the number of quadratic relative extension $K/F$ with $\Disc(K/F) = D/\Disc(F/k)^2$ is bounded by $O_{[F:\Q], \epsilon}(h_2(F) \cdot D^{\epsilon})$ from class field theory. Applying Lemma \ref{lem:two-part} and summation over all divisors of $D$, we obtain the upper bound $O_{[k:\Q], i+1, \epsilon}( D^{\epsilon} \Disc(k)^{\alpha_{i+1}})$, where $\alpha_{i+1}$ can be taken as $\alpha_{i}+ 2^{i-1}$. 
\end{proof}

\begin{theorem}\label{thm:counting-thin-2-extensions}
	Let $k$ be a number field $k$ and let $G \subseteq S_{2^m}$ be a transitive permutation $2$-group without a transposition.  
	Then there exists $\alpha>0$ depending at most on $[k:\Q]$ and $m$ such that
	$$|E_k(G, X)| = O_{[k:\Q],m,\epsilon}(X^{1/2+\epsilon} \Disc(k)^{\alpha}).$$
\end{theorem}
\begin{proof}
	Notice that when $G$ does not contain a transposition, the discriminant ideal $\disc(K/k)$ must have exponent at least $2$ at every ramified prime, i.e., it is powerful.  Thus, its norm, $\Disc(K/k)$, must be a powerful integer, and there are $O(X^{1/2})$ powerful integers below $X$.  The result now follows from Lemma \ref{lem:disc-multi-general}. 
\end{proof}

The $\ell$-torsion conjecture then would imply that the summation of $|\Cl_K[3]|$ for $G$-extensions $K$ without a transposition should be also thin, i.e. $O(X^{1/2+\epsilon})$.
By applying the trivial bound for class group $h_3(K)= O_{[K:\Q], \epsilon}(\Disc(K)^{1/2+\epsilon})$, we can conclude immediately that $\sum_{K\in E_k(G,X)} h_3(K)= O_{k, \epsilon}(X^{1+\epsilon})$.  However, it is crucial for our main theorem that we prove something slightly better than this. 

\begin{theorem}\label{thm:3-torsion-thin-2-ext}
Let $k$ be a number field and let $G\subset S_{2^m}$ be a transitive permutation $2$-group without a transposition. 
Then there exist $\delta>0$ and $\alpha$ depending on $G$ and $[k:\Q]$ such that
	$$\sum_{K\in E_k(G,X)} h_3(K/k)= O_{[k:\Q],G}(X^{1-\delta} \cdot \Disc(k)^{\alpha}).$$
\end{theorem}
\begin{proof}
	A $G$-extension $K$ can be constructed as a relative quadratic extension over a degree $2^{m-1}$ $2$-extension $F$ with $\Gal(F/k) =: H\subset S_{2^{m-1}}$ by Lemma \ref{lem:tower}. We fix $\delta_1=1/24$
	and $\epsilon_1=1/4$. 
	Let $\gamma$, $\beta$ and $D_0$ be absolute constants allowable in Lemma \ref{lem:prime-number-theorem-k-dependence}.  Let $Y = X^{\delta'}$ for 
	some $\delta'$ that we will choose sufficiently small in terms of $m$ and $[k:\Q]$.
Then by Theorem \ref{thm:tail} and dyadic summation, taking $\delta'\leq \delta_m$, we have
		$$
			\sum_{\substack{ K/F/k\\ \Disc(K)\le X \\ \Disc(F)\ge Y}}  h_3(K/k) \ll_{[k:\Q], m,\epsilon} \Disc(k)^{\alpha_m} X^{1-\delta'+\epsilon}.
		$$
	Therefore it suffices to consider the summation over $G$-extensions where the associated $H$-extension $F/k$ has $\Disc(F)\le Y = X^{\delta'}$. 
	
	We will apply an argument similar to Section \ref{ssec:initial-critical} to complete the proof. 
	  For each $F$, denote the set $\mathcal{E}(F) = \mathcal{E}(F, X/\Disc(F)^2, \epsilon_1)$ of exceptional quadratic extensions as given by Lemma \ref{lem:effective-Chebo-over-k-average}, for $X \gg 1$. The set has size $|\mathcal{E}(F)| \ll_{[F:\Q],\epsilon_1 } X^{\epsilon_1}$.  Let $x = (X/(2Y^2))^{\delta_1}$.  We assume $X$ is sufficiently large in terms of $[k:\mathbb{Q}]$
	  and $\delta'$ sufficiently small in terms of our absolute constants,
	 so that for $\sigma_1 = \max\{1- \epsilon_1/4c,1/2\}$ with $c = c_{[F:\Q]}$ in Theorem \ref{thm:zero-density} and $F$ with $\Disc(F) \geq D_0$, we may apply Lemmas \ref{lem:effective-Chebo-over-k-average} and \ref{lem:prime-number-theorem-k-dependence} in concert to conclude for $K\not\in \mathcal{E}(F)$,
		\begin{equation}\label{eqn:thin-primes}
			\pi_F(x;K, e) \geq \frac{1}{8}\pi_F(x/2)- C_{[F:\mathbb{Q}],\epsilon_1} x^{\sigma_1} (\log (X\Disc(F)))^2 \gg_{[k:\Q],\epsilon} \frac{x}{\Disc(F)^{\gamma}\log x}.
		\end{equation}
	We also assume that $X$ is sufficiently large 
	that this holds for the finite number of fields $F$ with $\Disc(F) \leq D_0$.
	
	Using the trivial bound for $h_3(K/k)$, the summation over $K/F/k$ with $K/F\in \mathcal{E}(F)$ is
	\begin{equation}\label{eqn:thin-2-extension-exceptional}
		\begin{aligned}
			\sum_{\substack{K/F/k \\ \Disc(F)\le Y\\\Disc(K)\le X \\K/F\in \mathcal{E}(F)}} h_3(K/k) & \ll_{[k:\Q],\epsilon}  \sum_{\Disc(F)\le Y}  \sum_{\substack{ K/F\in \mathcal{E}(F)}} \Disc(K)^{1/2+\epsilon}  \\
			& \ll_{[k:\Q],\epsilon_1,m, \epsilon} \Disc(k)^{\alpha_{m-1}} \cdot X^{1/2+\epsilon_1+ \epsilon} Y,
		\end{aligned}
	\end{equation}	
	where the last inequality follows from combining the bound on $|\mathcal{E}(F)|$ with an upper bound on the number of such $F$, for example as is provided by Corollary \ref{cor:general-bound-m} and dyadic summation. 
	
	For fields $K/F \notin \mathcal{E}(F)$, we use Lemma \ref{lem:relative-E-V}, \eqref{eqn:thin-primes}, and the trivial bound on $h_3(F/k)$ to obtain that $h_3(K/k) \ll_{[k:\mathbb{Q}],m,\epsilon} X^{1/2-\delta_1+\epsilon}Y^{2\delta_1+\gamma}$.  Then we find
		\begin{align*}
			\sum_{\substack{K/F/k \\  X/2\le \Disc(K)\le X \\  \Disc(F)\le Y\\K/F\notin \mathcal{E}(F)\\ \Gal(K/k) = G}}  h_3(K/k)
				&\ll_{[k:\Q],m,\epsilon} \sum_{\substack{K/k \\ X/2 \leq \Disc(K) \leq X \\  \Disc(F)\le Y\\ \mathrm{Gal}(K/k) = G}} X^{1/2-\delta_1+\epsilon}Y^{2\delta_1+\gamma} \\
				&\ll_{[k:\Q],m,\epsilon} \Disc(k)^{\alpha} X^{1-\delta_1+\epsilon}Y^{2\delta_1+\gamma},
		\end{align*}
	by Theorem \ref{thm:counting-thin-2-extensions}, where $\alpha$ is associated to $[k:\Q]$ in Theorem \ref{thm:counting-thin-2-extensions}. 
	We can use dyadic summation to remove the lower bound on $\Disc(K)$ and obtain the same bound (with a different implied constant).  
	
	Finally, for $X\ll_{[k:\Q],\delta'} 1$, we have
				\begin{align*}
			\sum_{\substack{K/F/k \\  X/2\le \Disc(K)\le X \\  \Disc(F)\le Y\\ \Gal(K/k) = G}}  h_3(K/k)
				&\ll_{[k:\Q],m,\delta'} \Disc(F)^{\alpha_m} \leq X^{\delta' \alpha_m}
		\end{align*}
by Corollary \ref{cor:general-bound-m}. Choosing $\delta'$ and sufficiently small in terms of $[k:\Q]$ and $m$, we conclude the theorem.
\end{proof}

\subsection{Proof of the Main Theorem}\label{ssec:proof-main-theorem}
Now we are ready to prove Theorem \ref{thm:main}. We start with the following lemmas that allow us to move from summing over $G$ extensions to summing over $H$-extensions, where $G\simeq C_2\wr H$.

\begin{lemma}\label{lem:2-grp-with-transposition-is-wreath-product}
	A transitive permutation $2$-group $G\subset S_{2d}$ contains a transposition if and only if $G\simeq C_2\wr H$ for some $H\subset S_{d}$. 	
For such a $G$, 	
if $K/k$ is a $G$-extension, then there exists a unique subfield $F/k$ with $[K:F] = 2$.  Moreover, $F/k$ is an $H$-extension. 
For such a $G$, there is only one permutation group $H$ up to isomorphism such that $G\simeq C_2\wr H$.
\end{lemma}
\begin{proof}
	The first statement is Lemma $5.5$ in \cite{Klu12}. For the second statement, we claim 
	if $M$ is a subgroup of $G$ and contains $\Stab_G(1)$ as a proper subgroup, then $C_2^d\rtimes \Stab_H(1)\subset M$.	
	 To show the claim, we write $\Stab_G(1) = C_2^{d-1}\rtimes \Stab_H(1)$ where $C_2^{d-1}$ is the subspace of $C_2^d$ with first entry $0$. Suppose $g=v\rtimes h\in M$ and $g\notin \Stab_G(1)$. If $h\in \Stab_H(1)$, then $v_1\neq 0$, and multiplication by all $u\rtimes h^{-1}\in \Stab_G(1)$ shows that $C_2^d \subset M$. 
	 If $h\notin \Stab_H(1)$, then for $u\in C_2^{d-1}$ with a $1$ in position $h^{-1}(1)$ and $0$'s elsewhere, we have that
$(v\rtimes h) u  (v\rtimes h)^{-1}\in C_2^d$ with a $1$ in position $1$ and $0$'s elsewhere, so we conclude the claim.
	The second and third statements of lemma then follow by Galois theory.
If $G \isom C_2\wr H$, then $H$ acts on $2$-element blocks of the elements that $G$ acts on, and all elements of $G$ preserve the block structure.  From this, it follows that two elements are together in a block if and only if $G$ contains a transposition that interchanges them.  From this it follows that $H$ is determined uniquely as a permutation group.
\end{proof}

We now want to be precise about exactly what we mean by permutation group isomorphisms.
\begin{definition}\label{D:permgroup}
A permutation group $G$ is a group $G$, a set $B_G$, and a faithful action of $G$ on $B_G$.  
The degree of $G$ is $|B_G|$.
An isomorphism of permutation groups $G\ra H$ is a bijection $\sigma:B_G \ra B_H$
so that the induced map $\sigma_*: G\ra H$ is a a group isomorphism (where $\sigma_*(g)(b)= \sigma g\sigma^{-1}(b)$).
We write $\Aut_{\operatorname{perm}}(G)$ for the group of permutation  automorphisms of $G$, which is also the normalizer $N_{S_{B_H}}(H)$ of $H$ in the symmetric group $S_{B_H}$. 
\end{definition}

In particular, a $G$-extension $K/k$ includes the data of a bijection between the embeddings $K\ra \bar{k}$ and $B_G$.  Note that $B_{C_2\wr H}=\{1,2\}\times B_H$.

\begin{lemma}\label{lem:KFbij}
Let $G$ be a transitive permuation $2$-group and $G=C_2\wr H$, with $H$ a permutation group of degree $d$.
Given an $H$-extension $F\in E_k(H,\infty)$, and a quadratic extension $K/F$ (in $\bar{\Q}$), then either (1)$\Gal(K/k)$ does not contain a transposition or
(2)$K/k$ can be realized as a $G$-extension with $2^d$ choices of 
$\Gal(K/k)\isom G$ 
(i.e. choices of $\{K \ra\tilde{K} \}\ra B_G$)
that are compatible with the permutation group isomorphism $\Gal(F/k)\isom H$ (i.e. choice of $\{F \ra\tilde{F} \}\ra B_H$)
in the quotient.  
Moreover, any $G$-extension $K\in E_k(G,\infty)$ arises in this construction from a unique $F$, one quadratic extension $K/F$, and one of the 
$2^d$ choices of $\Gal(K/k)\isom G$.
\end{lemma}
\begin{proof}
Given $F\in E_k(H,\infty)$ and a quadratic extension $K/F$, we have $2^d$ choices of identification of the embeddings $K\ra \tilde{K}$ with $\{1,2\}\times B_H$ that are compatible with
the map from the embeddings of $F$ to $B_H$.  Any of these gives an injection of groups $\Gal(K/k)\sub C_2\wr H$ compatible with the permutation actions.
Any transposition in $C_2\wr H$ is an element of $C_2^d$ that is non-trivial in exactly one coordinate.  Since $H$ is transitive, this implies that any subgroup of 
$C_2\wr H$ containing a transposition and with image in $H$ all of $H$
must contain all of $C_2^d$ and thus be all of $C_2\wr H$.
So if $\Gal(K/k)$ contains a transposition, we see that $\Gal(K/k)\isom G$, and there are $2^d$ ways of choosing such a permutation group isomorphism that are compatible with the map
from the embeddings of $F$ to $B_H$.

Given a $G$-extension $K\in E_k(G,\infty)$, from Lemma~\ref{lem:2-grp-with-transposition-is-wreath-product} we see that it 
arises in this way from a unique $F$, and the lemma follows.
\end{proof}

 Next we recall the following important result by Datskovsky--Wright \cite{DW88}. 

\begin{theorem}[Datskovsky--Wright]\label{thm:cubic-count} 
	Let $F$ be a number field.  For a group signature $\Sigma$ of a quadratic extension $K/F$, let $u(\Sigma)$ be the number of  infinite places 
	of $F$ where $\Sigma$ is trivial (which is also the relative unit rank of $K/F$). 
	Then
	\[
	\sum_{ \substack{ [K:F]=2\\
	\textrm{signature $\Sigma$} \\ |\mathrm{Disc}(K/F)| \leq X }} h_3(K/F)
	\sim X  \cdot 
\frac{\mathrm{Res}_{s=1}\zeta_F(s)}{2^{r_1(F)+r_2(F)}\zeta_F(2)}(1+3^{-u(\Sigma)})
\quad \textrm{and}  \sum_{ \substack{ [K:F]=2\\
	\textrm{signature $\Sigma$} \\ |\mathrm{Disc}(K/F)| \leq X }} 1
	\sim X  \cdot 
\frac{\mathrm{Res}_{s=1}\zeta_F(s)}{2^{r_1(F)+r_2(F)}\zeta_F(2)}.
	\]
\end{theorem}
\begin{proof}
	 This follows from combining Theorems 4.2 and 5.1 in \cite{DW88}.  
\end{proof}

We now turn to the proof of our main theorem.  

\begin{proof}[Proof of Theorem \ref{thm:main}]
Suppose $G \subseteq S_{2^m}$ is a $2$-group with a transposition, so by Lemma \ref{lem:2-grp-with-transposition-is-wreath-product}, $G = C_2 \wr H$ for a unique $H \subseteq S_{2^{m-1}}$.
Let $\bar{\Sigma}$ be the group signature for $H$ that is the image of $\Sigma$.
So a $G$-extension $K/k$ with group signature $\Sigma$ has an index two subfield $F/k$ with the group signature $\bar{\Sigma}$.
We next wish to show that the group signature of $K/k$ is determined by $F$ and the group signature of $K/F$.

We first consider the structure of order $2$ elements of $G$, to understand the possible group signatures.
For $g\in G=C_2 \wr H$, let $\bar{g}$ denote the image of $g$ in $H$.  
The group $H$ acts on a set of $2^{m-1}$ elements $B_H$.  The element $g\in G$ has order $2$ if and only if  $\bar{g}$  has order $2$ and the $C_2$ coordinates of $g$ (of which there is one for each element of $B_H$) are constant on $\bar{g}$ orbits of $B_H$.  
Moreover, if $\bar{g}$ transposes two elements $a,b$ of $B_H$, and $g'$ is obtained from $g$ by changing the $C_2$ coordinates at positions $a,b$, then $g'$ is conjugate to $g$.  In particular, a conjugacy class $[g]$ of an element $g\in G$ is determined by the image of the conjugacy class in $H$, and for any element $h$ of that conjugacy class in $H$, the $C_2$ coordinates of the $h$ fixed points of $B_H$. 

Next, we need to relate the infinite places of $F$ above a place $v$ of $k$ to the permutation group $H$.  To do this, at each real place $v$ of $k$ we choose an embedding $i_v: \bar{\Q}\ra \C$ that restricts to $v$ on $k$.  
If $F$ is an $H$-extension, the elements $B_H$ correspond to embeddings $\tau: F\ra \tilde{F}$ (as part of the data of the $H$-extension), and via $i_v:\bar{\Q}\ra \C$ these correspond to  embeddings $i_v\circ \tau: F\ra \C$.  So for each real place $v$ of $k$, we now have a correspondence between $B_H$ and the complex embeddings of $F$ that restrict to $i_v$ on $k$.
With these choices of $i_v$, we can now specify an actual element of $H$ corresponding to complex conjugation over $v$ (as opposed to a conjugacy class).  Precisely, 
we define $\sigma_v(F)$ to be the element of $H$ that acts on the embeddings $\tau: F\ra \tilde{F}$ so that $\sigma_v(F)(\tau)=i_v^{-1} \circ \overline{i_v\circ\tau}$, where the bar denotes complex conjugation, i.e. $\sigma_v(F)$ 
is the pullback of complex conjugation via $i_v$.
 Similarly, for a $G$-extension $K/k$ of group signature $\Sigma$, we have 
$\sigma_v(K) \in {\Sigma}_v$.
If $F/k$ is the index two subfield of $K$, then
at the fixed points of $\sigma_v(F)$, our chosen complex conjugation, on $B_H$, the $C_2$ coordinate of $\sigma_v(K)$ is trivial when $K/F$ is split at the corresponding place of $F$ and non-trivial when $K/F$ is ramified at that place.  
Using the conclusion of the previous paragraph, we conclude that if $K/k$ is a $G$-extension with index two subfield $F/k$,
the conjugacy class of $\sigma_v(K)$, and hence the group signature of $K/k$, is determined by $F$ and the group signature of $K/F$. 

Next, we wish to determine the number of group signatures of the quadratic extension $K/F$ that will lead to group signature $\Sigma$ for $K/k$.
Let $F/k$ be an $H$-extension with group signature $\bar{\Sigma}.$ 
 Let $M_F$ be the set of group signatures $\Sigma_2$ for quadratic extensions of $F$ such that   
$K/k$ has group signature $\Sigma$ when $K/F$ has group signature $\Sigma_2$.  
For a $G$-extension $K/k$ of group signature $\Sigma$, we let $u(\Sigma)$ be the difference in unit ranks between $K$ and $F$, which only depends on $\Sigma$.
  Note that for any $\Sigma_2\in M_F$, we have that $u(\Sigma)$ is the number of split places in $\Sigma_2$. 
  Also let $M_\Sigma:=|M_F|$, and note that it only depends on $\Sigma$, since  if we choose $\sigma_v\in \Sigma_v$ for all real places $v$ of $k$, then
$M_\Sigma$ is the product over $v$ of the number of elements $u\in C_2^{(B_H^{\sigma_v})}$  such that $u\sigma_v$ is conjugate to $\sigma_v$
(where $B_H^{\sigma_v}$ denotes the elements of $B_H$ fixed by $\sigma_v$).  
  Given an $H$-extension $F$, let $\mathcal{Q}_F^{\Sigma}$ be the set of quadratic extensions $K/F$ (in $\bar{\Q}$) of group signature in $M_F$.

Let $Y \geq 1$, and suppose $X_0 = X_0(Y)$ is sufficiently large so that: 1) $X_0\geq Y^2$; 2) and 
for any $X\geq X_0$ and each $F \in E_k(H,Y)$, we have
	\begin{equation} \label{eqn:X-hypothesis}
		\left|\sum_{ \substack{ K\in \mathcal{Q}_F^{\Sigma} \\ \mathrm{Disc}(K/F) \leq \frac{X}{\Disc(F)^2 }
		}} h_3(K/F)
			- \frac{X}{\Disc(F)^2} \cdot \frac{\mathrm{Res}_{s=1} \zeta_F(s)}{2^{r_1(F)+r_2(F)}\zeta_F(2)} \cdot \left(1 + 3^{-u(\Sigma)}\right) \right| \leq \frac{X}{2^{2^{m-1}} h_3(F) Y |E_k(H,Y)|}.
	\end{equation}
Since there are only finitely many fields $F \in E_k(H,Y)$, such an $X_0$ formally exists by Theorem \ref{thm:cubic-count}.  By Theorem \ref{thm:tail} and dyadic summation, we find for any $X \geq X_0$
	\begin{align*}
		\sum_{K \in E_k^{\Sigma}(G,X)} h_3(K)
			&= \sum_{\substack{ K \in  E_k^{\Sigma}(G,X) \\ \Disc(F) \leq Y }} h_3(K) + O_{k,m,\epsilon}(X/Y^{1-\epsilon} + X^{1-\delta_m}),
	\end{align*}
where $F \in E_k^{\bar{\Sigma}}(H,Y)$ is the unique index 2 subfield of $K$ containing $k$.
Then by
Lemma~\ref{lem:KFbij} and
 Theorem \ref{thm:3-torsion-thin-2-ext},  we find, for $X\geq X_0$, 
	\begin{align*}
		\sum_{\substack{ K \in  E_k^{\Sigma}(G,X) \\ \Disc(F) \leq Y }} h_3(K)
			&= 	
			\sum_{\substack{F \in E_k^{\bar{\Sigma}}(H,Y) }} 
			h_3(F)
			\sum_{\Sigma_2\in M_F} 2^{2^{m-1}}
			 \sum_{ \substack{  [K:F]=2\\\textrm{signature $\Sigma_2$} \\ \mathrm{Disc}(K/F) \leq X/\Disc(F)^2 
			 }} h_3(K/F) + O_k(X^{1 - \delta}). 
	\end{align*}
	Then by \eqref{eqn:X-hypothesis},	 for $X\geq X_0$, 
		\begin{align*}
		\sum_{\substack{ K \in  E_k^{\Sigma}(G,X) \\ \Disc(F) \leq Y }} h_3(K)
			&= X 			\sum_{\substack{F \in E_k^{\bar{\Sigma}}(H,Y) }} 
			 \frac{M_\Sigma 2^{2^{m-1}}h_3(F) \mathrm{Res}_{s=1} \zeta_{F}(s) }{2^{
		r_1(F)+
			r_2(F)} \zeta_{F}(2)\Disc(F)^2 
			} \cdot \left(1 + 3^{-u(\Sigma)}  \right) + O_k(X^{1-\delta}) + O(X/Y),
	\end{align*}
and thus also
	\[
		\frac{1}{X}\sum_{K \in E_k^{\Sigma}(G,X)} h_3(K)
			= \sum_{F \in E_k^{\bar{\Sigma}}(H,Y)} 
			\frac{		M_\Sigma 2^{2^{m-1}}	h_3(F) \mathrm{Res}_{s=1} \zeta_{F}(s) }{2^{
		r_1(F)+
			r_2(F)} \zeta_{F}(2)\Disc(F)^2 
			} \cdot \left(1 +3^{-u(\Sigma)}\right) + O_{k,m,\epsilon}(Y^{-1+\epsilon} + Y^{-\delta_m} + Y^{-\delta}),
	\]
since we assumed $X \geq X_0 \geq Y^2$.	Letting  $Y \to \infty$, the sum over $F \in E_k^{\bar{\Sigma}}(H,Y)$ converges by virtue of Theorem \ref{thm:counting-2-ext},  the trivial bounds $h_3(F) \ll_{[F:\mathbb{Q}],\epsilon} \Disc(F)^{1/2+\epsilon}$ and $\Res_{s=1}\zeta_F(s) \ll_{[F:\mathbb{Q}],\epsilon} \Disc(F)^{\epsilon}$, and partial summation.  
The error term on the right-hand side also converges, so the left-hand side must converge as well.  It follows that
	\[
		\lim_{X \to \infty} \frac{1}{X} \sum_{K \in E_k^{\Sigma}(G,X)} h_3(K)
			= \sum_{F \in E_k^{\bar{\Sigma}}(H,\infty)} 
 \frac{M_\Sigma 2^{2^{m-1}} h_3(F) \mathrm{Res}_{s=1} \zeta_{F}(s) }{2^{
		r_1(F)+
			r_2(F)}\zeta_{F}(2)\Disc(F)^2 } \cdot \left(1 +3^{-u(\Sigma)}\right).			
	\]
	Let $\Sigma_0$ be a group signature for an $H$-extension $F/k$.  Then 
$$
\sum_{\substack{\Sigma\\ \bar{\Sigma}=\Sigma_0}}  M_\Sigma =2^{r_1(F)}
\quad \textrm{and}\quad
\sum_{\substack{\Sigma\\ \bar{\Sigma}=\Sigma_0}}  M_\Sigma 3^{-u(\Sigma)} =\sum_{i=0}^{r_1(F)} \binom{r_1(F)}{i} 3^{-i-r_2(F)}=
\frac{2^{2r_1(F)}}{3^{r_1(F)+r_2(F)}}.
$$	
	where the sums are over group signatures $\Sigma$ for $G$-extensions, since the choices of $\Sigma$ with $i$ real places of $F$ split have $u(\Sigma)=i+r_2(F).$
		We can then sum over all group signatures to obtain
\begin{align*}
		\lim_{X \to \infty} \frac{1}{X} \sum_{K \in E_k(G,X)} h_3(K)
		&=\sum_{\bar{\Sigma}} 
	\sum_{\substack{\Sigma\\ \textrm{given }\bar{\Sigma}}}  
	\sum_{F \in E_k^{\bar{\Sigma}}(H,\infty)} 
 \frac{M_\Sigma 2^{2^{m-1}} h_3(F) \mathrm{Res}_{s=1} \zeta_{F}(s) }{2^{
		r_1(F)+
			r_2(F)}\zeta_{F}(2)\Disc(F)^2 } \cdot \left(1 +3^{-u(\Sigma)}\right)		
   \\		
   	&= \sum_{F \in E_k(H,\infty)} \frac{2^{2^{m-1}}h_3(F) \mathrm{Res}_{s=1} \zeta_{F}(s) }{2^{r_2(F)}\zeta_{F}(2)\Disc(F)^2 } \cdot 
 \left(1 + \frac{2^{r_1(F)}}{3^{r_1(F)+r_2(F)}}\right).
\end{align*}
	
To count each field $K$ exactly once (instead of once for each $G$ extension structure), we can divide the above sum by $|\Aut_{\operatorname{perm}}(G)|$.
We then sum over (isomorphism classes of ) $G$ of the form $C_2\wr H$, and by Lemma~\ref{lem:2-grp-with-transposition-is-wreath-product}
	 this will give a sum over $T_m$, the set of all (isomorphism classes of) transitive permutation $2$-groups $H$  of degree $ 2^{m-1}$.  Using Theorem~\ref{thm:3-torsion-thin-2-ext}, we obtain
	 \begin{align*}
&\lim_{X \to \infty}   \frac{1}{X} {\sum_{K \in E_k(m,X)} h_3(K)}\\
			&= \sum_{\substack{H\in T_m } }
\frac{1}{|\Aut_{\operatorname{perm}}(C_2\wr H)|}			
			 \sum_{F \in E_k(H,\infty)} \frac{2^{2^{m-1}}h_3(F) \mathrm{Res}_{s=1} \zeta_{F}(s) }{2^{r_2(F)}\zeta_{F}(2)\Disc(F)^2 } \cdot 
 \left(1 + \frac{2^{r_1(F)}}{3^{r_1(F)+r_2(F)}}\right).		 	 
\end{align*}
	 Above, each field $F$ appears in the sum $|\Aut_{\operatorname{perm}}(H)|$ times.
Next, we claim that the factors have value $|\Aut_{\operatorname{perm}}(H)|/ |\Aut_{\operatorname{perm}}(C_2\wr H)|=2^{-2^{m-1}}.$
The blocks of $B_G$ that $C_2\wr H$ acts on are the pairs that appear in $G$ as transpositions, so these are preserved by any permutation isomorphism of $C_2\wr H$, and we have a map 	 $\Aut_{\operatorname{perm}}(C_2\wr H)\ra \Aut_{\operatorname{perm}}(H)$.  It is easy to see this is a surjection and the kernel is given by all the permutations of $\{1, 2\}\times B_H$ that fix the $B_H$ coordinate, showing the claim.
	 Thus we have
	 $$
\lim_{X \to \infty} \frac{1}{X} {\sum_{K \in E_k(m,X)} h_3(K)}
			= 	
			 \sum_{F \in E_k(m-1,\infty)} 
 \frac{h_3(F) \mathrm{Res}_{s=1} \zeta_{F}(s) }{2^{r_2(F)}\zeta_{F}(2)\Disc(F)^2 } \cdot \left(1 + \frac{2^{r_1(F)}}{3^{r_1(F)+r_2(F)}}\right).		 	 
	 $$ 
	
Combined with Theorem \ref{thm:counting-2-ext}, this yields the theorem.
(Note we may obtain the constants in Theorem \ref{thm:counting-2-ext} by the same argument as above, using the count of quadratic extensions of $F$ in Theorem~\ref{thm:cubic-count} in place of the average $3$-torsion result.)
\end{proof}

\subsection{Relative Class Group Averages}

For the families  in Theorem \ref{thm:main}, the approach of the proof of Theorem \ref{thm:main} also permits us to determine the average size of the $3$-torsion subgroup of the relative class group $\Cl_{K/F}$,
where $F$ is the index two subfield of $K$, proving Theorem~\ref{T:rel} (when combined with Proposition~\ref{P:asCM}). 
These averages are particularly nice when the Galois closure group and
$u := r_1(K)+r_2(K)-r_1(F)-r_2(F)$
is fixed.  
For $G$ a transitive permutation $2$-group with a transposition, let $E_k^u(G,X)$ denote the set of $G$-extensions $K/k$ with $\Disc K\leq X$ and
$K/F$ of relative unit rank $u$, where $F/k$ is the unique index $2$ subfield of $K$.

\begin{theorem}\label{thm:relative-average}
	Let $k$ be a number field, $m$ a positive integer, and $G \subseteq S_{2^m}$ a transtive $2$-group with a transposition.  If $u$ is such that 
$E_k^u(G,\infty)$ is non-empty, then
		\[
			\lim_{X \to \infty} \frac{1}{|E_k^u(G,X)|} \sum_{K \in E_k^u(G,X)} h_3(K/F)
				= 1 + 3^{-u}.
		\]
\end{theorem}

\begin{proof}
	The proof uses the main input we have built for the proof of Theorem \ref{thm:main}
	and 
	is nearly identical to that proof, so we will be brief. 
Let $F/k$ be the index $2$ subfield of $K$.	
	 In particular, using the trivial inequality $h_3(K/F) \leq h_3(K/k)$, it follows from  Theorem \ref{thm:tail} that for any $X,Y>0$,
		\[
			\sum_{K \in E_k^u(G,X)} h_3(K/F)
				= \sum_{\substack{ K \in E_k^u(G,X) \\ \Disc(F) \leq Y}} h_3(K/F)
					+ O_{k,m,\epsilon}(X/Y^{1-\epsilon}+X^{1-\delta_m}).
		\]
			Again using the trivial inequality $h_3(K/F) \leq h_3(K/k)$, we see from Lemma~\ref{lem:KFbij}  and Theorem \ref{thm:3-torsion-thin-2-ext} that
		\[
			\sum_{\substack{ K \in E_k^u(G,X) \\ \Disc(F) \leq Y}} h_3(K/F)
				= \sum_{F \in E_k(H,X)} 2^{2^{m-1}} \sum_{\substack{ [K:F] = 2 \\ \Disc(K/F) \leq X / \Disc(F)^2 \\ \mathrm{rk}(\mathcal{O}_K^\times / \mathcal{O}_F^\times) = u}} h_3(K/F) + O_{m,k}(X^{1-\delta}),
		\]
	for some $\delta>0$ depending on $m$ and $k$. 
	Let $r$ be the number of group signatures of quadratic extensions of $F$ such that $K/F$ of that signature have  $\mathrm{rk}(\mathcal{O}_K^\times / \mathcal{O}_F^\times) = u$.
We take $X_0(Y)$ as in the proof of 	Theorem \ref{thm:main}, and have, for $X\geq X_0$
$$
			\sum_{\substack{ K \in E_k^u(G,X) \\ \Disc(F) \leq Y}} h_3(K/F)
				= \sum_{F \in E_k(H,X)} 
 \frac{r2^{
			2^{m-1}-r_1(F)-r_2(F)} \mathrm{Res}_{s=1} \zeta_{F}(s) }{ \zeta_{F}(2)\Disc(F)^2 
			} \cdot \left(1 + 3^{-u}  \right)				
				 + O_{m,k}(X^{1-\delta}+X/Y)
$$
Proceeding as in the proof of Theorem \ref{thm:main}, 
the result follows.
\end{proof}

\section{Comparison to the Cohen--Martinet heuristics}\label{sec:CM}

\subsection{Cohen--Martinet prediction for the average of $h_3(K/F)$}

Theorem \ref{thm:relative-average} verifies new cases of the Cohen--Martinet heuristics, as we now show.

\begin{proposition}\label{P:asCM}
The Cohen--Martinet heuristics \cite[Hypoth\'ese 6.6]{Cohen1990} predict that the average of Theorem~\ref{thm:relative-average} is as proved in the theorem.
\end{proposition}

\begin{proof}
 Let $G$ be a transitive permutation $2$-group with a transposition and $S_K$ the stabilizer of an element $1\in B_G$.  Our $K\in \mathcal{F}^u(G)$ is then the $S_K$ fixed field of a Galois $G$ extension $\tilde{K}/k$ with Galois group $G$.  By Lemma~\ref{lem:2-grp-with-transposition-is-wreath-product}, we have $G=C_2 \wr H$, where $H$ is a permutation group of degree $d$.
  In this notation, we have that $S_K=(1\times C_2^{d-1}) \rtimes \Stab_H(1)$, where $1\in B_H$ is the image of $1\in B_G$.  
 Then we define $S_F=(C_2^d) \rtimes \Stab_H(1)$ (the fixed field of $S_F$ will be the field we call $F$, the index $2$ subfield of $K$).
  Let $U$ be the sign representation of $S_F/S_K$ over $\F_3$.  

We claim 
$h_3(K/F)=|\Hom_{S_F} (\Cl_{\tilde{K}/k},U )|$, where $\Hom_{S_F}$ denotes morphisms of $S_F$ modules.
To see this, let $e\in \F_3[G]$ be the element $e=|S_K|^{-1}{\sum_{g\in S_K} g}$.  We have a map $m_e:\Cl_{\tilde{K}/k}/3 \ra \Cl_{K/k}/3$ given by multiplication by $e$, which is a surjection (as inclusion of ideals gives a section).   Since  $e$ acts as the identity on $U$, we have that any element of 
  $\Hom_{S_F} (\Cl_{\tilde{K}/k},U )$ factors through $m_e$, so is determined by a map $\Hom_{S_F} (\Cl_{K/k}/3,U ).$  Now $\Cl_{K/k}/3$ is a product of $\pm 1$ eigenspaces for $S_F/S_K$, and  $\Cl_{K/F}/3$ is exactly the $-1$ eigenspace, which proves the claim.

Let $V:= \Ind_{{S_F}}^{G} U$. We claim that $V$ is an irreducible representation of $G$.  If $h_i$ are coset representatives of $\Stab_H(1)$ in $H$, then
$1\rtimes h_i$ are coset representatives for $S_F$ in $G$.
Then $V$ has a basis $e_i:= (1\rtimes h_i) e_1$. The action of $G$ on $V$ is $(x\rtimes s) e_i = x_{s(i)} e_{s(i)}$. So on the one hand, it is permuting the $e_i$ from the permutation action of $s$, and on the other hand, it can change the sign for any particular basis independently.
Thus if $\sum c_i e_i$ is a non-zero element of $V$, we can act by some $g\in G$ to obtain $gx=e_1$, and thus we can generate all of $V$ from $G$ actions, which
shows that $V$ is irreducible.  

So, 
$$|\Hom_{S_F} (\Cl_{\tilde{K}/k},U )|=|\Hom_G(\Cl_{\tilde{K}/k} ,V)|=1+|\Sur_G(\Cl_{\tilde{K}/k} ,V)|.$$
From \cite[Theorem 6.2 and Theorem 4.1]{Wang2021}, we have that Cohen and Martinet predict the average of $|\Sur_G(\Cl_{\tilde{K}/k} ,V)|$ to be $\prod_{v}=|V^{\sigma_v}|^{-1}$,
where the product is over infinite places of the base field $k$, and $\sigma_v$ is a decomposition group for $v$.  (Note that we can subdivide our family, which only has a fixed relative unit rank of $K/F$, based on their decomposition groups at each of the infinite places of $k$.  We will see that the prediction for each of these subfamilies is the same, and only depends on the relative unit rank of $K/F$.)  The action of $\sigma_v$ on the elements that $H$ acts on, or equivalently the cosets $h_iS_F$, has fixed points corresponding to the infinite places of 
$F$ (over $v$) that are split over $k$ and $2$-cycles corresponding to infinite places of $F$ (over $v$) that are ramified over $k$.  We can write $\sigma_v$ as generated by $x_v \rtimes r_v$. 
If $h_i$ corresponds to a fixed point of $r_v$, then $\sigma_v \F_3 e_i=\F_3 e_i$, and the corresponding infinite place of $F$ is split in $K$ if and only if the $i$th coordinate of $x_v$ is trivial, which happens if and only if $\F_3 e_i$ is a trivial representation of $\sigma_v$.
If $h_i$ and $h_j$ correspond to a $2$-cycle of $r_v$, then $\F_3 e_i +\F_3 e_j$ is the regular representation of $\sigma_v$ and has one dimension of trivial representation.
Thus $\sum_v \dim_{\F_3} V^{\sigma_v}$ is the number of infinite places of $F$ that are not ramified in $K$, which is exactly the relative unit rank as defined above.
 \end{proof}
 
 \subsection{Cohen--Martinet prediction for the average of $h_3(K/k)$}\label{SS:wrongpred}

Using a similar approach, we can determine the Cohen--Martinet prediction for the average of $h_3(K/k)$ for $G$-extensions $K/k$ of a particular enriched or group signature.
In the case when $G=D_4$, we give below a numerical computation of the proven average for comparison.
  
 Let $G$ be a transitive permutation $2$-group with a transposition and $S_K$ the stabilizer of an element.  We consider the representation $\Ind_{S_K}^G \F_3$
of $G$ over $\F_3$ (which is the permutation representation of $G$).
Define $S_K$, $S_F$,  $U,$ and $V$ as in the proof of Proposition~\ref{P:asCM}.  Then $\Ind_{S_K}^G \F_3=\Ind_{S_F}^{G} \Ind_{S_K}^{S_F}  \F_3=\Ind_{S_F}^{G} \F_3 \times V.$  Let $W=\Ind_{S_F}^{G} \F_3$, which is the permutation representation of $H$.
We see that $W$ contains no copies of the irreducible representation $V$, since the action $G$ on 
$W$ factors through $H$ and the action of $G$ on $V$ does not.

Suppose that $W=\prod_i V_i^{a_i},$ where the $V_i$ are irreducible representations of $G$ (and note they all have $G$-action that factors through $H$).  
We note that by Frobenius reciprocity  $W$ contains exactly one copy of the trivial representation, and we let $W'$ be the quotient of of $W$ by this trivial representation.

We have 
$h_3(K/k)=|\Hom (\Cl_{K/k},\F_3 )|$.
Let $e,m_e$ be 
as in the proof of Proposition~\ref{P:asCM}.
Since any element of $\Hom_{S_K} (\Cl_{\tilde{K}},\F_3 )$ factors through $m_e \Cl_{\tilde{K}/k}=\Cl_{K/k}$, we have
a natural bijection between $\Hom (\Cl_{K/k},\F_3 )$ and $\Hom_{S_K} (\Cl_{\tilde{K}/k},\F_3 )|$.
Also $|\Hom_{S_K} (\Cl_{\tilde{K}/k},\F_3 )|=|\Hom_{G} (\Cl_{\tilde{K}/k},W \times V )|$.
If we let $e'=|G|^{-1}\sum_{g\in G} g$, we have that the map $m_e: \Cl_{\tilde{K}} \ra \Cl_k$ is a map whose kernel is exactly the relative class group
$\Cl_{\tilde{K}/k},$ and it is also the map that gives the maximal trivial representation quotient of any $G$ representation. 
 Thus $\Cl_{\tilde{K}/k}$ has no trivial representation part, and 
$h_3(K/k)=|\Hom_{G} (\Cl_{\tilde{K}/k},W '\times V )|.$

From \cite[Theorem 6.2, Theorem 4.1]{Wang2021}, for a $G$-representation $Z$ with no trivial component, we have that Cohen and Martinet predict the average of $|\Sur_G(\Cl_{\tilde{K}} ,Z)|$ to be $\prod_{v}=|Z^{\sigma_v}|^{-1}$,
where the product is over infinite places $v$ of the base field $k$, in a family of $G$-extensions where $\sigma_v\in G$ is an element of the conjugacy class of complex conjugation over $v$. 
Thus the predicted average of $h_3(K/k)$ is
\begin{align*}
&\left(1+\prod_{v} |V^{\sigma_v}|^{-1}\right) \prod_i \left(
1+\prod_{v} |V_i^{\sigma_v}|^{-1} +\cdots + \prod_{v} |V_i^{\sigma_v}|^{-a_i}
 \right),
\end{align*}
where the left product is the predicted average of $h_3(K/F)$ and the right product is the predicted average of $h_3(F/k)$ (as can be worked out similarly to the above) for families with the corresponding behavior at infinite places.

 Let $F/k$ be the index two subfield of $K$.  
The action of $\sigma_v$ on the permutation basis elements of $W$ has fixed points corresponding to the split places of $F$ (over $v$) and $2$-cycles corresponding to the ramified places of $F$ (over $v$).  Thus $\prod_v |W^{\sigma_v}|=3^{r_1(F)+r_2(F)}$ and 
$\prod_v |(W')^{\sigma_v}|=3^{r_1(F)+r_2(F)-r_1(k)-r_2(k)}$.
If $W'$ is irreducible, this gives a nice formula for the predicted average of $h_3(K/k)$, which is $(1+3^{-u(K/F)} )(1+3^{-u(F/k)} )$,
where  $u(K/F)$ is the difference in unit ranks of $K$ and $F$ in the given family (and similarly for $u(F/k)$) .  However, when $W$ is not irreducible, the formula can be more complicated.

\begin{example}
Let $H=C_2$, so $G=D_4=\langle (1234), (24)\rangle$ and $S_K=\langle (24)\rangle$ and $S_F=\langle (13),(24)\rangle$.
Here $W'$ is irreducible, the sign representation of $H$, and thus the  average conjecture by Cohen and Martinet is    $(1+3^{-u(K/F)} )(1+3^{-u(F/k)} )$.
For example, if $k=\Q$, in the following table we give the averages of $h_3(K)$ predicted by the Cohen--Martinet heuristics, and a numerical computation of the proven averages from Theorem~\ref{thm:main}, in families with the given group signatures. 
 \begin{center}
\begin{tabular}{ c|c|c|c|c|c|c } 
$\sigma_v $& $F$ & $r_1(K)$ & $|V^{\sigma_v}|$ & $|(W')^{\sigma_v}|$  & CM predicted &  Thm~\ref{thm:main} proven\\ 
&  & & &&  average of $h_3(K)$ &  average of $h_3(K)$\\ 
  \hline
()  & real & 4 & 9 & 3&$(1+1/9)(1+1/3)=40/27\approx 1.48$&$\approx 1.12$\\
 (24)  & real & 2 & 3 & 3&$(1+1/3)(1+1/3)=16/9\approx 1.78$&$\approx 1.34$\\ 
 (13)(24)  & real & 0 & 1 &3&$(1+1)(1+1/3)=8/3\approx 2.67$&$\approx 2.01$\\ 
(12)(34)  & imag. & 0 & 3 & 1&$(1+1/3)(1+1)=8/3\approx 2.67$&$\approx 1.41$\\ 
\end{tabular}
\end{center}
\end{example}

Essentially, in the proof of Theorem~\ref{thm:main}, we see that the averages of the $h_3(K/F)$ and $h_3(F/k)$ factors are independent, and the 
$h_3(K/F)$ has average as predicted by Cohen and Martinet.  
When we order fields in a family up to discriminant $X$, and take a uniform average and then let $X\ra\infty$, we will call this a \emph{discriminant-uniform} average.
A consequence of Cohen and Martinet's conjecture is that the discriminant-uniform average \emph{over $G$ extensions $K$} of $h_3(F_K/k)$  is the same
as the discriminant-uniform average \emph{over $H$ extensions $F$}  of $h_3(F/k)$  when the $G$ and $H$ extensions have corresponding behavior at infinite places (see \cite[Theorem 9.2]{Wang2021}).  

In contrast, what we see in the proof of Theorem~\ref{thm:main} is that 
discriminant-uniform average over $G$ extensions $K$ of $h_3(F_K/k)$ is provably a \emph{weighted} average of $h_3(F/k)$ over $F$, against the measure on $H$-extensions
in which a field $F$ has measure proportional to $\Res_{s=1}\zeta_F(s)/(\zeta_F(2)\Disc(F)^2)$.    There is no particular reason to think that this weighted average of $h_3(F/k)$ over $F$ will give the same result as the discriminant-uniform average over $F$.  Indeed, the weighted average is heavily influenced by the $F$ of small discriminant.
When $H=C_2$,  the discriminant uniform average 
over $C_2$-extensions $F$
of $h_3(F/k)$  is known by Datskovsky--Wright, and is as predicted by Cohen--Martinet.

Since the quadratic fields of small discriminant have very little $3$-torsion in their class groups, one expects the weighted average of $h_3(F/k)$ to be smaller than the discriminant-uniform average over $F$, and indeed that is what we see in the chart above.  

This perspective also explains the counterexample to the Cohen-Lenstra-Martinet heuristics of Bartel and Lenstra~\cite{Bartel2020}.  Let $f$ be the indicator function of whether the $3$-torsion in the class group of a quadratic field is trivial.  For a cyclic quartic field $K$, let $F_K$ be the quadratic subfield.  
Bartel and Lenstra prove that the discriminant-uniform average over cyclic quartic $K$ of $f(F_K)$ is a weighted average of $f(F)$ over quadratic $F$.  On the other hand, the Cohen-Lenstra-Martinet heuristics predict that the discriminant-uniform average over cyclic quartic $K$ of $f(F_K)$ is the discriminant-uniform average of $f(F)$ over quadratic $F$.

\section{Averages for other groups}
\label{sec:other-groups}

Our methods apply to more groups than just $2$-groups.  We focus here on results that may be obtained purely from the results of Section \ref{sec:uniform-cubic}, and whose proofs in particular do not rely on arguments from Sections \ref{sec:arakelov} and \ref{sec:zero-density}.  This permits a proof of the following general result.

\begin{theorem}\label{thm:other-groups}
	Let $H \subseteq S_n$ be transitive and set $G = C_2 \wr H$.  Then $G$-extensions $K$ of a number field $k$ have a unique index two subfield $F_K/k$.  Moreover
	$F_K$ is an $H$-extension of $k$.  For $u \in \mathbb{Z}$, let $E_k^u(G,X) \subseteq E_k(G,X)$ be the subset of those $K$ for which $\rk \mathcal{O}_K^* - \rk \mathcal{O}_{F_K}^* = u$.
	
	1) If $E_k^u(G,\infty)$ is non-empty and $|E_k(H,X)| \ll_{k,H,\epsilon} X^{2/3+\epsilon}$ for every $X\geq 1$, then
		\[
			\lim_{X \to \infty} \frac{1}{E_k^u(G,X)} \sum_{K \in E_k^u(G,X)} h_3(K/F_K)
				= 1 + 3^{-u}.
		\]
		
	2) If $E_k(H,\infty)$ is non-empty and
		\[
			\sum_{F \in E_k(H,X)} h_3(F/k)
				\ll_{k,H,\epsilon} X^{2/3+\epsilon}
		\]
	for every $X \geq 1$, then there is an explicit constant $c_{k,G,3}$ such that
		\[
			\lim_{X \to \infty} \frac{1}{E_k(G,X)} \sum_{K \in E_k(G,X)} h_3(K)
				= c_{k,G,3}.
		\]
\end{theorem}

Before we discuss the proof of Theorem \ref{thm:other-groups}, we note that the hypotheses of the first case are satisfied, for example, by any group $H \neq C_2$ in its regular representation that occurs as a Galois group over $k$ \cite[Proposition 1.3]{EV06}, and by any nilpotent group $H$ without a transposition \cite[Corollary 1.8]{Alb}.  The hypotheses of the second case are satisfied for any $p$-group with $p \neq 5$ odd, as follows from \cite[Corollary 7.3]{KlunersMalle} and the trivial bound on $h_3(F/k)$ for $p \geq 7$, and from \cite{KluWan} for $p=3$.

\begin{proof}
	The claims about the subfield $F_K$ follows as in Lemma \ref{lem:2-grp-with-transposition-is-wreath-product}.  
	Suppose now that $E_k^u(G,\infty)$ is non-empty.  We mimic the proof of Theorem \ref{thm:main}, indicating the necessary modifications.  Let $X \geq 1$.  From Corollary \ref{cor:general-bound}, it follows for any $Y\geq 1$ that
		\[
			\sum_{\substack{ K \in E_k^u(G,X) \\ \Disc(F_K) \geq Y}} h_3(K/F_K)
				\ll_{[k:\mathbb{Q}],G,\epsilon} \sum_{F \in E_k(H,X^{1/2}) \setminus E_k(H,Y)} \frac{h_2(F)^{2/3} X}{\Disc(F)^{1-\epsilon}}
				\ll_{[k:\mathbb{Q}],G,\epsilon} \frac{X}{Y^{\frac{1}{3n[k:\mathbb{Q}]}-\epsilon}},
		\]
	where the second inequality is by partial summation and \cite[Theorem 1.1]{BSTTTZ}.  This provides an analogue of Theorem \ref{thm:tail}.
	We now establish a soft analogue of Theorem \ref{thm:3-torsion-thin-2-ext}, with the remainder of the proof then proceeding as in those of Theorems \ref{thm:main} and \ref{thm:relative-average}.
	Note that if $G^\prime \subseteq G$ surjects onto $H$ and contains an element conjugate to $\sigma = (1,0,\dots,0) \rtimes 1_H$ in $G$, then in fact $G^\prime = G$.  

	For any fixed $H$-extension $F$ with discriminant at most $Y$, let $\mathcal{P}$ be a finite set of primes of $k$ that split completely in $F$. 
If some prime of $F$ above a prime $\mathfrak{p} \in\mathcal{P}$ is inert in a quadratic extension $K/F$, while all the other primes of
$F$ above $\mathfrak{p}$ are split in $K$, then by the note above we have that $K/k$ is a $C_2\wr H$ extension.
If this happens, then we say that $\mathfrak{p}$ is bad for $K/F$.
	 It follows from the methods of Datskovsky and Wright \cite{DW86,DW88} (see also \cite[Theorem 2]{BSW15}) that 
		\[
			\sum_{ \substack{ [K:F] = 2 \\ \text{group signature $\Sigma$} \\ \Disc(K/F) \leq X \\ \forall \mathfrak{p} \in \mathcal{P}, \mathfrak{p} \textrm{ not bad} }} h_3(K/F)
				\sim X \prod_{\mathfrak{p} \in \mathcal{P}} \left(1 - \frac{n \Nm_{k/\mathbb{Q}} \mathfrak{p}^n}{2^n(\Nm_{k/\mathbb{Q}} \mathfrak{p} + 1)^n} \right) \cdot \frac{\mathrm{Res}_{s=1}\zeta_F(s)}{2^{r_1(F)+r_2(F)}\zeta_F(2)}(1+3^{-u(\Sigma)}).
		\]
	The expressions in the product above are uniformly (in $\mathfrak{p}$) bounded away from $1$, and in particular the product is $O(\beta_n^{\#\mathcal{P}})$ for some constant $\beta_n <1$ depending at most on $n$.  Since there are infinitely many primes that split completely in $F$, this product may be made arbitrarily small, and thus
		\[
			\sum_{ \substack{ [K:F] = 2 \\ \Disc(K/F) \leq X \\ \mathrm{Gal}(K/k) \not\simeq G } } h_3(K/F)
				= o_F(X).
		\]
	Proceeding now as in the proof of Theorem \ref{thm:main}, the first claim follows.  The second follows analogously.
\end{proof}

\bibliographystyle{alpha}
\bibliography{RS.bib}

\newcommand{\etalchar}[1]{$^{#1}$}
\begin{thebibliography}{PTBW20}

\bibitem[Alb20]{Alb}
Brandon Alberts.
\newblock The weak form of {M}alle's conjecture and solvable groups.
\newblock {\em Res. Number Theory}, 6(1):Paper No. 10, 23, 2020.

\bibitem[An20]{An2020}
Chen An.
\newblock $\ell$-torsion in class groups of certain families of {$D_4$}-quartic
  fields.
\newblock {\em Journal de Th\'eorie des Nombres de Bordeaux}, 32(1):1--23,
  2020.

\bibitem[ASVW21]{D4VS}
S.~Ali Altu\u{g}, Arul Shankar, Ila Varma, and Kevin~H. Wilson.
\newblock The number of {$D_4$}-fields ordered by conductor.
\newblock {\em J. Eur. Math. Soc. (JEMS)}, 23(8):2733--2785, 2021.

\bibitem[Bha05]{Bha05}
M.~Bhargava.
\newblock The density of discriminants of quartic rings and fields.
\newblock {\em Ann. of Math.}, 162(2):1031--1063, September 2005.

\bibitem[BL20]{Bartel2020}
Alex Bartel and Hendrik~W. Lenstra.
\newblock On class groups of random number fields.
\newblock {\em Proceedings of the London Mathematical Society},
  121(4):927--953, 2020.

\bibitem[Bra47]{Brauer}
Richard Brauer.
\newblock On the zeta-functions of algebraic number fields.
\newblock {\em Amer. J. Math.}, 69:243--250, 1947.

\bibitem[BST13]{BST}
M.~Bhargava, A.~Shankar, and J.~Tsimerman.
\newblock On the {D}avenport-{H}eilbronn theorems and second order terms.
\newblock {\em Invent. Math.}, 193:439--499, 2013.

\bibitem[BST{\etalchar{+}}20]{BSTTTZ}
M.~Bhargava, A.~Shankar, T.~Taniguchi, F.~Thorne, J.~Tsimerman, and Y.~Zhao.
\newblock Bounds on 2-torsion in class groups of number fields and integral
  points on elliptic curves.
\newblock {\em J. Amer. Math. Soc.}, 33(4):1087--1099, 2020.

\bibitem[BSW15]{BSW15}
Manjul Bhargava, Arul Shankar, and Xiaoheng Wang.
\newblock Geometry-of-numbers methods over global fields i: Prehomogeneous
  vector spaces.
\newblock 2015.

\bibitem[BTT21]{Bhargava2021}
Manjul Bhargava, Takashi Taniguchi, and Frank Thorne.
\newblock Improved error estimates for the {{Davenport}}-{{Heilbronn}}
  theorems.
\newblock {\em arXiv:2107.12819 [math]}, July 2021.

\bibitem[BV16]{BhargavaVarma}
Manjul Bhargava and Ila Varma.
\newblock The mean number of 3-torsion elements in the class groups and ideal
  groups of quadratic orders.
\newblock {\em Proc. Lond. Math. Soc. (3)}, 112(2):235--266, 2016.

\bibitem[CL84]{Cohen1984}
Henri Cohen and Hendrik~W. Lenstra, Jr.
\newblock Heuristics on class groups of number fields.
\newblock In {\em Number Theory, {{Noordwijkerhout}} 1983 ({{Noordwijkerhout}},
  1983)}, volume 1068 of {\em Lecture {{Notes}} in {{Math}}.}, pages 33--62.
  {Springer}, {Berlin}, 1984.

\bibitem[CM87]{Cohen1987}
H.~Cohen and J.~Martinet.
\newblock Class groups of number fields: Numerical heuristics.
\newblock {\em Mathematics of Computation}, 48(177):123--137, 1987.

\bibitem[CM90]{Cohen1990}
Henri Cohen and Jacques Martinet.
\newblock {\'E}tude heuristique des groupes de classes des corps de nombres.
\newblock {\em Journal f\"ur die Reine und Angewandte Mathematik}, 404:39--76,
  1990.

\bibitem[CyDO02]{CDyDO02}
H.~Cohen, F.~Diaz y~Diaz, and M.~Olivier.
\newblock Enumerating quartic dihedral extensions of $\mathbb{Q}$.
\newblock {\em Compositio Math.}, 133(1):65--93, 2002.

\bibitem[DH71]{DH71}
H.~Davenport and H.~Heilbronn.
\newblock On the density of discriminants of cubic fields. {II}.
\newblock {\em Proc. Roy. Soc. London. Ser. A}, 322(1551):405--420, 1971.

\bibitem[DW86]{DW86}
Boris Datskovsky and David~J. Wright.
\newblock The adelic zeta function associated to the space of binary cubic
  forms. {II}. {L}ocal theory.
\newblock {\em J. Reine Angew. Math.}, 367:27--75, 1986.

\bibitem[DW88]{DW88}
Boris Datskovsky and David~J. Wright.
\newblock Density of discriminants of cubic extensions.
\newblock {\em J. Reine Angew. Math.}, 386:116--138, 1988.

\bibitem[EPW17]{Ellen16}
Jordan Ellenberg, Lillian~B. Pierce, and Melanie~Matchett Wood.
\newblock On {$\ell$}-torsion in class groups of number fields.
\newblock {\em Algebra Number Theory}, 11(8):1739--1778, 2017.

\bibitem[EV06]{EV06}
J.~S. Ellenberg and A.~Venkatesh.
\newblock The number of extensions of a number field with fixed degree and
  bounded discriminant.
\newblock {\em Ann. of Math.}, pages 723--741, 2006.

\bibitem[EV07]{EV}
Jordan~S. Ellenberg and Akshay Venkatesh.
\newblock Reflection principles and bounds for class group torsion.
\newblock {\em Int. Math. Res. Not. IMRN}, (1):Art. ID rnm002, 18, 2007.

\bibitem[FK06]{Fouvry2006a}
{\'E}tienne Fouvry and J{\"u}rgen Kl{\"u}ners.
\newblock On the 4-rank of class groups of quadratic number fields.
\newblock {\em Inventiones mathematicae}, 167(3):455--513, November 2006.

\bibitem[FS99]{FrSko}
Eduardo Friedman and Nils-Peter Skoruppa.
\newblock Relative regulators of number fields.
\newblock {\em Invent. Math.}, 135(1):115--144, 1999.

\bibitem[FW18]{Frei2018}
Christopher Frei and Martin Widmer.
\newblock Average bounds for the $\ell$-torsion in class groups of cyclic
  extensions.
\newblock {\em Research in Number Theory}, 4(3):34, August 2018.

\bibitem[FW21]{Frei2021}
Christopher Frei and Martin Widmer.
\newblock Averages and higher moments for the $\ell$-torsion in class groups.
\newblock {\em Mathematische Annalen}, 379(3):1205--1229, April 2021.

\bibitem[Ger84]{Gerth1984}
Frank Gerth, III.
\newblock The \$4\$-class ranks of quadratic fields.
\newblock {\em Inventiones Mathematicae}, 77(3):489--515, 1984.

\bibitem[Ger87]{GerthIII1987}
Frank Gerth, III.
\newblock Densities for ranks of certain parts of \$p\$-class groups.
\newblock {\em Proceedings of the American Mathematical Society}, 99(1):1--8,
  1987.

\bibitem[HL21]{HoughLee}
Robert {H}ough and Eun~Hye {L}ee.
\newblock Subconvexity of {S}hintani's zeta functions, 2021.

\bibitem[HP17]{Heath-Brown2017}
D.~R. {Heath-Brown} and L.~B. Pierce.
\newblock Averages and moments associated to class numbers of imaginary
  quadratic fields.
\newblock {\em Compositio Mathematica}, 153(11):2287--2309, November 2017.

\bibitem[HV06]{HV06}
H.~A. Helfgott and A.~Venkatesh.
\newblock Integral points on elliptic curves and 3-torsion in class groups.
\newblock {\em J. Amer. Math. Soc.}, 19(3):527 -- 550, 2006.

\bibitem[IK04]{IwaniecKowalski}
Henryk Iwaniec and Emmanuel Kowalski.
\newblock {\em Analytic number theory}, volume~53 of {\em American Mathematical
  Society Colloquium Publications}.
\newblock American Mathematical Society, Providence, RI, 2004.

\bibitem[Kl{\"u}12]{Klu12}
J.~Kl{\"u}ners.
\newblock The distribution of number fields with wreath products as {G}alois
  groups.
\newblock {\em Int. J. Number Theory}, (8):845--858, 2012.

\bibitem[Kly20]{Klys2016a}
Jack Klys.
\newblock The distribution of {$p$}-torsion in degree {$p$} cyclic fields.
\newblock {\em Algebra Number Theory}, 14(4):815--854, 2020.

\bibitem[KM04]{KlunersMalle}
J\"{u}rgen Kl\"{u}ners and Gunter Malle.
\newblock Counting nilpotent {G}alois extensions.
\newblock {\em J. Reine Angew. Math.}, 572:1--26, 2004.

\bibitem[KP18]{KP18}
P.~Koymans and C.~Pagano.
\newblock On the distribution of {C}l$(k)[\ell^{\infty}]$ for degree $\ell$
  cyclic fields.
\newblock {\em arxiv: 1812.06884v1}, 2018.

\bibitem[KW20]{KluWan}
Jürgen Klüners and Jiuya Wang.
\newblock $\ell$-torsion bounds for the class group of number fields with an
  $\ell$-group as galois group.
\newblock 2020.

\bibitem[LMO79]{LMO}
J.~C. Lagarias, H.~L. Montgomery, and A.~M. Odlyzko.
\newblock A bound for the least prime ideal in the {C}hebotarev density
  theorem.
\newblock {\em Invent. Math.}, 54(3):271--296, 1979.

\bibitem[LOTZ21]{Oliver2021}
Robert~J. {L}emke {O}liver, Jesse Thorner, and Asif Zaman.
\newblock An approximate form of {A}rtin's holomorphy conjecture and
  non-vanishing of {A}rtin {$L$}-functions.
\newblock 2021.

\bibitem[Mal02]{Mal02}
G.~Malle.
\newblock On the distribution of {G}alois groups.
\newblock {\em J. Number Theory}, 92(2):315--329, 2002.

\bibitem[Mal04]{Mal04}
G.~Malle.
\newblock On the distribution of {G}alois groups, {I}{I}.
\newblock {\em Experiment. Math.}, 13(2):129--135, 2004.

\bibitem[{Pas}17]{Pasten}
Hector {Pasten}.
\newblock {Shimura curves and the abc conjecture}. {Appendix by R.J. Lemke
  Oliver and J. Thorner}.
\newblock {\em arXiv e-prints}, page arXiv:1705.09251, May 2017.

\bibitem[Pie05]{Pie05}
L.~B. Pierce.
\newblock The 3-part of class numbers of quadratic fields.
\newblock {\em J. London Math. Soc.}, 71:579--598, 2005.

\bibitem[Pie06]{Pie06}
L.~B. Pierce.
\newblock A bound for the 3-part of class numbers of quadratic fields by means
  of the square sieve.
\newblock {\em Forum Math.}, (18):677--698, 2006.

\bibitem[PTBW20]{PTBW}
Lillian~B. Pierce, Caroline~L. Turnage-Butterbaugh, and Melanie~Matchett Wood.
\newblock An effective {C}hebotarev density theorem for families of number
  fields, with an application to {$\ell$}-torsion in class groups.
\newblock {\em Invent. Math.}, 219(2):701--778, 2020.

\bibitem[Sch08]{Schoof}
R.~Schoof.
\newblock Computing arakelov class groups.
\newblock In {\em Algorithmic number theory: lattices, number fields, curves
  and cryptography}, volume~44 of {\em Math. Sci. Res. Inst. Publ.}, pages
  447--495. Cambridge Univ. Press, Cambridge, 2008.

\bibitem[Shi72]{Shintani}
Takuro Shintani.
\newblock On {D}irichlet series whose coefficients are class numbers of
  integral binary cubic forms.
\newblock {\em J. Math. Soc. Japan}, 24:132--188, 1972.

\bibitem[Smi17]{Smith2017}
Alexander Smith.
\newblock $2^\infty$-{{Selmer}} groups, $2^\infty$-class groups, and
  {{Goldfeld}}'s conjecture.
\newblock {\em arXiv:1702.02325 [math]}, February 2017.

\bibitem[Sou00]{Sound-Divisibility}
K.~Soundararajan.
\newblock Divisibility of class numbers of imaginary quadratic fields.
\newblock {\em J. London Math. Soc. (2)}, 61(3):681--690, 2000.

\bibitem[Sta74]{Stark}
H.~M. Stark.
\newblock Some effective cases of the {B}rauer-{S}iegel theorem.
\newblock {\em Invent. Math.}, 23:135--152, 1974.

\bibitem[{Tan}06]{Taniguchi}
Takashi {Taniguchi}.
\newblock {Distributions of discriminants of cubic algebras}.
\newblock {\em arXiv Mathematics e-prints}, page math/0606109, Jun 2006.

\bibitem[Tho11]{Th12}
F.~Thorne.
\newblock Four perspectives on secondary terms in the {D}avenport-{H}eilbronn
  theorems.
\newblock {\em Integers Volume 12 B, Proceedings of the Integers Conference
  2011}, 2011.

\bibitem[TT13]{TT13}
T.~Taniguchi and F.~Thorne.
\newblock Secondary terms in counting functions for cubic fields.
\newblock {\em Duke Math. J.}, 162(13):2451--2508, 2013.

\bibitem[TZ17]{ThornerZaman}
Jesse Thorner and Asif Zaman.
\newblock An explicit bound for the least prime ideal in the {C}hebotarev
  density theorem.
\newblock {\em Algebra Number Theory}, 11(5):1135--1197, 2017.

\bibitem[TZ19a]{ZTChebo}
Jesse Thorner and Asif Zaman.
\newblock A unified and improved {C}hebotarev density theorem.
\newblock {\em Algebra Number Theory}, 13(5):1039--1068, 2019.

\bibitem[TZ19b]{Thorner2019}
Jesse Thorner and Asif Zaman.
\newblock A zero density estimate for {{Dedekind}} zeta functions.
\newblock {\em arXiv:1909.01338 [math]}, September 2019.

\bibitem[TZ21]{ThornerZaman-LargeSieve}
Jesse Thorner and Asif Zaman.
\newblock An unconditional {${\rm GL}_n$} large sieve.
\newblock {\em Adv. Math.}, 378:107529, 24, 2021.

\bibitem[\v{S}54]{Shafa}
I.~R. \v{S}afarevi\v{c}.
\newblock Construction of fields of algebraic numbers with given solvable
  {G}alois group.
\newblock {\em Izv. Akad. Nauk SSSR. Ser. Mat.}, 18:525--578, 1954.

\bibitem[Wan20]{Wang2020}
Jiuya Wang.
\newblock Pointwise {{Bound}} for $\ell$-torsion in {{Class Groups II}}:
  {{Nilpotent Extensions}}.
\newblock {\em arXiv:2006.10295 [math]}, June 2020.

\bibitem[Wan21]{Wang2021a}
Jiuya Wang.
\newblock Pointwise bound for {$\ell$}-torsion in class groups: {{Elementary}}
  abelian extensions.
\newblock {\em Journal f\"ur die reine und angewandte Mathematik (Crelles
  Journal)}, 2021(773):129--151, April 2021.

\bibitem[Wid18]{Widmer2018}
Martin Widmer.
\newblock Bounds for the {$\ell$}-torsion in class groups.
\newblock {\em Bulletin of the London Mathematical Society}, 50(1):124--131,
  2018.

\bibitem[Woo18]{Wood2018}
Melanie~Matchett Wood.
\newblock Cohen-{{Lenstra}} heuristics and local conditions.
\newblock {\em Research in Number Theory}, 4(4):41, September 2018.

\bibitem[Wri82]{WrightThesis}
David~James Wright.
\newblock {\em D{IRICHLET} {SERIES} {ASSOCIATED} {WITH} {THE} {SPACE} {OF}
  {BINARY} {CUBIC} {FORMS} {WITH} {COEFFICIENTS} {IN} {A} {NUMBER} {FIELD}}.
\newblock ProQuest LLC, Ann Arbor, MI, 1982.
\newblock Thesis (Ph.D.)--Harvard University.

\bibitem[Wri85]{WrightAdelic}
David~J. Wright.
\newblock The adelic zeta function associated to the space of binary cubic
  forms. {I}. {G}lobal theory.
\newblock {\em Math. Ann.}, 270(4):503--534, 1985.

\bibitem[Wri89]{Wright}
David~J. Wright.
\newblock Distribution of discriminants of abelian extensions.
\newblock {\em Proc. London Math. Soc. (3)}, 58(1):17--50, 1989.

\bibitem[WW21]{Wang2021}
Weitong Wang and Melanie~Matchett Wood.
\newblock Moments and interpretations of the
  {{Cohen}}\textendash{{Lenstra}}\textendash{{Martinet}} heuristics.
\newblock {\em Commentarii Mathematici Helvetici}, 96(2):339--387, June 2021.

\bibitem[Zam17]{ZamThesis}
A.~Zaman.
\newblock Analytic estimates for the {C}hebotarev density theorem and their
  applications.
\newblock {\em Ph.D. thesis, University of Toronto}, 2017.

\end{thebibliography}
\end{document}